\documentclass[11pt]{article}
\usepackage{amsmath,amsfonts,amssymb,amsthm,authblk}
\usepackage{graphicx}
\usepackage{xcolor}
\usepackage{dsfont}
\usepackage{hyperref}
\usepackage{mathrsfs}
\usepackage{ulem}
\usepackage{cancel}

\newcommand{\M}{\mathcal{M}}
\renewcommand{\L}{\mathbb{L}}
\newcommand{\E}{\ensuremath{\mathbb{E}}}
\renewcommand{\P}{\ensuremath{\mathbb{P}}}
\newcommand{\R}{\ensuremath{\mathbb{R}}}
\newcommand{\pen}{\mathrm{pen}}
\newcommand{\p}{\mathrm{p}}
\newcommand{\Ne}{\mathbb{N}}
\DeclareMathOperator*{\argmin}{arg\,min\,}
\newcommand{\1}{\mathds{1}}

\newcommand{\e}{\varepsilon}
\newcommand{\la}{\lambda}
\renewcommand{\hat}{\widehat}
\renewcommand{\tilde}{\widetilde}
\newcommand{\J}{\mathcal{J}}
\newcommand{\Lp}{{\mathbb L}_p}

\newtheorem{theorem}{Theorem}[section]
\newtheorem{corollary}[theorem]{Corollary}
\newtheorem{lemma}[theorem]{Lemma}
\newtheorem{proposition}[theorem]{Proposition}
\newtheorem{remark}[theorem]{Remark}
\begin{document}

\title{Is model selection possible for the $\ell_p$-loss?\\ PCO estimation  for  regression models}
\author[1]{Claire Lacour}
\author[2]{Pascal Massart}
\author[3,2]{Vincent Rivoirard}
\affil[1]{Univ Gustave Eiffel, Univ Paris Est Creteil, CNRS, LAMA UMR8050,
       F-77447, Marne-la-Vall\'ee, France}
\affil[2]{Universit\'e Paris-Saclay, CNRS, Inria, Laboratoire de math\'ematiques d'Orsay, 91405, Orsay, France.}
\affil[3]{CEREMADE, CNRS, Universit\'e Paris-Dauphine, Universit\'e PSL, 75016 PARIS, FRANCE}
%\author{Claire Lacour
%\footnote{LAMA, Univ Gustave Eiffel, Univ Paris Est Creteil, CNRS, 
%       F-77447, Marne-la-Vall\'ee, France}
%, Pascal Massart
%\footnote{Universit\'e Paris-Saclay, CNRS, Inria, Laboratoire de math\'ematiques d'Orsay, 91405, Orsay, France.
%}
%, Vincent Rivoirard
%\footnote{CEREMADE, CNRS, UMR 7534, Universit\'e Paris Dauphine, PSL Research University 75016 Paris, France}
%\footnote{Universit\'e Paris-Saclay, CNRS, Inria, Laboratoire de math\'ematiques d'Orsay, 91405, Orsay, France.
%}}

\maketitle

\begin{abstract}
This paper addresses the problem of  model selection in the sequence model $Y=\theta+\varepsilon\xi$, when $\xi$ is sub-Gaussian, for non-euclidian loss-functions. In this model,  the Penalized Comparison to Overfitting procedure is studied for the weighted $\ell_p$-loss, $p\geq 1.$ Several oracle inequalities are derived from concentration inequalities for sub-Weibull variables.  Using judicious collections of models and penalty terms,  minimax rates of convergence are stated for Besov bodies $\mathcal{B}_{r,\infty}^s$.  These results are applied to the functional model of nonparametric regression.
\end{abstract}

\medskip

- Secondary: {62G08, 60E15\\

\medskip


%%%%%%%%%%%%%%%%%%
%%%%%%%%%%%%%%%%%%
\section{Introduction}
The problem of selecting a model from among several candidates is essential in statistics and machine learning, as well as in many application fields. In the most general sense, the aim of model selection is to construct data-driven criteria for selecting a model $m$ from a given collection $\mathcal{M}$.
In other words, if one observes some random variable $\xi^{(n)}$ (which can be typically a random vector of size $n$) with unknown distribution depending on some quantity $f$ (the target) belonging to a set $\mathcal{S}$, a flexible approach to estimate $f$ is to consider some
collection of preliminary estimators $\big(\hat{f}_{m}\big)_{m\in\mathcal{M}}$ and then try to design some genuine data-driven procedure $\hat
{m}\in\mathcal{M}$ to produce a new estimator $\hat{f}_{\hat{m}}$. Considering some loss function $\ell$, we measure the quality of each estimator $\hat{f}_{m}$, through the quantity $\ell\big(f,\hat {f}_{m}\big)$ and mathematical results on estimator selection are
formulated in terms of upper bounds on $\ell\big(  f,\hat{f}_{\hat{m}}\big)  $ that allow to measure how far
this quantity is from what is usually called the \textit{oracle} risk
$\inf_{m\in\mathcal{M}}\mathbb{E}\,\ell\big[\big(  f,\hat{f}_{m}\big)\big]$. These
comparison inequalities are called oracle inequalities. The quadratic loss $\ell\big(  f,\hat{f}_{m}\big)  =\big\Vert
f-\hat{f}_{m}\big\Vert ^{2}$ is a standard choice, but of course other losses are considered in the literature. 

Actually, in the classical model selection framework, developed and popularized by Birg\'e and Massart \cite{BM01}, \cite{BM07}, \cite{Massart2007}, the
list of estimators and the loss function are intimately related in the sense
that they derive from the same \textit{contrast function} (also called
\textit{empirical risk} in the machine learning literature). More precisely,
a contrast function $L_{n}$ is a function on the set $\mathcal{S}$ depending
on the observation $\xi^{(n)}$ in such a way that
\[
g\in{\mathcal S}\longmapsto\mathbb{E}\left[  L_{n}\left(  g\right)  \right]
\]
achieves a minimum at point $f$. Given some collection of subsets $(S_{m}) _{m\in\mathcal{M}}$ of $\mathcal{S}$, called models in the sequel, for every $m\in\mathcal{M}$, some estimator
$\hat{f}_{m}$ of $f$ is obtained by minimizing $L_{n}$ over $S_{m}$ ($\hat
{f}_{m}$ is called \textit{minimum contrast estimator} or \textit{empirical
risk minimizer}). In the case where $\xi^{\left(  n\right)
}=\left(  \xi_{1},\ldots,\xi_{n}\right)  $, an empirical criterion $L_{n}$ can be
defined as an empirical mean
\[
L_{n}\left(  g\right)  =P_{n}\left[  L\left(  g,.\right)  \right]  :=\frac
{1}{n}\sum_{i=1}^{n}L\left(  g,\xi_{i}\right)  \text{,}%
\]
which justifies the terminology of empirical risk. Empirical risk minimization includes maximum likelihood and least squares estimation. The penalized empirical risk selection procedure consists in considering some
proper penalty function $\operatorname*{pen}$: $\mathcal{M}\rightarrow
\mathbb{R}_{+}$ and taking $\hat{m}$ minimizing
\begin{equation}\label{def:sele}
L_{n}\left(  \hat{f}_{m}\right)  +\pen(m)
\end{equation}
over $\mathcal{M}$. We can then define the selected model $S_{\hat{m}}$ and
the corresponding selected estimator $\hat{f}_{\hat{m}}$. Penalized criteria have been proposed in the early seventies by Akaike or
Schwarz (see \cite{AIC} and \cite{BIC}) for penalized maximum
log-likelihood in the density estimation framework and Mallows for penalized
least squares regression (see %\cite{dan-wood} and 
\cite{Mallows}). In both
cases the penalty functions are proportional to the number of parameters $D_{m}$ of the corresponding model $S_{m}$. In such settings, the performance of model selection estimators can be studied for the  \textquotedblleft
natural\textquotedblright\ (non negative) loss function  $\ell$ attached to $L_{n}$ through the simple definition
\begin{equation}\label{e1l}
\ell\left(  f,g\right)  =\mathbb{E}[L_{n}(g)]
-\mathbb{E}[  L_{n}(  f)],\quad g\in\mathcal{S}.
\end{equation}
This approach has been successfully applied in many settings: density estimation or Poisson intensity estimation where $\ell$ is the Kullback-Leibler divergence or the $\mathbb L_2$-loss, nonparametric regression  for the $\mathbb L_2$-loss,
binary classification {for the 0-1 loss} 
or the Gaussian white noise model still for the $\mathbb L_2$-loss. See \cite{Massart2007} and references therein. Although it is not derived from an empirical contrast, we mention the model selection work associated with the Hellinger distance \cite{MR816706}, \cite{MR2834722}. This explains why, although successfully applied in many settings, model selection procedures have only been defined and analyzed for very specific loss functions. In particular, non-Euclidian loss functions, as $\L_p$-losses, have rarely been considered for model selection procedures. It is not the case for other classical nonparametric estimation procedures.

Although the use of $\L_p$-loss in nonparametric statistics goes back at least to \cite{bretagnollehuber,IbragimovHasminski80,stone82}, an important advance was made by \cite{nemirovski} who established optimal rates of convergence in the problem of multivariate nonparametric regression for $\L_p$-loss and functions belonging to $\L_q$-Sobolev classes with possibly $p\neq q$.
Some years later, fundamental works have been made by Oleg Lepski, who proposed the so-called famous Lepski-method for selecting a bandwidth of a kernel estimator.
In \cite{Lepski91}, he studied adaptive estimation under $\L_p$-loss, $1\leq  p \leq \infty$ over the collection of H\"{o}lder classes.
Then \cite{LepskiMammenSpokoiny} introduced a local bandwidth selection scheme to give kernel estimates  which achieve optimal rates of convergence over Besov classes in the  Gaussian white noise model, the anisotropic multivariate case being  examined by \cite{KLP}.
Then \cite{GL08} developed a powerful methodology for selecting a bandwidth of a kernel estimate, which works in a multitude of contexts and allows to establish oracle inequalities and to derive optimal rates of convergence: see \cite{GL2011}, \cite{GL13}, \cite{GL14}. Nevertheless, the implementation of this method needs two steps of minimization, each requiring a thorough calibration.
%
%and references therein.
Note that the estimation under the $\L_1$-loss for bandwidth-selected  kernel estimators has been investigated in 
\cite{devroyelugosi}, 
%% on pourrait aussi citer Devroye Lugosi 2001 MR1843146
and  the $\L_p$-aggregation of estimators has been studied by \cite{goldenshluger09}.

In a very different spirit, another very popular method for adaptive estimation of functions is wavelet thresholding. In a series of papers, Donoho, Johnstone, Kerkyacharian and Picard have shown the power of wavelet thresholding for $\L_p$-estimation over the scale of Besov classes: see \cite{DJKP2},  \cite{DJKP},\cite{DJKP3}, \cite{DJ98}. Refinements have been proposed in  \cite{Juditsky97}, \cite{KPT}, \cite{HKP}, \cite{johnstonesilverman}, to name but a few. %\cite{johnstonesilverman} for hard threshold wavelet estimator but with empirical bayes selection of thresholds.
It must be noted that these thresholding methods generally suffer from logarithmic losses in the rates of convergence. We refer the reader to the book by \cite{hkpt} which details the construction of wavelet bases, their use in statistical estimation, $\L_p$-minimax results as well as computational aspects.  Johnstone's recent book \cite{johnstone2019} addresses nonparametric function estimation by carefully studying the infinite Gaussian sequence model, with many results and thoughts about wavelet thresholding and adaptive minimaxity over ellipsoids. 

We finally mention the forthcoming manuscript \cite{ace24} which considers Bayesian nonparametric concentration rates of procedures based on Heavy-tailed and Horseshoe priors for $\ell_p$-norms in the Gaussian white noise model.

%Model selection actually proceeds in two steps: first choose some family of models Sm with m ? M together with estimators söm with values in Sm . In general, söm derives from a classical estimation procedure, here the maximum likelihood, under the assumption that the model Sm is true (s ? Sm). Then use the data to select a value mö of m and take sömö as the final estimator. A ÒgoodÓ model selection procedure is one for which the risk of the resulting estimator is as close as possible to the minimal risk of the estimators söm , m ? M.

The main goal of this paper is then to answer the following natural question: Is it possible to design a model selection procedure in the same spirit as in \eqref{def:sele}, and in particular resulting from one minimization step, so that it achieves optimal theoretical performances for non-Euclidian losses? In particular, we have in mind that classical model selection procedures are able to achieve optimal oracle properties and sharp minimax rates of convergence on classical functional spaces for Euclidian losses. We tackle this issue by considering the classical infinite sequence model and study a specific model selection procedure, called \textit{PCO}, for weighted $\ell_p$-loss functions, with $1\leq p<\infty$.
As explained by \cite{johnstone2019} (see his Preface and Section 1.5), the Gaussian sequence model "captures many of the conceptual issues associated with non-parametric estimation".
%%%%%%%%%%%%%%%%%%
\subsection{The PCO estimation procedure for the sub-Gaussian sequence model}\label{sec:model}
For $\Lambda$ a countable set, we consider the following classical sequence model: 
\begin{equation}\label{def:model}
Y_{\la}=\theta_{\la} + \e \xi_{\la}, \qquad \la \in \Lambda.
\end{equation}
In this model, the noise level $\varepsilon$ is assumed to be smaller than a constant, say 1, and $\varepsilon\to 0$ defines the asymptotic setting of our study. The $\xi_{\la}$'s are i.i.d. centered variables, satisfying the sub-Gaussian property, i.e.
\begin{equation}
\label{subgauss}
\P(|\xi_\lambda|\geq t)\leq 2e^{-t^2/2},\quad t\geq 0.
\end{equation}
The previous definition refers,  for instance, to Proposition~2.5.2 of \cite{vershynin2018} with scale parameter $K_1=\sqrt{2}$; 
note that fixing $K_1=\sqrt{2}$ is not a restriction since, in our setting, we can replace the noise level $\e$ with $K_1\e$ without loss of generality.

We aim at estimating the sequence $\theta=(\theta_\lambda)_{\lambda\in\Lambda}$ by using a finite number of observations, say $(Y_{\la})_{\lambda\in\Lambda^{(N)}}$ where $\Lambda^{(N)}$ denotes the first $N$ elements of $\Lambda $. The integer $N$ may increase as $\e$ decreases, so we actually face with a nonparametric problem.  To connect our setting with nonparametric regression,  we may have in mind that $N\propto\e^{-2}$ (see \cite{DJKP2,Juditsky97,KP2000} and Sections~\ref{sec:minimax} and \ref{sec:regression}) but, unless specified, our results hold for any $N$.

For each $m$ a subset of $\Lambda^{(N)}$, called model in the sequel, we set
$$\hat \theta^{(m)}=\big(Y_{\la}\1_{\{\la \in m\}}\big)_{\la \in \Lambda}.$$
{In particular, $\hat \theta^{(m)}_\lambda=0$ for $\lambda\notin\Lambda^{(N)}$. For $1\leq p<\infty$, and $w=(w_\lambda)_{\lambda\in\Lambda}$ a sequence of non negative weights, we denote $\ell_p(w)$ the weighted $\ell_p$-norm on $\R^\Lambda$: 
\begin{equation}\label{defi-normlp}
\|\vartheta\|_{\ell_p(w)}^p=\sum_{\la \in \Lambda} w_\lambda|\vartheta_{\la}|^p,\quad \vartheta\in\R^\Lambda.
\end{equation}
In Model \eqref{def:model}, we consider the risk associated with this weighted $\ell_p$-norm. We assume in the sequel that $\|\theta\|_{\ell_p(w)}<\infty$.

Given a collection of models $\mathcal{M}\subset\mathcal{P}(\Lambda^{(N)})$, we wish to select $\hat m\in \mathcal{M}$ in the best possible way. % and investigate its performances for the $\ell_p(w)$-norm.
For this purpose, as explained previously, we rely on the {\it PCO criterion} introduced by \cite{PCO1, PCO2}. The heuristic of this approach is to build the goodness of fit criterion by using the estimator which has the smallest bias among the collection $(\hat \theta^{(m)})_{m\in\mathcal{M}}$. In our setting, it means that we have to consider $\hat\theta^{(\Lambda^{(N)})}$ (see \eqref{eq:BV}). Then, adding as usual in the nonparametric setting a penalty term, we set 
\begin{align*}
\hat{m}&=\argmin_{m\in \mathcal{M}}\left\{\|\hat\theta^{(m)}-\hat\theta^{(\Lambda^{(N)})}\|_{\ell_p(w)}^p+\pen(m)\right\}
%&=\argmin_{m\in \mathcal{M}}\left\{-\sum_{\la \in m}w_\lambda|Y_{\la}|^p+\pen(m)\right\}
\end{align*}
and estimate $\theta$ by $$\tilde\theta=\hat \theta^{(\hat m)}.$$
The idea of this methodology is to use $\|\hat\theta^{(m)}-\hat\theta^{(\Lambda^{(N)})}\|_{\ell_p(w)}^p$ as a preliminary estimator of the bias of $\hat\theta^{(m)}$. The role of $\pen(m)$ is then twofold: adjusting this preliminary step and taking into account the variance of $\hat \theta^{(m)}$. The estimator $\tilde\theta$ will be called the \textit{Penalized Comparison to Overfitting} (abbreviated as \textit{PCO}) in the sequel. This terminology is justified by the overfitting properties of $\hat\theta^{(\Lambda^{(N)})}$. Of course, setting for $m\in \mathcal{M}$,
$$\mathrm{Crit}(m)=-\sum_{\la \in m}w_\lambda|Y_{\la}|^p+\pen(m),$$
we obtain
$$\hat{m}=\argmin_{m\in \mathcal{M}}\mathrm{Crit}(m).$$
The heuristic of this approach is then different from the classical approach based on the contrast function. However, observe that if we take $p=2$, the criterion function $\mathrm{Crit}$ corresponds to the one used in regression for various famous criteria such as Mallows's $C_p$ \cite{Mallows}, AIC \cite{AIC} or BIC \cite{BIC} for instance. Two remarks are in order:
Unlike Lepski type procedures, the derivation of the PCO estimate $\tilde\theta$ involves only one minimization step, so its computational cost is much lower.
Furthermore, if we take $\pen(m)$ of the form
\begin{equation}\label{pen-seuillage}
\pen(m)=\sum_{\lambda\in m}w_\lambda t^p,
\end{equation}
for some $t>0$, we have: 
$$\hat{m}=\big\{\lambda\in\Lambda^{(N)}:\quad |Y_\lambda|>t\big\}$$
and the PCO estimate corresponds to the thresholding estimate with threshold $t$:
$$ \tilde\theta_\lambda=\left\{
\begin{array}{cc}
Y_\lambda\times 1_{\{|Y_\lambda|>t\}}&\lambda\in\Lambda^{(N)},\\
0&\lambda\notin\Lambda^{(N)}.
\end{array}
\right.
$$
Results of Sections~\ref{sec:oracle}, \ref{sec:minimax} and \ref{sec:regression} show that we have to refine the definition of $\pen(m)$ given in \eqref{pen-seuillage} to obtain optimal results. This is described in the next subsection.\\
%\vinc{Je pense que ce n'est pas la peine de discuter pourquoi on ne met pas un $c_\lambda$ devant $Y_\lambda$ }
%%%%%%%%%%%%%%%%%%
\subsection{Contributions}
As observed before, Lespki-type procedures achieve optimal properties in many settings but their computational cost is prohibitive. The main contributions of this paper consist in showing that under a convenient choice of $\pen(m)$, the PCO estimate, which is based on the simple $\ell_p(w)$-criterion $\mathrm{Crit}$ and whose computational cost is reasonable, is able to achieve optimal results in oracle and minimax settings for any value of $p\in [1,+\infty)$. 

We first  prove in Theorem~\ref{theo:oracle1} that $\tilde\theta$ satisfies a very general oracle inequality whatever the expression of the penalty term $\pen(m)$. For this purpose, we analyse the behavior of bias and variance terms of any estimate $\hat\theta^{(m)}$. This first result shows how to choose $\pen(m)$, so that we obtain a more specific oracle inequality established in Theorem~\ref{theo:oracleesp}. As usual, these results depend on sharp concentration inequalities and our results rely on Theorem~\ref {theo:concentration} involving sharp concentration of terms of the form $\sum_{\lambda \in \mathcal{I}} |\xi_\lambda|^p$ around their mean. The sharp tail bound involves the sum of two terms: a quadratic one, proportional to $\sqrt{x}$, which is classical, and a second one proportional to $x^{p/2}$. This last term is linear for $p=2$ but it raises many technical difficulties otherwise in particular for $p>2$. This result then reveals an elbow phenomenon depending on whether $p$ is larger than 2 or not. We take into account this elbow to propose a more refined function $\pen(m)$ for the case $p>2$ in Theorem~\ref{theo:oracleesp2}. This result allows to deal with very large collections of models $\mathcal{M}$ when $p>2$, which is crucial for the minimax setting.

Minimax rates of convergence for the estimate $\tilde\theta$ are first derived when weights are constant and sequences $\theta$ have tails with a polynomial decreasing. We then consider the classical class of Besov bodies $\mathcal{B}_{r,\infty}^{s}(R)$ and study rates for any $r\geq 1$ and any $s> 1/r$ for the $\ell_p(w)$-risk. We prove in Theorem~\ref{rateslp} that for $p\leq 2$, $\tilde\theta$ is optimal under a suitable choice of the penalty function. For $p>2$, $\tilde\theta$ is also optimal if $r\geq p$ and if $r\leq p/(2s+1)$. For the case $p>2$ and $p/(2s+1)<r<p$, the upper bound differs from the lower bound by a logarithmic term. To deal with the case $r<p$, we need to consider very large collection of models, in particular if $r\leq p/(2s+1)$. The last contribution of our paper consists in extending these last results to the functional framework. In Section~\ref{sec:regression}, we consider  the nonparametric regression model
$$
X_i = f\bigg(\frac{i}{n}\bigg) +\sigma \eta_i,\qquad 1\leq i \leq n,
$$
(see Model~\eqref{model-NP}) and propose a PCO estimate of the function $f$ based on wavelet representations. The generalization of results of Theorem~\ref{rateslp} allows to obtain Theorem~\ref{ratesLp} that provides rates of our procedure on functional Besov spaces  for the standard functional $\L_p$-loss and to discuss optimality. The conclusions are similar to those of the sequential case. We also present the main steps of the methodology to derive functional minimax rates, which represents an interest per se.
%%%%%%%%%%%%%%%%%%
\subsection{Plan of the paper and notation}
The paper is organized as follows. Section~\ref{sec:oracle} is devoted to oracle results and the statement of concentration inequalities used in this paper. Section~\ref{secg:minimax} presents the minimax rates of convergence achieved by the PCO estimator. 
Section~\ref{sec:regression} is devoted to the nonparametric regression model. Finally, Section~\ref{sec:proofs} presents the proofs of the results.

For any set $A$ we denote by $|A|$ the cardinal of $A$, and $\mathcal{P}(A)$ the set of its subsets.  The notation $\Ne$ is the set of non-negative integers: $\Ne=\big\{0,1,2,\cdots\}.$
We denote by $u_\e\lesssim v_\e$ when there
exists $0 < A < \infty$ such that $ u_\e \leq A v_\e$ for all $\e>0$. When  $u_\e\lesssim v_\e$ and $v_\e\lesssim u_\e$ we write $u_\e\approx v_\e$. Remember that, for $1\leq p\leq\infty$, $\ell_p(w)$ denotes the weighted $\ell_p$-norm, defined in \eqref{defi-normlp}.  When weights $w_\lambda$ are all equal to 1, we use the classical notation $\ell_p$ instead of $\ell_p(w)$. In the sequel, for short, we set $\|\cdot\|_p=\|\cdot\|_{\ell_p(w)}$. The functional norm on the Banach space $\Lp(\R)$ will be denoted by $\|\cdot\|_{\Lp}$. 
%%%%%%%%%%%%%%%%%%
%%%%%%%%%%%%%%%%%%
\section{Oracle approach and concentration inequalities}\label{sec:oracle}
Given any $p\geq 1$, the goal of this section is to provide some optimality results in the oracle setting for the $\ell_p(w)$-risk in Model~\eqref{def:model}. In particular, in the sequel, except in Theorem~\ref{theo:concentration} (for the first point), we assume that the $\xi_\lambda$'s are i.i.d. centered sub-Gaussian variables. Along this section, we denote
$$\sigma_q:=\big(\E[|\xi_{\lambda}|^q]\big)^{1/q},\quad 1\leq q<\infty.$$
We first derive in subsequent Theorem~\ref{theo:oracle1} a very general result which holds for any penalty function. Actually, by highlighting the key role of sums of the form $\sum_{\lambda \in \mathcal{I}} |\xi_\lambda|^p$, Theorem~\ref{theo:oracle1} allows to determine the ideal choice for the penalty $\pen(m)$. Concentrations of such sums around their mean are precisely studied in Theorem~\ref{theo:concentration} and Corollary~\ref{coro:concentration}. This allows to refine Theorem~\ref{theo:oracle1}, and sharp oracle inequalities are established in Theorems \ref{theo:oracleesp} and \ref{theo:oracleesp2}.

Before stating these results, we first observe that for any model $m$, we can easily compute the distance of the estimator $\hat \theta^{(m)}$ with respect to $\theta$ for the norm $\ell_p(w)$. Indeed, we have for $m\subset\Lambda^{(N)}$:
\begin{align*}
\|\hat \theta^{(m)} - \theta \|_p^p &=\sum_{\la \in m} w_\lambda|Y_{\la}-\theta_{\la}|^p+\sum_{\la \notin m} w_\lambda|0-\theta_{\la}|^p\\
%&=\e^p\sum_{\la \in m} w_\lambda|\xi_{\la}|^p+\sum_{\la \notin m} w_\lambda|\theta_{\la}|^p\\
&=V_p(m)+B_p(m),
\end{align*}
with
\begin{equation}\label{eq:BV}
B_p(m):=\sum_{\la \in \Lambda\setminus m}w_\lambda | \theta_{\la}|^p\quad\mbox{and}\quad V_p(m):=\e^p\sum_{\la \in m}w_\lambda  | \xi_{\la}|^p.
\end{equation}
In the last decomposition, $B_p(m)$ (resp. $V_p(m)$) can be viewed as an $\ell_p(w)$-bias term (resp. an $\ell_p(w)$-variance term). We now study $\tilde\theta=\hat \theta^{(\hat m)}$ in the oracle setting.
%%%%%%%%%%%%%%%%%%
\subsection{A general oracle inequality}
%\textcolor{red}{Je te propose deux resultats : un avec $\alpha$ quelconque pour montrer que l'on peur erre assez general. Un autre avec $\alpha=1$ pour simplifier ; et a prendre pour la suite ?}
Recall that $\tilde\theta=\hat \theta^{(\hat m)}$ with
$$\hat \theta^{(m)}=(Y_{\la}\1_{\{\la \in m\}})_{\la \in \Lambda},\quad\hat{m}=
\argmin_{m\in \mathcal{M}}\left\{-\sum_{\la \in m}w_\lambda|Y_{\la}|^p+\pen(m)\right\}.$$
We obtain the following general oracle inequality for  any $p\in[1,+\infty)$.
\begin{theorem}\label{theo:oracle1} 
If $p>1$, for any arbitrary $m\in \mathcal{M}$, we have for any $\alpha\in (0,2)$:
$$\|\tilde \theta-\theta\|_p^p\leq M_{p,\alpha} \|\hat \theta^{(m)}-\theta\|_p^p +\frac{2}{\alpha}
\Big[(1+\alpha)V_p(\hat m)-\pen(\hat m)\Big]
-\frac{2}{\alpha} \Big[(1+\alpha)V_p(m)-\pen(m)\Big],$$
where $M_{p,\alpha}$ depends only on $p$ and $\alpha$.
% and $M_{p,\alpha}\approx\alpha^{-p}$ when $\alpha\to 0$. \textcolor{red}{laisser ce point ?}\\ 
In particular, with $\alpha=1$, we obtain:
\begin{equation}\label{alpha=1}
\|\tilde \theta-\theta\|_p^p\leq M_{p}\|\hat \theta^{(m)}-\theta\|_p^p +2
\Big[2V_p(\hat m)-\pen(\hat m)\Big]
-2 \Big[2V_p(m)-\pen(m)\Big],
\end{equation}
where $M_p=M_{p,1}$ is given in Equation \eqref{valeurdeM}  (see the proof).\\ If $p=1$, the previous inequalities are true by replacing $\frac{2}{\alpha}$ by $\frac{1}{\alpha\wedge 1}$ in the right hand side of the first inequality.
\end{theorem}
The proof of Theorem~\ref{theo:oracle1} is provided in Section~\ref{prooforacle}. 
\begin{remark}\label{alpha}
The value of $M_{p,\alpha}$ is rather intricate (see the proof of Theorem~\ref{theo:oracle1}) and the best choice for $\alpha$ depends on $p$. Nevertheless, we numerically observe that $M_{p,1}$ is not far from the minimum of the function $\alpha\longmapsto M_{p,\alpha}$, and is even optimal for $p=1$ and $p=2$. This is why we focus on the case $\alpha=1$ and state Inequality~\eqref{alpha=1}.
\end{remark}

In view of the first result of Theorem~\ref{theo:oracle1}, optimality of the estimate $\tilde \theta$ will be achieved among all estimates $(\hat \theta^{(m)})_{m\in\M}$ if we are able to find $\pen(m)$ such that for a constant $\alpha\in (0,2)$, $\pen(m)$ is close to $(1+\alpha)V_p(m)$ for all $m\in\M$. Note that $V_p(m)$ is not observable. Therefore, we need concentration inequalities to find the suitable expression of $\pen(m)$ that is involved in our procedure.  Since $V_p(m)=\e^p\sum_{\la \in m}w_\lambda  | \xi_{\la}|^p$, we need to study  sums of  $|\xi_\lambda|^p$ where the $\xi_{\lambda}$'s are independent sub-Gaussian variables. 

%%%%%%%%%%%%%%%%%%
\subsection{Concentration inequalities}
Let $\mathcal{I}\subset\Lambda$. We denote 
$D=\mbox{card}(\mathcal{I})$ and $$Z:= \sum_{\lambda \in \mathcal{I}} |\xi_\lambda|^p.$$
Using, for $r>0$, the Orlicz norm of a random variable $X$, defined by
$$ \|X\|_{\psi_{r} } = \inf \big\{\eta >0:\quad \E [\exp((|X|/\eta)^r)]\leq 2\big\},$$ %\quad\text{ with } r=2/p.$$
we set
$b_\lambda=
\Big\||\xi_{\lambda}|^p-\E|\xi_{\lambda}|^p\Big\|_{\psi_{2/p}}$.
We have the following result.
 \begin{theorem}\label{theo:concentration}
Let $p\geq 1$. Assume that the $\xi_{\lambda}$'s are centered independent sub-Gaussian variables. 
There exist positive constants  $d_{1,p}$ and  $d_{2,p}$ only depending on $p$ such that, for any $x>0$,
$$\P(|Z-\E(Z)|\geq d_{1,p}\|b\|_{\ell_2}\sqrt{x}+d_{2,p}
\|b\|_{\ell_{1/(1-p/2)_+}}x^{p/2})\leq 2e^{-x}.$$
Moreover,  if the $\xi_{\lambda}$'s are also identically distributed, then,   for any $x>0$,
$$\P\Big(|Z-\E[Z]|\geq c_{1,p}\sqrt{Dx}+c_{2,p}D^{(1-p/2)_+}x^{p/2}\Big)\leq 2e^{-x},$$
where $c_{1,p}$ and $c_{2,p}$ only depend on $p$ and $\|\xi_\lambda\|_{\psi_2}$.
\end{theorem}
Note that for the Gaussian i.i.d. case, when $p=1$, the Cirelson-Ibragimov-Sudakov inequality gives  $c_{1,1}=\sqrt{2}$ and $c_{2,1}=0$ (see Theorem 3.4 in \cite{Massart2007}).
When $p=2$,  we retrieve the well-known inequality for chi-squared variables and $c_{1,2}=c_{2,2}=2$ works: see Lemma~1 of \cite{LaurentMassart}. This also matches with
the Hanson-Wright inequality with identity matrix or Bernstein inequality for sub-exponential variables, see e.g. \cite{vershynin2018}.\\
%Bersntein theorem 2.8.2
%Hanson-Wright Th 6.2.1
In the general case $p\geq 1$, the result ensues from concentration theorems for sub-Weibull variables. Indeed, if we denote $X_{\lambda}=|\xi_{\lambda}|^p-\E|\xi_{\lambda}|^p$, we observe that the sub-Gaussianity property of $\xi_{\lambda}$ entails that $X_\lambda$ has a Weibull behavior: 
$$\P(|X_{\lambda}|\geq t) \leq C\exp(-ct^{2/p}),$$
for $C$ a constant. Such a variable, with bounded Orlicz norm with function $e^{x^r}-1$, is called a sub-Weibull variable.
Note that the Weibull parameter here is $r=2/p\leq 2$. Now our theorem directly follows from recent Theorem 1 of \cite{ZhangWei22},  or Theorem 3.1 of \cite{KuchibhotlaChakrabortty} (see also their Equation~(3.6)), both giving  explicit formulas for $d_{1,p}$ and $d_{2,p}$. In the i.i.d case, all  $b_\lambda$'s are equal and bounded by $\||\xi_\lambda|^p\|_{\psi_{2/p}}=\|\xi_\lambda\|_{\psi_{2}}$, up to a universal constant.

Observe that Theorem~\ref{theo:concentration} can also be deduced from older results, like moment bounds of \cite{GluskinKwapien95} for the case $r>1$. In particular, their corollary shows that our bound cannot be improved when $p<2$. For $r<1$, Theorem 6.2 of \cite{Hitczenko97} gives a two-sided moments inequality that leads to our tail result. 
% (see also Lemma 3.6 de Adamczak et al 2011).
Their bound cannot be improved meaning that if the $\xi_{\lambda}$'s are Gaussian, the upper bound is achieved up to a constant. 
Same optimality considerations are raised by \cite{KuchibhotlaChakrabortty}.

From Theorem~\ref{theo:concentration}, we obtain the following corollary.
\begin{corollary}\label{coro:concentration}
Let $p\geq 1$. Assume that the $\xi_{\lambda}$'s are i.i.d.  centered sub-Gaussian variables. For any $x\geq 1$,
with probability larger than $1-2\exp(-x)$,
\begin{equation}\label{conc-coro}
\frac{1}{2}\sigma_p^pD-\kappa_pD^{\big(1-\frac{p}{2}\big)_+}x^{\frac{p}{2}}\leq \sum_{\lambda \in \mathcal{I}} |\xi_\lambda|^p\leq \frac{3}{2}\sigma_p^pD+\kappa_pD^{\big(1-\frac{p}{2}\big)_+}x^{\frac{p}{2}},
\end{equation}
where $\kappa_p=c_{2,p}+c_{1,p}\max(1, c_{1,p}/(2\sigma_p^p))$ is a positive constant only depending on $p$ and $\sigma_p$ and $\|\xi_\lambda\|_{\psi_2}$.
%and $\sigma_p$.
\end{corollary}
\begin{proof} We obviously have $\E[Z]=D \sigma_p^p$. If $p\geq  2$, we use $2\sqrt{Dx}\leq \theta D+\theta^{-1}x$ with $\theta=\sigma_p^p/c_{1,p}$, and the inequality $x\leq x^{p/2}$. 
If $p<2$ and $x\leq D$, we use the same bound for $\sqrt{Dx}$ with the same $\theta$, %=2\epsilon/c_{1p}$, 
and this time 
$x\leq D^{1-p/2}x^{p/2}$. Finally, if $p<2$ and $x> D$, we directly write $\sqrt{Dx}\leq D^{1-\frac{p}{2}}x^{\frac{p}{2}}$.
\end{proof}
Note that Corollary~\ref{coro:concentration} holds by replacing 1/2 (resp. 3/2) in \eqref{conc-coro} by $(1-\epsilon)$ (resp. $(1+\epsilon)$) for $\epsilon $ arbitrary small (with $\kappa_p$ depending on $\epsilon$).
\subsection{Refined oracle inequalities}
 In this section, we apply the general oracle inequality of Theorem~\ref{theo:oracle1} with $\alpha=1$ (see Remark~\ref{alpha}) and use sharp concentration inequalities of Corollary~\ref{coro:concentration} to derive suitable penalties. %\comm{Inequality \eqref{subgauss}  means that $\|\xi_\lambda\|_{\psi_2}$ is a fixed constant here.}

Typically, weights $w_\lambda$ may not depend on $\lambda$ or may be constant on some slices (see subsequent sections). Therefore, we consider the following partition of $\Lambda^{(N)}$:
\begin{equation}\label{permu}
\Lambda^{(N)}=\bigcup_{j\in\J} \Lambda_j
\end{equation}
so that $w_\lambda$ is constant for any $\lambda\in \Lambda_j$ with $w_\lambda=\omega_j$. For ease of notation, we omit the dependence of the $\Lambda_j$'s and $\J$ on $N$. We also assume that $\omega_j\not= \omega_{j'}$ if $j\not= j'$. Therefore, up to some permutation of elements of $\J$, the partition \eqref{permu} is unique. Then, given  \eqref{permu}, we consider for any model $m\in\M$
$$m_j=m\cap \Lambda_j,\quad j\in\J,$$
and we set
$${\mathcal M}_j=\big\{m_j:\ m\in\M\big\},\quad j\in\J.$$
In view of Inequality~\eqref{alpha=1} and Corollary~\ref{coro:concentration}, we take
\begin{equation}\label{pen}
%\pen(m):=(1+\alpha)\e^p\sum_{j=1}^J\omega_j p(m_j),
\pen(m):=2\e^p\sum_{j\in\J}\omega_j \p_j(m_j),
\end{equation}
with, for  some $x_{m_j}\geq 1$,
\begin{equation}\label{p1}
\p_j(m_j)=\frac{3}{2}\sigma_p^p|m_j|+\kappa_p2^{\frac{(p-2)_+}2}|m_j|^{\big(1-\frac{p}{2}\big)_+}x_{m_j}^{\frac{p}{2}}.
\end{equation}
Observe that:
$$\p_j(m_j)=\left\{
\begin{array}{lc}
\frac{3}{2}\sigma_p^p|m_j|+\kappa_p|m_j|^{1-\frac{p}{2}}x_{m_j}^{\frac{p}{2}}&\mbox{ if }p\leq 2,\\
\frac{3}{2}\sigma_p^p|m_j|+\kappa_p2^{\frac{p}{2}-1}x_{m_j}^{\frac{p}{2}}&\mbox{ if }p\geq 2.
\end{array}
\right.$$
The factor $x_{m_j}$, depending on $m_j$, will be specified later. But note that, mimicking the computations of Section~\ref{sec:model}, the thresholding rule corresponds to the case where $\p_j(m_j)$ is proportional to $|m_j|$, which is obtained for instance by taking $x_{m_j}$ proportional to $|m_j|$ when $p\leq 2$ and in this case the threshold is proportional to the noise level $\e$ as expected. However, subsequent results show that the resulting estimate is suboptimal in some situations by at least a logarithmic factor.
When $p=2$, $\pen(m)$ corresponds to the penalty extensively used to derive oracle inequalities for model selection procedures on Hilbert spaces. See, for instance, oracle inequalities established in Theorems 4.2, 4.5 and 4.18 of \cite{Massart2007}. The extension of these results to the case $p\not=2$ is provided by the following theorem.
\begin{theorem}\label{theo:oracleesp}
Let %$\alpha>0$ and 
$p\geq 1$. We consider the estimate $\tilde\theta=\hat \theta^{(\hat m)}$ associated with the penalty defined in \eqref{pen} and $\p_j(m_j)$ given in \eqref{p1}.
  Then  
\begin{equation}\label{inegoracleesp}
\E\Big[\|\tilde \theta-\theta\|_p^p\Big]\leq 
\tilde M_{p}%\tilde M_{p,\alpha}
\inf_{m\in{\mathcal M}}\left\{\E\Big[\|\hat \theta^{(m)}-\theta\|_p^p\Big] +\pen(m)\right\}+\breve M_{p}%\breve M_{p,\alpha}
 \varepsilon^p R(\M)
\end{equation}
with 
\begin{equation}\label{RM-def}
R(\M)=\sum_{j\in\J}\omega_j\sum_{m_j\in \M_j,\,{m_j\neq \emptyset}}|m_j|^{\left(1-\frac{p}{2}\right)_+}e^{-x_{m_j}},
\end{equation}
and $\tilde M_{p}$ and $\breve M_{p}$ are two constants only depending on $p$ and $\sigma_p$.
\end{theorem}
The proof of Theorem~\ref{theo:oracleesp} is provided in Section~\ref{preuveoraclesp2}.
\begin{remark}
Theorem~\ref{theo:oracleesp} remains true if we replace the equality in \eqref{pen} by the inequality:
$$\pen(m)\geq 2\e^p\sum_{j\in\J}\omega_j \p_j(m_j).$$
\end{remark}
Now let us discuss the choice of the factors $x_{m_j}$. We fix them in order 
$\tilde \theta$ to be optimal in the oracle setting,  meaning that 
\begin{equation}\label{opt-oracle}
\E\Big[\|\tilde \theta-\theta\|_p^p\Big]\lesssim \inf_{m\in{\mathcal M}}\E\Big[\|\hat \theta^{(m)}-\theta\|_p^p\Big].
\end{equation}
Now, observe that for any model $m$,
\begin{align*}
\E\Big[\|\hat \theta^{(m)}-\theta\|_p^p\Big]&=B_p(m)+\sum_{j\in\J} \E\big[V_p(m_j)\big]
%\\&
=\sum_{\la \notin m}w_\lambda | \theta_{\la}|^p+\e^p\sigma_p^p\sum_{j\in\J}w_j|m_j|
\end{align*}
and $\pen(m)$ can be compared to the second term of the right hand side.
In particular, Theorem~\ref{theo:oracleesp} shows that $\tilde \theta$ is optimal as soon as we have
\begin{equation}\label{cond-exp}
|m_j|^{\left(1-\frac{p}{2}\right)_+}x_{m_j}^{\frac{p}{2}}\lesssim |m_j|\iff x_{m_j}\lesssim |m_j|^{2/\max(2,p)},\quad \text{for all } j\in\J
\end{equation}
and 
\begin{equation}\label{RM}
R(\M)<\infty.
\end{equation}
 More precisely, under \eqref{cond-exp} and \eqref{RM}, we obtain:
$$\E\Big[\|\tilde \theta-\theta\|_p^p\Big]\lesssim \inf_{m\in{\mathcal M}}\E\Big[\|\hat \theta^{(m)}-\theta\|_p^p\Big]+ \varepsilon^p,$$
which corresponds to \eqref{opt-oracle} up to the residual term $\varepsilon^p.$ Therefore, we wish to fix the factors $x_{m_j}$ so that Conditions \eqref{cond-exp} and \eqref{RM} are satisfied. Condition~\eqref{cond-exp} means that the factors cannot be too large. But the condition $R(\M)<\infty$ holds only if $\M$ is not too large or factors are large enough. In particular, to have \eqref{RM}, we need:
\begin{equation}\label{inegmj}
\sum_{m_j\in \M_j,\,{m_j\neq \emptyset}}|m_j|^{\left(1-\frac{p}{2}\right)_+}e^{-x_{m_j}}<\infty
\end{equation}
(see \eqref{RM-def}).
Given $d\geq 1$, consider situations where the number of models of size $d$ is exponential in $d$, say $\exp(cd)$ for $c>0$. When $p\leq 2$, we can choose $x_{m_j}$ satisfying \eqref{cond-exp} and such that \eqref{inegmj} holds by taking
$$x_{m_j}=\tilde c |m_j| \quad\text{ or }\quad x_{m_j}=\tilde c\log(|m_j|)$$
with $\tilde c$ large enough.
But, when $p>2$, if \eqref{cond-exp} is verified, we have, for $c_*$ a positive constant,
\begin{align*}
\sum_{m_j\in \M_j,\,{m_j\neq \emptyset}}|m_j|^{\left(1-\frac{p}{2}\right)_+}e^{-x_{m_j}}\geq\sum_{d=1}^{
D_j}\exp(cd- c_* d^{2/p}),
\end{align*}
with $D_j=\max_{m_j\in \M_j}(|m_j|)$ (assumed to be larger than 1). The right hand side becomes very large when $D_j$ is large, which occurs when the model collection is large.
Therefore, the optimality of $\tilde \theta$ for the case $p>2$ requires a modification of the penalty to deal with large collections of models. Since the elbow occurs at the value $p=2$, we use $\p_j(m_j)$ with the value $p=2$ and we introduce
\begin{equation}\label{p2}
\p_j^\#(m_j)=\frac32\sigma_2^2|m_j|+\kappa_2 x_{m_j}.
\end{equation}
We obtain the following theorem.
\begin{theorem}\label{theo:oracleesp2}
Let $q>1$ and $p\geq 2$. Let $m\in{\mathcal M}$. For  $j\in \J$ and $m_j\in{\mathcal M}_j$, we take $x_{m_j}\geq 1$ and we consider the estimate $\tilde\theta=\hat \theta^{(\hat m)}$ associated with the penalty
 %$$\pen(m):=(1+\alpha)\e^p\sum_{j=1}^J\omega_j \min\Big(p(m_j),(2q\log N)^{\frac{p}{2}-1}p^\#(m_j)\Big),$$
 \begin{equation} \label{pen2}
 \pen(m):=2\e^p\sum_{j\in\J}\omega_j \min\Big(\p_j(m_j),(2q\log N)^{\frac{p}{2}-1}\p_j^\#(m_j)\Big),
 \end{equation}
 where $\p_j(m_j)$ and $\p_j^\#(m_j)$ are defined in \eqref{p1} and \eqref{p2} respectively.
 Then  
\begin{equation}\label{inegoracleesp2}
\E\Big[\|\tilde \theta-\theta\|_p^p\Big]\leq 
\tilde M_{p}%\tilde M_{p,\alpha}
\inf_{m\in{\mathcal M}}\left\{\E\Big[\|\hat \theta^{(m)}-\theta\|_p^p\Big] +\pen(m)\right\}+
\breve M_{p,q}%\breve M_{p,q,\alpha}
\Big(N^{1-q}\|\theta\|_p^p +\varepsilon^p R^\#(\M)\Big),
\end{equation}
with $\tilde M_{p}$ (resp. $\breve M_{p,q}$) a constant only depending on $p$ (resp. $p$, $\sigma_p$ and $q$) and
$$R^\#(\M)=N^{(1-q)/2}\sum_{\lambda\in\Lambda}w_\lambda+ (\log N)^{\frac{p}{2}-1}R(\M),$$
where $R(\M)$ is defined in \eqref{RM-def}.
\end{theorem}
The proof of Theorem~\ref{theo:oracleesp2} is provided in Section~\ref{preuveoraclesp3}. Unfortunately, the use of $\p_j^\#(m_j)$ requires to add the multiplicative logarithmic term $(2q\log N)^{\frac{p}{2}-1}$ which vanishes only if $p=2$.

If $\log N$ is of order of $|\log(\e)|$, by taking $q$ large enough, the residuable term of the oracle inequality is the same as in \eqref{inegoracleesp} up to the logarithmic term $|\log(\e)|^{\frac{p}{2}-1}$. Two scenarios are then of interest.
\begin{enumerate}
\item If we can take $x_{m_j}$ satisfying \eqref{cond-exp} such that $R(\M)<\infty$, we use
 $$\pen(m)\leq 2\e^p\sum_{j\in\J}\omega_j \p_j(m_j)$$
and Theorem~\ref{theo:oracleesp2} gives:
$$\E\Big[\|\tilde \theta-\theta\|_p^p\Big]\lesssim \inf_{m\in{\mathcal M}}\E\Big[\|\hat \theta^{(m)}-\theta\|_p^p\Big]+ \varepsilon^p |\log(\e)|^{\frac{p}{2}-1}.$$
\item If we can take $x_{m_j}\approx |m_j|$ such that $R(\M)<\infty$, we use
 $$\pen(m)\lesssim|\log(\e)|^{\frac{p}{2}-1}\e^p\sum_{j\in\J}\omega_j \p_j^\#(m_j)\lesssim |\log(\e)|^{\frac{p}{2}-1}\e^p\sum_{j\in\J}w_j|m_j|$$
and Theorem~\ref{theo:oracleesp2} gives:
$$\E\Big[\|\tilde \theta-\theta\|_p^p\Big]\lesssim |\log(\e)|^{\frac{p}{2}-1}\inf_{m\in{\mathcal M}}\E\Big[\|\hat \theta^{(m)}-\theta\|_p^p\Big]+ \varepsilon^p |\log(\e)|^{\frac{p}{2}-1}.$$
\end{enumerate}
The first scenario provides optimality of $\tilde \theta$ up to the residual term $\varepsilon^p |\log(\e)|^{\frac{p}{2}-1}.$ An additional logarithmic factor $|\log(\e)|^{\frac{p}{2}-1}$ is required for the main term for the second scenario. In each case, the condition $R(\M)<\infty$, depending on the chosen collection $\M$ is crucial. In the next section devoted to the minimax setting, $\M$ will be specified  in order to study optimality of $\tilde \theta$ under different regularity conditions for $\theta$.
%
%%%%%%%%%%%%%%%%%%
%%%%%%%%%%%%%%%%%%
\section{Minimax approach}\label{secg:minimax}
The goal of this section is to prove the optimality of our procedure. For this purpose, we consider the minimax setting. Two cases are considered. In next Section~\ref{sec:constweights}, we illustrate our results on a simple situation, namely the case of constant weights and sequences $\theta$ whose tails have a polynomial decreasing. Then, in Section~\ref{sec:minimax}, we investigate the minimax rates of convergence of our procedure on Besov bodies.
%%%%%%%%%%%%%%%%%%
\subsection{Case of constant weights}\label{sec:constweights}
In this section, we assume that weights $w_\lambda$ do not depend on $\lambda$ and without loss of generality, we assume:
\[
w_\lambda=1,\quad \lambda\in\Lambda.
\]
We study minimax rates of convergence of our procedure when tails of $\theta$ have a polynomial decreasing. Therefore, assuming without loss of generality that $\Lambda$ is the set of positive integers denoted $\Ne^*$, we introduce for $s>0$ and $R>0$, the set ${\mathbb B}^s_p(R)$ defined by 
\[
{\mathbb B}^s_p(R):=\left\{\vartheta=(\vartheta_{\lambda})_{\lambda\in\Ne^*},\quad \sup_{k\in\Ne^*}k^s\Big(\sum_{\lambda>k} |\vartheta_\lambda|^p\Big)^{1/p}\leq R \right\}.
\]
This functional class, allowing for an easy control of the bias term, is natural in our setting. 
 In the sequel, $\theta$ is assumed to belong to ${\mathbb B}^s_p(R)$. Now, we take:
\[
\M=\Big\{\{1,\dots, k\},\quad  1\leq k\leq N\Big\}
\]
and we apply Theorem~\ref{theo:oracleesp} by plugging the value $x_m=a\log(|m|)$ (and $x_{\emptyset}=0$) in $\pen(m)$  with $a>1+\left(1-\frac{p}{2}\right)_+$. In this case, since  for each value $d\in\Ne^*$ there is only one model $m\in\M$ with cardinal $d$, we obtain:
\begin{align*}
R(\M)&=\sum_{j=1}^J\omega_j\sum_{m_j\in \M_j,\,{m_j\neq \emptyset}}|m_j|^{\left(1-\frac{p}{2}\right)_+}e^{-x_{m_j}}\\
&\leq\sum_{d=1}^{+\infty}d^{\left(1-\frac{p}{2}\right)_+}e^{-a\log(d)}<\infty.
\end{align*}
In the previous inequality, we have used that $\mathcal J=\{1\}$ and $\Lambda_1=\Lambda^{(N)}$; therefore, $m_j=m$ for any $j$ and any $m\in \mathcal M$. Theorem \ref{theo:oracleesp} gives
\begin{equation*}
\E\Big[\|\tilde \theta-\theta\|_p^p\Big]\leq 
\tilde M_{p}%\tilde M_{p,\alpha}
\inf_{m\in{\mathcal M}}\Big\{\E\Big[\|\hat \theta^{(m)}-\theta\|_p^p\Big] +\pen(m)\Big\}+\breve M_{p}
 \varepsilon^p R(\M).
\end{equation*}
Since 
\[ 
\E\Big[\|\hat \theta^{(m)}-\theta\|_p^p\Big]=B_p(m)+\E\big[V_p(m)\big]
\]
we have
\[
\pen(m)\lesssim \e^p|m|\lesssim \E\big[V_p(m)\big]
\]
and we finally obtain:
\[
\E\Big[\|\tilde \theta-\theta\|_p^p\Big]\leq 
\tilde M_{p}'\inf_{m\in{\mathcal M}}\E\Big[\|\hat \theta^{(m)}-\theta\|_p^p\Big],
\]
for $\tilde M_{p}'$ a constant only depending on $p$. For a model $m=\{1,\dots, k\}$, 
\[
\E\big[V_p(m)\big]=\sigma_p^p k\e ^p\quad \mbox{and}\quad 
B_p(m)=\sum_{i>k} |\theta_i|^p.
\]
We obtain: 
\[
\sup_{\theta\in {\mathbb B}^s_p(R)}\E\Big[\|\tilde \theta-\theta\|_p^p\Big]\lesssim 
\inf_{1\leq k\leq N}\Big\{R^pk^{-sp}+\e^p k\Big\}\lesssim R^{\frac{p}{1+ps}}\e^{\frac{p^2s}{1+ps}}
\]
as soon as $N\geq  R^p\e^{-p}$. In particular, when $p=2$, we obtain the bound
\[
\sup_{\theta\in {\mathbb B}^s_2(R)}\E\Big[\|\tilde \theta-\theta\|_2^2\Big]\lesssim R^{\frac{2}{1+2s}}\e^{\frac{4s}{1+2s}}.
\]
In particular, the rate in the right hand side is the optimal minimax rate on the class ${\mathbb B}^s_2(R)$, see \cite{Rivoirard04}. 
% \vinc{On devrait se faire allumer par le referee pour le cas $p\not=2$.}
%%%%%%%%%%%%%%%%%%%
\subsection{Case of non-constant weights and minimax rates on Besov bodies}\label{sec:minimax}
%%%%%%%%
\subsubsection{Setting}\label{sec:setting}
In this section, we wish to consider minimax rates on the very general class of Besov bodies. The latter is naturally associated with the wavelet framework. Therefore, we naturally adapt previous notations to this framework and,  without loss of generality, we can rewrite the setting of Section~\ref{sec:model} by assuming that
$$\Lambda=\bigcup_{j\geq -1} \Big\{\{j\}\times K_j\Big\},\quad K_j=\Big\{k\in\Ne: \ 0\leq k<2^j\Big\}.$$
Such adaptations are also justified by the extensions of subsequent results to the regression framework studied in Section~\ref{sec:regression}. 
For any $j\geq 0$, we have $|K_j|=2^j$ and $|K_{-1}|=1$ since $K_{-1}=\{0\}$. Now, our statistical model writes:
\begin{equation}\label{def:model2}
Y_{jk}=\theta_{jk} + \e \xi_{jk}, \qquad j\geq -1, \ k\in K_j.
\end{equation}
In the sequel, we shall assume that the sequence $\theta=(\theta_{jk})_{(j,k)\in\Lambda}$ belongs to a Besov ball with smoothness $s$ defined, as usual, by
$$\mathcal{B}_{r,\infty}^{s}(R)
%=\mathcal{N}_s^\alpha(R)
=\left\{\vartheta=(\vartheta_{jk})_{(j,k)\in\Lambda}\in{\R^{\Lambda}}:\quad \sup_{j\geq -1} 2^{j(s+\frac12-\frac1r)}\Bigg(\sum_{k\in K_j}|\vartheta_{jk}|^r \Bigg)^{1/r}\leq R \right\}$$
for $0<s<\infty$, $1\leq r<\infty$, $0<R<\infty$ and
$$\mathcal{H}^s(R)=\mathcal{B}_{\infty,\infty}^{s}(R)=\left\{\vartheta\in \R^{\Lambda}:\quad \sup_{j\geq -1,\, k\in K_j} 2^{j(s+\frac12)}|\vartheta_{jk}|\leq R \right\}.$$
The sequence $\theta$ will be estimated by using the first $N$ observations $Y_{jk}$, with $N=2^{J+1}$ for some $J\geq 0$.
% \vinc{[Je passe sous silence $\J$ qui serait ici de cardinal $J+2$...}\comm{ Ca me parait embetant que le $J$ de il y a 2 pages devienne $J+2$, je suggere de se passer de $J$ dans la section 2, et de ne bosser qu'avec du $\mathcal{J}$, ensuite ici on aura $\mathcal{J}=\{-1, \dots, J\}$]}
 It means that 
$$\Lambda^{(N)}=\bigcup_{-1\leq j\leq J}\Lambda_j,\quad \Lambda_j=\Big\{\{j\}\times K_j\Big\}.$$
We still denote for any sequence $\vartheta$ in $\R^\Lambda$, 
\begin{equation}\label{def-normlp}
\|\vartheta\|_p^p=\sum_{j=-1}^{+\infty}\sum_{k\in K_j} w_{jk}|\vartheta_{jk}|^p=\sum_{j=-1}^{+\infty}\omega_j\sum_{k\in K_j} |\vartheta_{jk}|^p,
\end{equation}
but we assume that the weights satisfy 
\begin{equation}\label{def-weights}
 \quad w_{jk}=\omega_j=2^{j\left(\frac{p}{2}-1\right)},\quad (j,k)\in\Lambda. 
 \end{equation}
Such weights are naturally justified by the subsequent regression framework and strong relationships between the sequence norm $\|\cdot\|_p$ and the classical functional $\L_p$-norm. See Section~\ref{sec:regression}.
%%%%%%%%
\subsubsection{Lower bounds}
In this section, we still consider Model \eqref{def:model2}  but we assume that the $\xi_{jk}$'s are i.i.d.  standard Gaussian variables.
The lower bound for the $\ell_p(w)$-risk is known  for $s>1/r$, see Theorems 7 and 9  of \cite{DJKP2} (by taking, using their notation, $(\sigma,p,q)=(s,r,\infty)$, $(\sigma',p',q')=(0,p,p)$ and $C=R$).

%\color{gray}
%Rk: si je ne me trompe pas le cas de la norme $(\sigma',p',q')=(0,p,1)$ nous interessera aussi car c'est $\sum_j\left(\omega_j\sum_{k}|\theta_{jk}|^p \right)^{1/p}$, mais la borne inf ci-dessous marche pour les deux.
%
%Remarque : resultat en proba donc marche avec esperance et n'importe quelle puissance
%\color{black}

%$s>(1/r-1/p)_+$:
\begin{theorem} 
Assume that $s>1/r$ and $\e^{-2}\lesssim N$. %$s>(1/r-1/p)_+$. 
Then, for $\e$ small enough, we have:
$$\inf_{\hat \theta}\sup_{\theta\in \mathcal{B}_{r,\infty}^{s}(R)}
\E\Big[\|\hat \theta-\theta\|_p^p\Big]
\geq c\begin{cases}
R^{\frac{p}{2s+1}}\e^{\frac{2ps}{2s+1}} & \text {if } r>\frac{p}{2s+1}\quad \\
R^{\frac{p}{2s+1}}\e^{\frac{2ps}{2s+1}}|\log(\e)|^{\frac{ps }{2s+1}+1} &  \text{if }    r=\frac{p}{2s+1}\\
R^{\frac{p-2}{2s+1-\frac2r}}
\big(\e^{2}|\log(\e)|\big)^{\frac{p(s-\frac{1}{r}+\frac{1}{p})}{2s+1-\frac2{r}}} & \text {if } r< \frac{p}{2s+1} \quad 
\end{cases}$$
where the  infimum is taken over all estimators, i.e. measurable functions of $(Y_{jk})_{(j,k)\in \Lambda}$ and $c$ is a positive constant depending on  $p,$ $r$ and $s$.
\end{theorem}

The lower bounds reveal several zones according to the sign of $r-p/(2s+1)$. The case $r<p/(2s+1)$ will be called the \textit{sparse case}. The dense case $r>p/(2s+1)$ can be decomposed into two cases: the case $r\geq p$ and the case $p/(2s+1)<r<p$ referred respectively as the \textit{homogeneous} and \textit{intermediate cases}. Finally the case $r=p/(2s+1)$ will be called the \textit{frontier case.} We refer to the zones for the upper bounds.

%
%%%%%%%%%%%%%%%%%%
\subsubsection{Upper bounds}
\label{sec:upperbounds}
Our aim is to obtain the rate of convergence on the class $\mathcal{B}_{r,\infty}^{s}(R)$ of our estimator for some well-chosen collection of models. 
Our collection of models is the union of three sub-collections $\M^H$, $\M^{I}$ and $\M^S$, each of them being the most suitable for homogenous ($H$), intermediate ($I$) and sparse {and frontier} cases ($S$). For each sub-collection, we use a penalty  (see \eqref{pen} and \eqref{pen2}) with corresponding factor $x_{m_j}^H$, $x_{m_j}^I$ and $x_{m_j}^S$ (see below).
%, and $q=p+1$ . 
The selection of the sub-collection and associated penalty is performed by our procedure automatically. We denote %$\mathcal{A}=\{H, I, S\}$ and 
$$\frak{M}=\bigcup_{a\in\{H, I, S\}} \M^a\times \{a\}$$
with $\M^a$ defined as follows (recall that $m=\cup_{-1\leq j\leq J}m_j$):

\begin{itemize}
\item strategy $a=H$: $m\in \M^H$ if there exists $L\in\{0,\dots, J\}$ such that 

for all $-1\leq j \leq J,\quad $ 
$m_j=\begin{cases}
 \Lambda_j & \text{ if } -1\leq j\leq L,\\
\emptyset& \text{ if } j>L
\end{cases}$ %\quad j=-1,\dots, J$
%$\M^H=\{\Lambda_{-1}\cup \Lambda_0\cup\dots\cup\Lambda_{J_m}, \quad -1\leq J_m\leq J\}$,

and $x^H_{m_j}=\frac{p}{2}\log|m_j|$ ($\log 0=0$); % with associated $\pen^H(m),\p^H_j(m_j)$ and $\p_j^{\# H}(m_j)$

\item strategy $a=I$:
%
%$\M^I=\bigcup_{L= 0} ^{J}\M(L)$ where $m\in \M(L)$ if for all $-1\leq j \leq J$
$m\in \M^I$ if there exists $L\in\{0,\dots, J\}$ such that 

{\small
for all $-1\leq j \leq J,\quad $ 
$m_j=\begin{cases}
 \Lambda_j & \text{ if } -1\leq j\leq {L-1}\\
m_{L+l} \subset \Lambda_{L+l}& \text{ if } l=j-L\geq 0 \text{ with }|m_{L+l}|=\lfloor 2^{L+l} 2^{-lp/2}(l+1)^{-3} \rfloor,
\end{cases}$}

and
$x^I_{m_j}=K|m_j|\bigg(1+\log\Big(\frac{2^j}{|m_j|}\Big)\bigg)$ with $K$ a constant only depending on $p$ (see the proof); 
%with associated $\pen^I(m),\p^I_j(m_j)$ and $\p_j^{\# I}(m_j)$

\item strategy $a=S$:
% $\M^S=\{\{j\}\times E, -1\leq j \leq J, E\in \mathcal{P}(K_j)\}$
$ m\in \M^S$ (full collection) if 

for all $-1\leq j \leq J,\quad $ 
$m_j=\{j\}\times E, \quad E\in \mathcal{P}(K_j),$

and  $x_{m_j}^S=(p+1)|m_j| j $. 
%with associated $\pen^S(m),p^S_j(m_j)$ and $\p_j^{\# S}(m_j)$
\end{itemize}
\color{black}
For $(m,a)\in \frak{M}$, having in mind Theorems~\ref{theo:oracleesp} and \ref{theo:oracleesp2}, denote 
\begin{equation*}\label{penalite}
\frak{pen}(m,a)=\pen^a(m)=
\begin{cases}
2\e^p\sum_{j=-1}^J\omega_j \frak{p}_j(m_j,a)
& \text{ if } p\leq 2\\
2\e^p\sum_{j=-1}^J\omega_j \min\Big(\frak{p}_j(m_j,a),(2q\log N)^{\frac{p}{2}-1}\frak{p}_j^\#(m_j,a)\Big)
& \text{ if } p> 2
\end{cases}
\end{equation*}
with $q=p+1$ and 
\begin{eqnarray*}
\frak{p}_j(m_j,a)&= &{\rm p}^a_j(m_j)=\frac32\sigma_p^p|m_j|+\kappa_p2^{\frac{(p-2)_+}2}|m_j|^{\big(1-\frac{p}{2}\big)_+}(x^a_{m_j})^{\frac{p}{2}},\\
\frak{p}_j^\#(m_j,a)&= &{\rm p}_j^{\# a}(m_j)=\frac32\sigma_2^2|m_j|+\kappa_2 x^a_{m_j}.
\end{eqnarray*}
We consider $(\hat m, \hat a)$ which minimizes over all $(m,a)\in\frak{M}$ the criterion
$$-\sum_{\la \in m}w_\lambda|Y_{\la}|^p+\frak{pen}(m,a).$$
It can also be defined by minimizing over $a$ the quantity $-\sum_{\la \in \hat m^a}w_\lambda|Y_{\la}|^p+\pen^a(\hat m^a)$ where 
$\hat m^a$ minimizes over $\M^a$ the quantity
$-\sum_{\la \in m}w_\lambda|Y_{\la}|^p+\pen^a(m)$. 
Finally our PCO estimator is $\tilde\theta=\hat\theta^{(\hat m, \hat a)}$. 
\begin{theorem}\label{rateslp}
Assume that $s>1/r$ and $R$ the radius of the Besov ball belongs to the interval $[\e,\e^{-1}]$.  We take the number of observations $N$ such that {$\left(\frac{R}{\e}\right)^2\leq N\leq\left(\frac{R}{\e}\right)^{\gamma}$}, with $\gamma\geq 2$. Then
$$\sup_{\theta\in\mathcal{B}_{r,\infty}^{s}(R)}\E\Big[\|\tilde\theta-\theta\|_p^p\Big]\leq C\left\{
\begin{array}{ccl}
R^{\frac{p}{2s+1}}\e^{\frac{2ps}{2s+1}}  &   \text{if}  &r\geq p   \\
R^{\frac{p}{2s+1}}\e^{\frac{2ps}{2s+1}}|\log(\e)|^{\frac{s(p-2)_+}{2s+1}}  &  \text{if} &   \frac{p}{2s+1}<r<p\\
R^{\frac{p}{2s+1}}\e^{\frac{2ps}{2s+1}}|\log(\e)|^{\frac{ps }{2s+1}+1} &  \text{if} &   r=\frac{p}{2s+1}\\
R^{\frac{p-2}{2s+1-\frac2r}} \big(\e^{2}|\log(\e)|\big)^{\frac{p(s-\frac{1}{r}+\frac{1}{p})}{2s+1-\frac2{r}}}  &  \text{if} &   r<\frac{p}{2s+1}
\end{array}
\right.$$
for $C$ a constant depending on $p$, $\sigma_p$, $s$, $r$ {and $\gamma$.}
%\\ \vinc{Comme on suppose $\in [R_0,R_1]$, est-ce que l'on met la dependance en $R$ ?}
\end{theorem}
We remark that the radius $R$ may decrease to 0 or increase to $+\infty$ when $\e\to 0$. The lower bound $R\geq \e$ allows to have $N\geq 1$. The upper bound $R\leq \e^{-1}$ can be replaced by any upper bound $R\leq \e^{-u}$, with $u> 0$, so we have $\log (N)\approx |\log(\e)|$. 
%En fait j'ai besoin de \c{c}a car dans les lower bounds du Theorem 3.1 il n'y a pas de dependance en $R$ dans les log. C'est tres bien ainsi :-) 
Note that if $R\in[R_0,R_1]$ with known constants $R_0$ and $R_1$ with $0<R_0<R_1<\infty$ then we can choose $N$ not depending on $R$.

The proof of Theorem~\ref{rateslp} is provided in Section \ref{proofupper}.
%Hypotheses utilisees : $N/2=2^J\geq (R/\e)^2$ et $\log(N)\lesssim \log(\e^{-1})$
%cas homogene  $N/2=2^J\geq (R/\e)^2$
%
%cas intermediaire $N/2=2^J\geq (R/\e)^2$ et $\log(N)\lesssim \log(\e^{-1})$
%
%cas  sparse $N/2=2^J\geq (R/\e)^2$ et $\log(N)\lesssim \log(\e^{-1})$
%We use three different collections of model $\M$ in the three cases.... \vinc{On peut dire quelque chose du cas $r=\frac{p}{2s+1}$. }\comm{ Dans ce cas la partie polynomiale est la meme mais  le  log devient $|\log(\e)|^{1+p\beta}$ au lieu de $|\log(\e)|^{p\beta}$. Donc sauf erreur ca colle avec la borne inf :-)}
%\vinc{The choice of $N$ proportional to $\e^{-2}$ allows to compare these upper bounds with minimax lower bounds.} 
Theorem~\ref{rateslp} shows that our procedure achieves the optimal rate in the homogeneous case $r\geq p$, in the sparse case $r<p/(2s+1)$ and in the frontier case  $r=p/(2s+1)$. In the intermediate case $\frac{p}{2s+1}<r<p$, we have a logarithmic loss with exponent $\frac{s}{2s+1}(1-\frac2p)_+$ which is non-zero if and only if $p>2$. When $p\leq 2$, the rate is optimal.
To the best of our knowledge, all other adaptive  methods suffer from a logarithmic loss in at least one case, except the adaptive procedure proposed in \cite{Juditsky97} based on an involved combination of thresholding and Lepski-type procedures. In particular, \cite{DJKP2} have 
a logarithmic loss, both in homegeneous and intermediate cases, with power $\frac{s}{2s+1}$. It is also the case in  \cite{LepskiMammenSpokoiny} (for the study of the white noise model) or  \cite{KLP} (for the study of the multivariate white noise model) or \cite{GL14} (for the study of the density model).
%\color{black}

%%

%%%%%%%%%%%%%
% Je commente la suite %
%%%%%%%%%%%%%

%%%%%%%%%%%%%%%%%%%%%%%%%%
%%%%%%%%%%%%%%%%%%%%%%%%%%
\section{Nonparametric regression}\label{sec:regression}
In this section, we consider the following nonparametric regression model
\begin{equation}\label{model-NP}
X_i = f(t_i) +\sigma \eta_i,\qquad 1\leq i \leq n,
\end{equation}
with  $\eta_i$ some i.i.d. centered sub-Gaussian variables modelling the noise and $\sigma>0$ assumed to be known. In the previous expression, $f$ has support included into $[0,1]$ and the $t_i$'s are deterministic and equidistant : $t_i=i/n,$ $i=1,\ldots,n$. We assume in the sequel that $\log_2(n)$ is an integer. Our goal is to estimate the function $f$ by using observations $(X_i)_{i=1,\ldots,n}$. The risk of any estimate will be evaluated by using the $\L_p$-norm for some $1\leq p<\infty$. For this purpose, we shall use results of Section~\ref{sec:upperbounds} and in particular the weights \eqref{def-weights}  introduced in Section~\ref{sec:setting}.

In this setting, we consider a decomposition of the signal $f$, assumed to be squared-integrable, on a wavelet basis. The expansion of $f$ is then of the form:
\begin{equation}\label{decompowavelet}
f=\sum_k\langle f,\phi_k\rangle \phi_k+\sum_{j=0}^{+\infty}\sum_k\langle f,\psi_{jk}\rangle \psi_{jk},
\end{equation}
where $\phi_k$ is the translation of a father wavelet $\phi$ and $\psi_{jk}$ is the dilation and translation a mother wavelet $\psi$: for any $x$, we have:
$$\phi_k(x)=\phi(x-k),\quad  \psi_{jk}(x)=2^{j/2}\psi(2^jx-k).$$
The expansion~\eqref{decompowavelet} can be of course rewritten as
\begin{equation}\label{decompowavelet2}
f=\sum_{j=-1}^{+\infty}\sum_{k\in K_j}\langle f,\varphi_{jk}\rangle \varphi_{jk},
\end{equation}
with for any $(j,k)$, $\varphi_{jk}=\psi_{jk}$ if $j\geq 0$ and $\varphi_{jk}=\phi_k$ if $j=-1$. Considering the specific construction of Section~4 of \cite{CDV} to obtain an orthonormal basis of $\L^2[0,1]$, we can assume that the wavelets that generate the basis have a compact support which is then included into the interval $[A,B]$ for some $0<A<B<\infty$ and is $M+1$ times weakly differentiable,
%$M$ vanishing moments, 
where $M$ can be chosen by the practitioner. Finally, the construction of \cite{CDV} shows that  we can take
$$ K_j=\big\{0,1,2,\ldots, 2^j-1\big\},\quad j\geq 0$$
and $K_{-1}=\{0\}$. %For more details, we refer the reader to Section~7.5 of \cite{MR1162107}.
%\vinc{[En fait, j'arnaque un peu car dans la construction de \cite{CDV} l'ecriture $$\phi_k(x)=\phi(x-k),\quad  \psi_{jk}(x)=2^{j/2}\psi(2^jx-k)$$ne marche pas pour les ondelettes de bord. Cela ne concerne que $2M$ fonctions a chaque niveau $j$ donc cela ne change rien]}

Now, as in Section~\ref{sec:minimax},  we consider $\Lambda^{(N)}$ of size $N=2^{J+1}$, for some $J\geq 0$, with
$$\Lambda^{(N)}=\bigcup_{-1\leq j\leq J}\Lambda_j,\quad \Lambda_j=\Big\{\{j\}\times K_j\Big\}$$
%Now let $(\varphi_{\lambda})$ be an orthonormal wavelet basis. \textcolor{blue}{We assume that $\varphi$ is compactly supported with $M$ moments. In this wavelet framework}, $\lambda$ is of the form $\lambda=(j,k)$ with $j$ an integer and $k\in K_j$, where $K_j$ is a set of consecutive integers of size $2^j$ (up to a constant). We consider \textcolor{blue}{$\Lambda_N$ a finite set of indices (with cardinal  $N$ large enough, to be precised later)} 
and we set for any $(j,k)\in\Lambda^{(N)}$,
$$Y_{jk}=\frac1n\sum_{i=1}^nX_i \varphi_{jk} (t_i), \quad \theta_{jk}=\frac1n\sum_{i=1}^nf(t_i) \varphi_{jk} (t_i),\quad \xi_{jk}=\frac1{\sqrt{n}}\sum_{i=1}^n \eta_i\varphi_{jk} (t_i).$$
Setting
$$\e=\frac{\sigma}{\sqrt{n}},$$
we have:
$$Y_{jk}=\theta_{jk}+\e\xi_{jk},\quad (j,k)\in\Lambda^{(N)},$$ 
where the $\xi_{jk}$'s are centered. We obtain the model of Section~\ref{sec:setting}, except that the $\xi_{jk}$'s are not independent
and not identically distributed. Furthermore, our target is $f$ and not $\theta=(\theta_{jk})_{(j,k)\in\Lambda^{(N)}}$. However, we consider $\tilde\theta$ the PCO estimate of Section~\ref{sec:upperbounds} (except that me modify the strategy $I$ by taking
$|m_{L+l}|=\lfloor 2^{L+l} 2^{-lp/2}(l+1)^{-3p/2} \rfloor$ when $p>2$) and we set
$$\tilde f=\sum_{(j,k)\in\Lambda^{(N)}}\tilde\theta_{jk}\varphi_{jk}$$ to estimate $f$. 
We study the upper bound of the $\Lp$-risk of $\tilde f$ on the class of Besov spaces. We refer the reader to Section~9.2 of \cite{hkpt} for the definition of Besov spaces in terms of modulus of continuity. In the sequel, we use the characterization of Besov spaces through wavelet coefficients (see Corollary~9.1 of \cite{hkpt}): Under additional mild conditions on the wavelet functions $\phi$ and $\psi$, for $1\leq r,q<\infty$  and $0<s<M+1$, %where $M+1$ is the regularity parameter of the wavelet functions $\phi$ and $\psi$ for some integer $M$, 
a function $g\in \L_r$ belongs to ${\mathcal B}^s_{r,q}$ if and only if
$$\sum_{j\geq -1}2^{qj(s+\frac12-\frac1r)}\Big(\sum_{k\in K_j}\big|\langle g,\psi_{jk}\rangle\big|^r\Big)^{\frac{q}{r}}<\infty.$$
When $q=\infty$, this condition becomes
$$\sup_{j\geq -1}2^{j(s+\frac12-\frac1r)}\Big(\sum_{k\in K_j}\big|\langle g,\psi_{jk}\rangle\big|^r\Big)^{\frac{1}{r}}<\infty.$$
We shall use this sequential characterization to define the radius of a Besov ball, which allows us to use the setting and notations of Section~\ref{sec:minimax}. We obtain the following result. 
\begin{theorem}\label{ratesLp}
Let $1\leq p, r<\infty$ and $s>0$ such that $1/r<s< M+1$. We recall that $\e=\frac{\sigma}{\sqrt{n}}.$ Assume that $R$, the radius of the Besov ball, belongs to an interval $[R_0,R_1],$ with $0<R_0<R_1<+\infty$. We take the resolution level $N=2^{J+1}$ such that {$\e^{-2}\leq N \leq \e^{-\gamma},\gamma\geq 2$}. Then, we have:
$$\sup_{f\in\mathcal{B}_{r,\infty}^{s}(R)}\E\Big[\|\tilde f-f\|_{\Lp}^p\Big]\leq C\left\{
\begin{array}{ccl}
\e^{\frac{2ps}{1+2s}}  &   \text{if}  &r\geq p   \\
\e^{\frac{2ps}{1+2s}}|\log(\e)|^{\frac{s(p-2)_+}{(1+2s)}}  &  \text{if} &   \frac{p}{2s+1}<r<p\\
\e^{\frac{2ps}{2s+1}}|\log(\e)|^{\frac{ps }{2s+1}+1} &  \text{if} &   r=\frac{p}{2s+1}\\
 \big(\e^{2}|\log(\e)|\big)^{\frac{p(s-\frac{1}{r}+\frac{1}{p})}{2s+1-\frac2{r}}}  &  \text{if} &   r<\frac{p}{2s+1}
\end{array}
\right.$$
for $C$ a constant depending on $p$, $s$, $r$, $\gamma$, $R_0$, $R_1$ and the wavelet functions $\phi$ and $\psi$.
\end{theorem}
%Remarque sur le random design.
Similarly to Theorem~\ref{rateslp}, Theorem~\ref{ratesLp} shows the optimality of our estimation procedure for the homogeneous case $r\geq p$, the frontier case    $r=\frac{p}{2s+1}$, the sparse case  $r<\frac{p}{2s+1}$ and the intermediate case $\frac{p}{2s+1}<r<p$ when $p\leq 2$.  We refer the reader to \cite{nemirovski} for corresponding lower bounds, see also \cite{DJ98} and Theorem 4 of \cite{DJKP2}. 
%\textcolor{gray}{(Nemirovski fait du Sobolev mais tout le monde se refere a ca,   DJ98 Besov  mais uniquement $p=2$, pas mis GL14 car densite)} 
When $p>2$, for the intermediate case 
$\frac{p}{2s+1}<r<p$, we obtain the additional logarithmic term $|\log(\e)|^{\frac{s(p-2)}{(1+2s)}}$ similarly to Theorem~\ref{rateslp}.
 
To end this section, we give main arguments of the proof of Theorem~\ref{ratesLp}. The $\Lp$-risk of $\tilde f$ is deduced from the following control:
\begin{align}\label{Decomp-LP}
\E\Big[\|\tilde f-f\|_{\Lp}^p\Big]&\leq 2^{p-1}\left[\E\Bigg[\Big\|\sum_{(j,k)\in\Lambda^{(N)}}(\tilde\theta_{jk}-\theta_{jk})\varphi_{jk}\Big\|_{\Lp}^p\Bigg]+\Big\|\sum_{(j,k)\in\Lambda^{(N)}}\theta_{jk}\varphi_{jk}-f\Big\|_{\Lp}^p\right]
\end{align}
and the application of  Theorem~\ref{rateslp} to bound the first term. However, several additional technical arguments are needed and we have to tackle several problems:
\begin{enumerate}
\item Concentration inequalities on the $\xi_{jk}$'s, which are now not i.i.d., were essential in previous sections. So, the question is the following: Do the noise variables $\xi_{jk}$ have sufficiently nice concentration properties to apply Theorem~\ref{rateslp}? 
\item How can we control the second term of the right hand side, which corresponds to an approximation term?
\item Can we connect the first term of the right hand side of \eqref{Decomp-LP} to the $\ell_p(w)$-risk of $\tilde\theta$?
%\item How can we connect the coefficients $\theta_{jk}$ to the wavelet coefficients $\langle f,\varphi_{jk}\rangle$?
\end{enumerate}
To address the first issue, we establish in the next proposition, proved in Section~\ref{sec:proofnoiseok}, that the $\xi_{j,k}$'s satisfy a result similar to the result of 
Corollary \ref{coro:concentration}%Theorem~\ref{theo:concentration}
. We use that the $\eta_i$'s are i.i.d. centered sub-Gaussian random variables.
\begin{proposition}\label{prop:noise}
For any $j$ and for any $\mathcal{I}_j\subset K_j$, we set
  $$Z_j:= \sum_{k\in \mathcal{I}_j} |\xi_{jk}|^p.$$
There exist positive constants  $c_\varphi,$ $\sigma_p$ and $\kappa'_p$ only depending on $p$ and the compactly father and mother wavelets $\phi$ and $\psi$ such that for any $x\geq 1$,
$$\P\Big(Z_j \geq  \frac{3}{2}\sigma_p^p|{\mathcal I}_{j}|+\kappa'_p|{\mathcal I}_{j}|^{\big(1-\frac{p}{2}\big)_+}x^{\frac{p}{2}}\Big)\leq {c_\varphi}e^{-x}.$$
\end{proposition}
Regarding the approximation term, we can prove  the following result (see Section~\ref{sec:prooflemmaapprox} for the proof). 
%\begin{proposition}[Proposition 2 of \cite{MR1426459}]\label{PropRemTerm} We consider an orthogonal wavelet basis with $M$ moments.
%We assume that $f\in{\mathcal B}^s_{r,\infty}(R)$, with $0<s\leq M$ and $s>1/r$. Then, if $p\geq r$,
%$$\Big\|\sum_{\lambda\in\Lambda}\theta_\lambda\varphi_\lambda-f\Big\|_{\Lp}\lesssim R|\Lambda |^{-(s-1/r+1/p)}.$$
%\end{proposition}
\begin{lemma} \label{lemma:approx}
Assume that $f$ belongs to the Besov set ${\mathcal B}^s_{r,\infty}(R)$ with $1/r<s< M+1$. Let $\theta_{jk}=\frac1n\sum_{i=1}^nf(t_i) \varphi_{jk} (t_i)$.
 Then, if $N\geq \frac{\e^{-2}}{|\log(\e)|}$,
$$\Big\|\sum_{(j,k)\in\Lambda^{(N)}}\theta_{jk}\varphi_{jk}-f\Big\|_{\Lp}^p \leq C\left\{
\begin{array}{ccl}
%\e^{\frac{2ps}{1+2s}}    &   \text{if}  &r\geq p   \\
 R^p\e^{\frac{2ps}{1+2s}}   &  \text{if} &   r\geq \frac{p}{2s+1},\\
   R^p \big(\e^{2}|\log(\e)|\big)^{\frac{p(s-\frac{1}{r}+\frac{1}{p})}{2s+1-\frac2{r}}} & \text{if} &   r<\frac{p}{2s+1},
\end{array}
\right.$$
for $C$ a constant depending on $\phi$, $\psi$, $s$, $r$ and $p$. 
\end{lemma}
Finally, to address the third issue, we wish to compare the $\Lp$-risk of $\tilde f$ (or rather the first terms of its decomposition) with the $\ell_p(w)$-risk of $\tilde\theta$ when weights are those of Section~\ref{sec:setting} (see \eqref{def-weights}). 
We first state the following lemma whose proof can be found in Section~\ref{sec:lp-besov}. 
\begin{lemma}\label{lp-besov}
Let $1\leq p<\infty$. For any function $g$, we have:
$$\|g\|_{\L_p}\leq C\|g\|_{{\mathcal B}^0_{p,p\wedge 2}},$$
for $C$ a constant and with
$$\|g\|_{{\mathcal B}^0_{p,p\wedge 2}}^{p\wedge 2}:=\sum_{j\geq -1}2^{j(p\wedge 2)(\frac12-\frac1p)}\Big(\sum_{k\in K_j}\big|\langle g,\varphi_{jk}\rangle\big|^p\Big)^{\frac{p\wedge 2}{p}}.$$
\end{lemma}
Now, we naturally distinguish two cases.
\begin{description}
\item[-] If $1\leq p\leq 2$, using Lemma~\ref{lp-besov}, we have:
\begin{eqnarray}\label{lp-Lp<2}
\hspace{-1.5cm}\E\Bigg[\Big\|\sum_{(j,k)\in\Lambda^{(N)}}(\tilde\theta_{jk}-\theta_{jk})\varphi_{jk}\Big\|_{\Lp}^p\Bigg]&
\lesssim&\E\Bigg[\Big\|\sum_{(j,k)\in\Lambda^{(N)}}(\tilde\theta_{jk}-\theta_{jk})\varphi_{jk}\Big\|_{{\mathcal B}^0_{p,p}}^p\Bigg]\nonumber\\
&\lesssim&\E\Big[ \sum_{j=-1}^J2^{j(\frac{p}{2}-1)}\sum_{k\in K_j}\big|\tilde\theta_{jk}-\theta_{jk}\big|^p\Big]\leq\E\|\tilde \theta-\theta\|_p^p,
\end{eqnarray}
with the $\|\cdot\|_p$-norm is defined in \eqref{def-normlp} with weights defined in \eqref{def-weights}.
Therefore, results of Section~\ref{sec:upperbounds} can be applied.
\item[-] If $2< p<\infty$, still using Lemma~\ref{lp-besov}, we have:
\begin{eqnarray}\label{lp-Lp>2}
\hspace{-1.5cm}\E\Bigg[\Big\|\sum_{(j,k)\in\Lambda^{(N)}}(\tilde\theta_{jk}-\theta_{jk})\varphi_{jk}\Big\|_{\Lp}^p\Bigg]&
\lesssim&\E\Bigg[\Big\|\sum_{(j,k)\in\Lambda^{(N)}}(\tilde\theta_{jk}-\theta_{jk})\varphi_{jk}\Big\|_{{\mathcal B}^0_{p,2}}^p\Bigg]\nonumber\\
&\lesssim&\E\left[\Bigg( \sum_{j= -1}^J2^{j(1-\frac{2}{p})}\Big(\sum_{k\in K_j}\big|\tilde\theta_{jk}-\theta_{jk}\big|^p\Big)^{\frac{2}{p}}\Bigg)^{\frac{p}{2}}\right]\nonumber\\
&\lesssim&\left[ \sum_{j= -1}^J\Bigg(\E\Big[2^{j(\frac{p}{2}-1)}\sum_{k\in K_j}\big|\tilde\theta_{jk}-\theta_{jk}\big|^p\Big]\Bigg)^{\frac{2}{p}}\right]^{\frac{p}{2}},
\end{eqnarray}
by using the generalized Minkowski inequality. Since we cannot insert the sum in $j$ before taking the power $2/p$, the control of the $\L_p$-risk by the  $\ell_p(w)$-one is not immediate. This last issue is addressed in Section~\ref{sec:proofTheoLp}.
\end{description}
These arguments and technical complements of Section~\ref{sec:proofTheoLp} allow to prove Theorem~\ref{ratesLp}.
%%%%%%%%%%%%%%%%%%%%
%%%%%%%%%%%%%%%%%%

%

%%%%%%%%%%%%%%%%%%%%%%
%%%%%%%%%%%%%%%%%%%%%%
\section{Proofs} \label{sec:proofs}
%%%%%%%%%%%%%%%%%%%%%%
\subsection{Proofs of oracle results of Section~\ref{sec:oracle}}
%%%%
\subsubsection{Technical lemmas}\label{technicallemmas}
\begin{lemma}\label{lemma:p1} Let $p\geq 1 $.  Let $K>0$ and $a,b$ two  reals such that $|a|\geq K |b| $. Then, for any $\alpha>0$, 
$$\left||a+b|^p - |a|^p-\alpha|b|^p\right|\leq C_{1p}(\alpha,K) |a|^p$$ 
and 
$$\left||a+b|^p - \alpha|a|^p-|b|^p\right|\leq C_{2p}(\alpha,K) |a|^p$$
where $C_{1p}(\alpha,K)$ and $C_{2p}(\alpha,K)$ are positive constants depending on $p,\alpha, K$ such that \\$\lim_{K\to \infty}C_{1p}(\alpha,K)=0$.
\end{lemma}

\begin{proof}   : 
The case $a=0$ is obvious, so we assume $a\neq 0$. Denoting $x=b/a$, it is sufficient to study the function 
$g_{\alpha}(x)=|1+x|^p-1-\alpha|x|^p$ on $[-K^{-1}, K^{-1}]$. Since it is a continuous function on a compact set,  it is bounded and we denote by $C_{1p}(\alpha,K)$ the maximum of $|g_{\alpha}|$ on $[-K^{-1}, K^{-1}]$.
Moreover $g_{\alpha}(0)=0$ so that $\lim_{K\to \infty}C_{1p}(\alpha,K)=0$. In the same way, the continuous function $|1+x|^p-\alpha-|x|^p$ is bounded on $[-K^{-1}, K^{-1}]$ and we denote by $C_{2p}(\alpha,K)$ its bound. 
\end{proof}

\begin{lemma}\label{lemma:p}
For any $p>1$, for any $\alpha>0$, there exists $C(\alpha,p)$ such that for any $x>0$ and $y>0$,
$$(x+y)^p\leq (1+\alpha)x^p+C(\alpha,p)y^p.$$
We have, when $\alpha\to 0$,
$$C(\alpha,p)\sim \Big(\frac{\alpha}{p-1}\Big)^{1-p}\to+\infty.$$
\end{lemma}
\begin{proof}
We prove that
$$C(\alpha,p)=\frac{1}{\left(1-(1+\alpha)^{-\frac{1}{p-1}}\right)^{p-1}}$$
by studying the function
$t\longmapsto C(\alpha,p)+(1+\alpha)t^p-(t+1)^p.$
\end{proof}
%\end{remark}
%%%%%%
\subsubsection{Proof of Theorem~\ref{theo:oracle1}}
\label{prooforacle}
Let us denote 
%\sout{$J(m)=\sum_{\la \notin m} w_\lambda\big(|\theta_{\la}|^p +\alpha\e^p|\xi_{\la}|^p-|Y_{\la}|^p\big).$}
$$J(m)=\sum_{\la \in \Lambda\setminus m} w_\lambda|\theta_{\la}|^p +\sum_{\la \in \Lambda^{(N)}\setminus m} w_\lambda\big(\alpha\e^p|\xi_{\la}|^p-|Y_{\la}|^p\big).$$
Setting
$$C_1=\e^p\sum_{\la\in \Lambda^{(N)}} w_\lambda|\xi_{\la}|^p\quad\mbox{and}\quad C_2=\sum_{\la\in \Lambda^{(N)}} w_\lambda|Y_{\la}|^p,$$
we can write
\begin{align*}
J(m)&= B_p(m)+\alpha(C_1-V_p(m))-(C_2-\sum_{\la \in m}w_\lambda |Y_{\la}|^p)\\
&= \|\hat \theta^{(m)} - \theta \|_p^p-(1+\alpha)V_p(m)+\sum_{\la \in m} w_\lambda|Y_{\la}|^p+\alpha C_1-C_2.
\end{align*}
But the definition of $\hat m$ gives
$$-\sum_{\la \in \hat m}w_\lambda|Y_{\la}|^p+\pen(\hat m)\leq -\sum_{\la \in m}w_\lambda|Y_{\la}|^p+\pen(m).$$
Then, since $C_1$ and $C_2$ do not depend on $m$,
$$\|\hat \theta^{(\hat m)} - \theta \|_p^p-(1+\alpha)V_p(\hat m)-J(\hat m)+\pen(\hat m)\leq \|\hat \theta^{(m)} - \theta \|_p^p-(1+\alpha)V_p(m)-J(m)+\pen( m).$$
Thus $$\|\hat \theta^{(\hat m)} - \theta \|_p^p\leq \|\hat \theta^{(m)} - \theta \|_p^p+
\Big[(1+\alpha)V_p(\hat m)-\pen(\hat m)\Big]
- \Big[(1+\alpha)V_p(m)-\pen(m)\Big]+J(\hat m)-J(m)$$
and it is sufficient to control $J(\hat m)-J(m)$.
Let us denote 
$$S_{\la}=w_\lambda\big(|\theta_{\la}|^p+\alpha\e^p|\xi_{\la}|^p-|Y_{\la}|^p\big)$$
so that, with $m^c=\Lambda^{(N)}\setminus m$ and $\hat m^c=\Lambda^{(N)}\setminus \hat m$,
\begin{align*}
J(\hat m)-J(m)&=\sum_{\la \in \hat m^c} S_{\la}-\sum_{\la \in  m^c} S_{\la}\\
&=\left(\sum_{\la \in \hat m^c\cap m } S_{\la}+\sum_{\la \in \hat m^c\cap m^c } S_{\la}\right)
-\left(\sum_{\la \in  m^c \cap \hat m^c} S_{\la}+\sum_{\la \in  m^c \cap \hat m} S_{\la}\right)\\
&=\sum_{\la \in \hat m^c\cap m } S_{\la}-\sum_{\la \in  m^c \cap \hat m} S_{\la}.
\end{align*}
\paragraph{Case $p>1$.} We first deal with the second term.
We have:
\begin{align*}
-\sum_{\la \in  m^c \cap \hat m} S_{\la}&=\sum_{\la \in  m^c \cap \hat m} w_\lambda\Big[|Y_{\la}|^p-|\theta_{\la}|^p-\alpha\e^p|\xi_{\la}|^p\Big]\\
&=\sum_{\la \in  m^c \cap \hat m} w_\lambda\Big[|\theta_{\la}+\e\xi_{\la}|^p-|\theta_{\la}|^p-\alpha\e^p|\xi_{\la}|^p\Big]\\
&\leq \sum_{\la \in  m^c \cap \hat m} w_\lambda\Big[\Big(1+\frac{\alpha}{2}\Big)|\e\xi_{\la}|^p+C(\alpha/2,p)|\theta_{\la}|^p-|\theta_{\la}|^p-\alpha\e^p|\xi_{\la}|^p\Big],
\end{align*}
by using notations of Lemma~\ref{lemma:p}.
Finally, 
\begin{align*}
-\sum_{\la \in  m^c \cap \hat m} S_{\la}&\leq\Big(1-\frac{\alpha}{2}\Big)\sum_{\la \in  \hat m}w_\lambda|\e\xi_{\la}|^p+\Big(C(\alpha/2,p)-1\Big)\sum_{\la \in  m^c}w_\lambda|\theta_{\la}|^p\\
&\leq\Big(1-\frac{\alpha}{2}\Big)V_p(\hat m)+\Big(C(\alpha/2,p)-1\Big)B_p(m).
\end{align*}
We now deal with the first term, namely $\sum_{\la \in \hat m^c\cap m } S_{\la}$.  {Let $K>0$.}
\begin{enumerate}
\item We assume that $|\theta_{\la}|\geq K\e|\xi_{\la}|$. Then, applying Lemma~\ref{lemma:p1}, we have
\begin{align*}
|S_{\la}|&=w_\lambda\Big||\theta_{\la}|^p +\alpha\e^p|\xi_{\la}|^p-|Y_{\la}|^p\Big|\\
&\leq {C_{1p}(\alpha,K)} w_\lambda|\theta_{\la}|^p\\
%\Big[p(1+K^{-1})^{p-1}K^{-1}+\alpha K^{-p}\Big]w_\lambda\theta_{\la}|^p\\
&\leq\Big(1-\frac{\alpha}{2}\Big)w_\lambda|\theta_{\la}|^p,
\end{align*}
choosing $K\equiv K_{p,\alpha}$ large enough. 
%Observe that $K$ is a constant when $\alpha\to 0$.
\item We assume that $|\theta_{\la}|< K\e|\xi_{\la}|$. Then
\begin{align*}
|S_{\la}|&\leq w_\lambda\Big||\theta_{\la}|^p +\alpha\e^p|\xi_{\la}|^p-|Y_{\la}|^p\Big|\\
&\leq {C_{2p}(\alpha,K^{-1})}
w_\lambda|\e\xi_{\la}|^p.
\end{align*}
using again Lemma~\ref{lemma:p1} with $|\e\xi_{\la}|\geq K^{-1}|\theta_{\la}|$. %\vinc{Ici, $K^{-1}<1$}
\end{enumerate}
Finally,
\begin{align*}
\sum_{\la \in \hat m^c\cap m } S_{\la}&\leq\Big(1-\frac{\alpha}{2}\Big)\sum_{\la \in  \hat m^c}w_\lambda|\theta_{\la}|^p+{C_{2p}(\alpha,K_{p,\alpha}^{-1})}\sum_{\la \in  m}w_\lambda|\e\xi_{\la}|^p\\
&\leq\Big(1-\frac{\alpha}{2}\Big)B_p(\hat m)+{C_{2p}(\alpha,K_{p,\alpha}^{-1})}V_p(m).
\end{align*}
We obtain
\begin{align*}
J(\hat m)-J(m)&\leq\Big(1-\frac{\alpha}{2}\Big) \|\hat \theta^{(\hat m)} - \theta \|_p^p+C'(\alpha,p) \|\hat \theta^{(m)} - \theta \|_p^p
\end{align*}
with $C'(\alpha,p)=\max\left(C\left(\frac{\alpha}{2},p\right)-1 , {C_{2p}(\alpha,K_{p,\alpha}^{-1})}\right)$.
Thus
 $$\frac{\alpha}{2}\|\hat \theta^{(\hat m)} - \theta \|_p^p\leq (1+C'(\alpha,p))\|\hat \theta^{(m)} - \theta \|_p^p+
\left[(1+\alpha)V_p(\hat m)-\pen(\hat m)\right]
- \left[(1+\alpha)V_p(m)-\pen(m)\right]$$
and
$$\|\hat \theta^{(\hat m)} - \theta \|_p^p\leq \frac2{\alpha}(1+C'(\alpha,p))\|\hat \theta^{(m)} - \theta \|_p^p+
\frac2{\alpha}\left[(1+\alpha)V_p(\hat m)-\pen(\hat m)\right]
- \frac2{\alpha}\left[(1+\alpha)V_p(m)-\pen(m)\right].$$
Thus the result is proved with 
\begin{equation}\label{valeurdeM}
M_{p,\alpha}=\frac{2}{\alpha}\left(1+\max\left(C\left(\frac{\alpha}{2},p\right)-1 , {C_{2p}(\alpha,K_{p,\alpha}^{-1})}\right)\right)
\end{equation}
where $K_{p,\alpha}$ is such that $C_{1p}(\alpha,K_{p,\alpha})\leq 1-\alpha/2$ and 
where $C_{1p}, C_{2p}, C(.,p)$  are defined in Lemmas \ref{lemma:p1} and \ref{lemma:p}.

%\vinc{Je pense que si second terme - troisieme terme $<0$ on ne peut pas prendre $M_{p,\alpha}$ ainsi}
%
%\color{blue} Les calculs sous $R$ donnent 
%
%\begin{itemize}
%\item $p=1.1$ : min atteint en $\alpha\approx 1.9$ et vaut $M\approx 5$ 
%\item $p=1.5$ : min atteint en $\alpha\approx 1.4$ et vaut $M\approx 9$
%\item $p=2$ : min atteint en $\alpha\approx 1$ et vaut $M\approx 18$
%\item $p=3$ : min atteint en $\alpha\approx 0.8$ et vaut $M\approx 304$
%\item $p=5$ : min atteint en $\alpha\approx 0.7$ et vaut $M\approx 138 000$
%\end{itemize}
%
%Si on prend $\alpha=1$, les valeurs de $M$ restent alors du meme ordre de grandeur que le $\min(M)$, on obtient en effet pour $M$ dans les cas precedents : 6.5/9.5/18/330/15000.

%\color{black}
\paragraph{Case $p=1$.}  In this case 
$$S_{\la}=w_\lambda \big(|\theta_{\la}| +\alpha\e|\xi_{\la}|-|Y_{\la}|\big)\geq (\alpha -1) \e w_\lambda |\xi_{\la}| .$$%\Rightarrow J(m)\geq 0$$
%Then $J(\hat m)-J(m)\leq J(\hat m)$.
%
Note also that 
$$S_{\la}=w_\lambda \big( |\theta_{\la}|+\alpha\e|\xi_{\la}|-|\theta_{\la}+\e\xi_{\la}|\big)
\leq w_\lambda (\alpha+1)\e|\xi_{\la}| .$$
Then, if $\alpha\geq 1$,
\begin{align*}
J(\hat m)-J(m)&=\sum_{\la \in \hat m^c} S_{\la}-\sum_{\la \in  m^c} S_{\la}
 \leq \sum_{\la \in \Lambda^{(N)}} S_{\la}-\sum_{\la \in  m^c} S_{\la}=\sum_{\la \in m} S_{\la}\\
 &\leq \sum_{\la \in m}  (\alpha + 1)  w_{\la}\e|\xi_{\la}| \leq  (\alpha +1) \|\hat \theta^{(m)} - \theta \|_1
\end{align*}
and then 
$$\|\tilde \theta-\theta\|_p^p\leq (2+\alpha)\|\hat \theta^{(m)}-\theta\|_p^p +
\Big[(1+\alpha)V_p(\hat m)-\pen(\hat m)\Big]- \Big[(1+\alpha)V_p(m)-\pen(m)\Big].$$
Now, if $0<\alpha<1$,
\begin{align*}
J(\hat m)-J(m)&=\sum_{\la \in \hat m^c\cap m } S_{\la}-\sum_{\la \in  m^c \cap \hat m} S_{\la}\\
&\leq \sum_{\la \in \hat m^c\cap m }  (\alpha + 1)  w_{\la}\e|\xi_{\la}|
+\sum_{\la \in  m^c \cap \hat m}   (1-\alpha )  w_{\la}\e|\xi_{\la}|\\
& \leq  (\alpha +1) \|\hat \theta^{(m)} - \theta \|_1+(1-\alpha)\|\hat \theta^{(\hat m)} - \theta \|_1
\end{align*}
and then 
$$\alpha\|\tilde \theta-\theta\|_p^p\leq (2+\alpha)\|\hat \theta^{(m)}-\theta\|_p^p +
\Big[(1+\alpha)V_p(\hat m)-\pen(\hat m)\Big]- \Big[(1+\alpha)V_p(m)-\pen(m)\Big].$$
Here $M_{p,\alpha}=M_{1,\alpha}=\max(2+\alpha,1+2/\alpha)$.

%vinc{Meme souci pour le choix de $M_{1,\alpha}$.}

%\color{blue} Ici $M$ est minimum quand $\alpha=1$ et vaut alors 3.
%\color{black}
%%%%%%%%%%%%%%%%%%%%%%

 %%%%
\subsubsection{Proof of Theorem~\ref{theo:oracleesp}}
\label{preuveoraclesp2}
We only consider the case $p>1$, the case $p=1$ is similar.\\
Starting from 
Theorem~\ref{theo:oracle1}, we have:
\begin{align*}
\|\tilde \theta-\theta\|_p^p&\leq  M_{p} \|\hat \theta^{(m)}-\theta\|_p^p +2
\Big[2V_p(\hat m)-\pen(\hat m)\Big]
- 2\Big[2V_p(m)-\pen(m)\Big]\\
&\leq M_{p}  \|\hat \theta^{(m)}-\theta\|_p^p +2\pen(m)
+2\Big[2V_p(\hat m)-\pen(\hat m)\Big]\\
&\leq  M_{p} \|\hat \theta^{(m)}-\theta\|_p^p +2\pen(m)
+4\varepsilon^p\sum_{j\in\J}\omega_j\sum_{m_j\in \M_j}\Big[Z( m_j)-\p_j(m_j)\Big]_+,
\end{align*}
where $Z(m_j)=\sum_{\la\in m_j }|\xi_{\la}|^p$, and 
$$\p_j(m_j)=\frac32|m_j|\sigma_p^p+\kappa_p2^{\frac{(p-2)_+}2}|m_j|^{\big(1-\frac{p}{2}\big)_+}x_{m_j}^{p/2}.
$$
Note that  $[Z(\emptyset)-\p_j(\emptyset)]_+=[-\p_j(\emptyset)]_+=0$,  so we have:
\begin{align*}\|\tilde \theta-\theta\|_p^p
&\leq   M_{p}\|\hat \theta^{(m)}-\theta\|_p^p +2\pen(m)
+4\varepsilon^p\sum_{j\in\J}\omega_j\sum_{m_j\in \M_j,m_j\neq \emptyset}\Big[Z( m_j)-\p_j(m_j)\Big]_+.
\end{align*}
Taking the expectation yields
\begin{align*}
\E\|\tilde \theta-\theta\|_p^p
&\leq  M_{p}\E \|\hat \theta^{(m)}-\theta\|_p^p +2\pen(m)
+4\varepsilon^p R,
\end{align*}
with 
\begin{align*}
R=\sum_{j\in\J}\omega_j\sum_{\substack{m_j\in \M_j\\m_j\neq\emptyset}}\E\Big[Z( m_j)-\p_j( m_j)\Big]_{+}=\sum_{j\in\J}\omega_j\sum_{\substack{m_j\in \M_j\\m_j\neq\emptyset}}\int_{0}^\infty\P\Big(Z( m_j)-\p_j( m_j)>u\Big)du.
\end{align*}
It remains to control the term $R$.  For the next computation, we denote $C_{j}=\kappa_p|m_j|^{\left(1-p/2\right)_+}$.
Using the change of variable 
$u=C_{j}2^{\frac{(p-2)_+}2}v^{p/2}$, we have:
\begin{align*}
&\hspace{-1cm}\int_{0}^\infty\P\Big(Z( m_j)-\p_j( m_j)>u\Big)du\\
=&\int_{0}^\infty\P\Big(Z( m_j)-\frac32\sigma_p^p|m_j|-C_{j}2^{\frac{(p-2)_+}2}x_{m_
j}^{\frac{p}{2}}>u\Big) du
\\
=&\int_{0}^\infty\P\Big(Z( m_j)-\frac32\sigma_p^p|m_j|-C_{j}2^{\frac{(p-2)_+}2}x_{m_
j}^{\frac{p}{2}}>C_{j}2^{\frac{(p-2)_+}2}v^{\frac{p}{2}}\Big)
\left[C_{j}2^{\frac{(p-2)_+}2}\frac{p}2 v^{\frac{p}{2}-1}\right]dv
\\
\leq &\int_{0}^\infty\P\Big(Z( m_j)>\frac32\sigma_p^p|m_j|+ C_{j}(x_{m_j}+v)^{\frac{p}{2}}\Big)
\left[C_{j}2^{\frac{(p-2)_+}2}\frac{p}2 v^{\frac{p}{2}-1}\right]dv
\end{align*}
since  $2^{\frac{(p-2)_+}2}(a^{p/2}+b^{p/2})\geq (a+b)^{p/2}.$
Corollary~\ref{coro:concentration} gives
\begin{align*}
\int_{0}^\infty\P\Big(Z( m_j)-\p_j(m_j)>u\Big)du
\leq  & 2\int_{0}^\infty e^{-(x_{m_j}+v)}
\left[C_{j}2^{\frac{(p-2)_+}2}\frac{p}2 v^{\frac{p}{2}-1}\right]dv\\
%\leq & 2e^{-x_m}
%\left[C_{j}2^{\frac{(p-2)_+}2}\frac{p}2\int_{0}^\infty e^{-v} v^{\frac{p}{2}-1}dv\right]\\
\leq & C(p)C_{j}e^{-x_{m_j}}=C(p)\kappa_p|m_j|^{\left(1-p/2\right)_+}e^{-x_{m_j}},
\end{align*}
for $C(p)$ a constant only depending on $p$. Finally, 
$$R\leq C(p)\kappa_p\sum_{j\in\J}\omega_j\sum_{m_j\in \M_j,m_j\neq \emptyset}|m_j|^{\left(1-p/2\right)_+}e^{-x_{m_j}}.$$
%%%
\subsubsection{Proof of Theorem~\ref{theo:oracleesp2}}
\label{preuveoraclesp3}
We only consider the case $p>1$, the case $p=1$ is similar.\\
We denote for $q> 1$,
$$\Omega(N,q):=\bigcap_{\lambda\in\Lambda^{(N)}}\left\{|\xi_\lambda|\leq\sqrt{2q\log N}\right\}.$$
Recall that we assume $\P(|\xi_\lambda|>t)\leq 2e^{-t^2/2}$ so
we have $\P(\Omega(N,q))\geq 1-2N^{1-q}.$ 
%\vinc{$\P(\Omega(N,q))\geq 1-2N^{1-q}$ pour faire comme Vershynin avec $K_1=1$ ?} 
Now,
\begin{align*}
\E\Big[\|\tilde \theta-\theta\|_p^p\Big]&=\E\Big[\|\tilde \theta-\theta\|_p^p\1_{\Omega(N,q)}\Big]+\E\Big[\|\tilde \theta-\theta\|_p^p\1_{\Omega(N,q)^c}\Big].
\end{align*}
First,
\begin{align*}
\E\Big[\|\tilde \theta-\theta\|_p^p\1_{\Omega(N,q)^c}\Big]&=\E\Big[\sum_{\la \notin \hat m}w_\lambda | \theta_{\la}|^p\1_{\Omega(N,q)^c}\Big]+\e^p\E\Big[\sum_{\la \in \hat m}w_\lambda |\xi_{\la}|^p\1_{\Omega(N,q)^c}\Big]\\
&\leq 2N^{1-q}\|\theta\|_p^p+ \e^p\sqrt{2}N^{(1-q)/2}\sigma_{2p}^p\sum_{\lambda\in\Lambda}w_\lambda.
\end{align*}
Then, starting from  Theorem~\ref{theo:oracle1}, we first have on $\Omega(N,q)$, for any $m\in\M$,
\begin{align*}
\|\tilde \theta-\theta\|_p^p&\leq M_p\|\hat \theta^{(m)}-\theta\|_p^p +2
\Big[2V_p(\hat m)-\pen(\hat m)\Big]
- 2\Big[2V_p(m)-\pen(m)\Big]\\
&\leq  M_p \|\hat \theta^{(m)}-\theta\|_p^p +2\pen(m)
+2\Big[2V_p(\hat m)-\pen(\hat m)\Big]\\
&\leq   M_p\|\hat \theta^{(m)}-\theta\|_p^p +2\pen(m)
+4\varepsilon^p\sum_{j\in\J}\omega_j\sum_{m_j\in \M_j}\Big[Z( m_j)-P_j( m_j)\Big]_+,
\end{align*}
where $Z(m_j)=\sum_{\la\in m_j }|\xi_{\la}|^p$, and 
$$P_j(m_j)=\min\Big(\p_j^1(m_j);\p_j^2(m_j)\Big)$$
with 
%$$\p_j^1(m_j)=\frac32\sigma_p^p|m_j|+\kappa_p2^{\frac{(p-2)_+}2}|m_j|^{\big(1-\frac{p}{2}\big)_+}x_{m_j}^{\frac{p}{2}},\quad \p_j^2(m_j)=(2q\log N)^{\frac{p}{2}-1}\left(\frac32\sigma_p^p|m_j|+\kappa_px_{m_j}\right).$$
$$\p_j^1(m_j)=\p_j(m_j),\quad \p_j^2(m_j)=(2q\log N)^{\frac{p}{2}-1}\p_j^\#(m_j).$$
Then, taking the expectation 
\begin{align*}
\E\Big[\|\tilde \theta-\theta\|_p^p\1_{\Omega(N,q)}\Big]
&\leq M_{p} \E \Big[\|\hat \theta^{(m)}-\theta\|_p^p\1_{\Omega(N,q)}\Big] +2\pen(m)
+4\varepsilon^p R
\end{align*}
with 
$$R=\sum_{j\in\J}\omega_j\sum_{\substack{m_j\in \M_j\\m_j\neq\emptyset}}\E\Big[\big(Z( m_j)-P_j( m_j)\big)_{+}\1_{\Omega(N,q)}\Big].$$
It remains to control the term $R$. We have:
\begin{align*}
R&=\sum_{j\in\J}\omega_j\sum_{\substack{m_j\in \M_j\\m_j\neq\emptyset}}\int_{0}^\infty\P\Big(Z( m_j)\1_{\Omega(N,q)}-\min\big(\p_j^1(m_j);\p_j^2(m_j)\big)\1_{\Omega(N,q)}>u\Big)du\\
&\leq\sum_{j\in\J}\omega_j\sum_{\substack{m_j\in \M_j\\m_j\neq\emptyset}}\Bigg[\int_{0}^\infty\P\Big(Z( m_j)-\p_j^1(m_j)>u\Big)du\\
&\hspace{4cm}+\int_{0}^\infty\P\Big(\big\{Z( m_j)-\p_j^2(m_j)>u\big\}\cap \Omega(N,q)\Big)du\Bigg].
\end{align*}
With the same computation as in the proof of Theorem~\ref{theo:oracleesp}, we have 
\begin{align*}
\int_{0}^\infty\P\Big(Z( m_j)-\p_j^1(m_j)>u\Big)du
\leq & C(p)|m_j|^{\left(1-p/2\right)_+}e^{-x_{m_j}}.
\end{align*}
On $\Omega(N,q)$, for $p\geq 2$,
\begin{align*}
Z( m_j)-\p_j^2(m_j)&=\sum_{\la\in m_j }|\xi_{\la}|^p-(2q\log N)^{\frac{p}{2}-1}\p_j^\#(m_j)\\
&\leq \big(\sqrt{2q\log N}\big)^{p-2}\sum_{\la\in m_j }\xi_{\la}^2-(2q\log N)^{\frac{p}{2}-1}\Big(\frac32\sigma_2^2|m_j|+\kappa_2x_{m_j}\Big)\\
&\leq (2q\log N)^{\frac{p}{2}-1}\Big(\sum_{\la\in m_j }\xi_{\la}^2-\Big(\frac32\sigma_2^2|m_j|+\kappa_2x_{m_j}\Big)\Big).
%\\
%&\leq (2q\log N)^{\frac{p}{2}-1}\Big(\sum_{\la\in m_j }\xi_{\la}^2-\p_j^1(m_j)\Big),
\end{align*}
%where we plug the expression of $\p_j^1(m_j)$ with $p=2$.
Therefore,
\begin{align*}
\int_{0}^\infty\P\Big(\big\{Z( m_j)-\p_j^2(m_j)>u\big\}\cap \Omega(N,q)\Big)du\
\leq  & (2q\log N)^{\frac{p}{2}-1}C(2)e^{-x_{m_j}}.
\end{align*}
We obtain, for $p\geq 2$, 
$$R\leq \max\big(C(p);C(2)\big)\sum_{j\in\J}\omega_j\sum_{\substack{m_j\in \M_j\\m_j\neq\emptyset}}\Big[1+(2q\log N)^{\frac{p}{2}-1}\Big]e^{-x_{m_j}}.$$
Finally, for $\breve M_{p,q}$ a constant only depending on $p$ and $q$, we have:
\begin{align*}
\E\Big[\|\tilde \theta-\theta\|_p^p\Big]&\leq  M_{p}\E \Big[\|\hat \theta^{(m)}-\theta\|_p^p\Big] +2\pen(m)+\breve M_{p,q}\Big(N^{1-q}\|\theta\|_p^p +\varepsilon^p R(\M)\Big),
\end{align*}
with
$$R(\M)=N^{(1-q)/2}\sum_{\lambda\in\Lambda}w_\lambda+ (\log N)^{\frac{p}{2}-1}\sum_{j\in\J}\omega_j\sum_{\substack{m_j\in \M_j\\m_j\neq\emptyset}}e^{-x_{m_j}}.$$

%%%%%%%%%%%%%%%%%%%%%%%%
\subsection{Proof of Theorem~\ref{rateslp}}\label{proofupper}
%In this proof, we use the following notation: we denote by $u_\e\lesssim v_\e$ when there exists $0 < A < \infty$ such that $ u_\e \leq A v_\e$ for all $\e>0$.
We first prove that $\|\theta\|_p<\infty$.
\begin{lemma}\label{normetheta}
Assume that $\theta\in\mathcal{B}_{r,\infty}^{s}(R)$ with $s>\frac{1}{r}$. Then, there exists $C$, only depending on $s$, $r$ and $p$, such that
\[
\|\theta\|_p^p\leq CR^p.
\]
\end{lemma} 
\begin{proof}
Let $j\geq -1$ be fixed. Assume first that $p\geq r$. Then
\begin{align*}
\Bigg(\sum_{k\in K_j}|\theta_{jk}|^p \Bigg)^{\frac1p}&\leq \Bigg(\sum_{k\in K_j}|\theta_{jk}|^r \Bigg)^{\frac1r}\leq R2^{-j(s+\frac12-\frac1r)}.
\end{align*}
Therefore,
\begin{align*}
2^{j\left(\frac{p}{2}-1\right)}\sum_{k\in K_j}|\theta_{jk}|^p&\leq R^p2^{j(\frac{p}{2}-1-ps-\frac{p}{2}+\frac{p}{r})}\leq R^p2^{j(-1-ps+\frac{p}{r})}\leq R^p2^{-j},
\end{align*}
since $s>\frac{1}{r}$. Now, assume that $p< r$. H\"{o}lder's inequality implies that
\[
\sum_{k\in K_j}|\theta_{jk}|^p\leq \Bigg(\sum_{k\in K_j}|\theta_{jk}|^r \Bigg)^{\frac{p}{r}}2^{j\left(1-\frac{p}{r}\right)}
\]
and
\begin{align*}
2^{j\left(\frac{p}{2}-1\right)}\sum_{k\in K_j}|\theta_{jk}|^p&\leq R^p2^{j(\frac{p}{2}-1-ps-\frac{p}{2}+\frac{p}{r}+1-\frac{p}{r})}\leq R^p2^{-jsp}.
\end{align*}
Finally, in both cases,
\[
\|\theta\|_p^p=\sum_{j=-1}^{+\infty}2^{j\left(\frac{p}{2}-1\right)}\sum_{k\in K_j} |\theta_{jk}|^p\leq CR^p,
\]
with $C$ only depending on $s$, $r$ and $p$.
\end{proof}

We distinguish the cases $p\leq 2$ and $p> 2$. When
 $p\leq 2$, 
we apply Theorem~\ref{theo:oracleesp} to the collection $\frak{M}$ and penalty $\frak{pen}$ and we obtain
\begin{equation}\label{RMM}
\E\|\hat \theta^{(\hat m^{\hat a})}-\theta\|_p^p\leq 
\tilde M_{p}
\inf_{(m,a)\in{\mathfrak M}}\left\{\E\|\hat \theta^{(m)}-\theta\|_p^p +\frak{pen}(m,a)\right\}+
\breve M_{p}\varepsilon^p R(\frak{M}).
\end{equation}
 When $p> 2$, we apply Theorem~\ref{theo:oracleesp2} to the collection $\frak{M}$ and penalty $\frak{pen}$. This gives
\begin{equation*}
\E\|\hat \theta^{(\hat m^{\hat a})}-\theta\|_p^p\leq 
\tilde M_{p}
\inf_{(m,a)\in{\mathfrak M}}\left\{\E\|\hat \theta^{(m)}-\theta\|_p^p +\frak{pen}(m,a)\right\}+
\breve M_{p,q}%\breve M_{p,q,\alpha}
\Big(N^{1-q}\|\theta\|_p^p +\varepsilon^p R^\#(\frak{M})\Big),
\end{equation*}
with
$R^\#(\mathfrak{M})=N^{(1-q)/2}\sum_{\lambda\in\Lambda}w_\lambda+ (\log N)^{\frac{p}{2}-1}R(\mathfrak{M}).$\\

Let us first analyse the remaining term $N^{1-q}\|\theta\|_p^p +\varepsilon^p R^\#(\frak{M})$.
In both cases observe that 
$$R(\frak{M})=\sum_{j=-1}^J\omega_j\sum_{\substack{(m_j,a)\in \M_j\times\{H,I,S\}\\ m_j\neq\emptyset}}|m_j|^{\left(1-\frac{p}{2}\right)_+}e^{-x^a_{m_j}}=R(\M^H)+R(\M^I)+R(\M^S).$$
We will prove in the following Sections~\ref{proofupperhomogeneous}, \ref{proofupperintermediate} , \ref{proofuppersparse}   that for each $a\in\{H,I,S\}$, 
$R(\M^a)$ is bounded by a constant, except in the intermediate case  for $p<2$ where the bound is $\log N$ up to a constant. Note that $$\sum_{\lambda} w_{\lambda}=\sum_{j=-1}^J\omega_j 2^j
\lesssim 2^{Jp/2} \lesssim N^{p/2}$$ and, using Lemma~\ref{normetheta},
 the remaining term is bounded as follows  (for $q\geq p+1$):
\begin{align*}
N^{1-q}\|\theta\|_p^p +\varepsilon^p R^\#(\frak{M})&\lesssim N^{1-q}R^p+N^{\frac{1+p-q}{2}}\e^p+\e^p(\log N)^{\frac{p}{2}-1}R(\mathfrak{M})\\
&\lesssim \left(\frac{R}{\e}\right)^{2(1-q)}R^p+\left(\frac{R}{\e}\right)^{1+p-q}\e^p+\e^p(\log N)^{\frac{p}{2}-1}R(\mathfrak{M})\\
&\lesssim \left(\frac{R}{\e}\right)^{2-2q+p}\e^p+\left(\frac{R}{\e}\right)^{1+p-q}\e^p+\e^p(\log N)^{\frac{p}{2}-1}R(\mathfrak{M}).
\end{align*}
Since $R\geq \e$, taking $q=p+1$ allows to show that this term is negligible. Indeed, we have that
$$\e^p(\log N)^{\frac{p}{2}}\lesssim R^{\frac{p}{2s+1}}\e^{\frac{2ps}{2s+1}}  \iff (\log N)^{\frac{p}{2}}\lesssim\left(\frac{R}{\e}\right)^{\frac{p}{2s+1}},$$
which is true. Thus the remaining term is negligible compared to the faster rate given in Theorem~\ref{rateslp}, and \textit{a fortiori} to the other rates.

%
%
%then 
%$$
%R^{\#}(\M)\lesssim N^{(1-q)/2}N^{p/2}+R(\M) (\log N)^{\frac{p}{2}-1}
%\lesssim (\log N)^{\frac{p}{2}-1}
%$$
%as soon as $q= p+1$. %, which we assume in the sequel.  
%We also have $N^{1-q}\|\theta\|_p^p\lesssim \e^p$ since $q=p+1$ and $N\gtrsim \e^{-2}$.
%

Now, we consider the main term, i.e.
$$\inf_{(m,a)\in{\mathfrak M}}\left\{\E\|\hat \theta^{(m)}-\theta\|_p^p +\frak{pen}(m,a)\right\}
= \min_{a\in\{H,I,S\}}\inf_{m\in \M^a}\left\{\E\|\hat \theta^{(m)}-\theta\|_p^p +\pen^a(m)\right\}.$$
In the following Sections \ref{proofupperhomogeneous} (homogeneous case), \ref{proofupperintermediate} (intermediate case), \ref{proofuppersparse} (sparse and frontier cases),  we bound for each $a\in\{H,I,S\}$ the quantities $R(\M^a)$ and we prove that 
$$\inf_{m\in \M^a}\left\{\E\|\hat \theta^{(m)}-\theta\|_p^p +\pen^a(m)\right\}\leq
\begin{cases}
 v_H(\e) \text{ if $a=H$ and } r\geq p\\
 v_I(\e) \text{ if $a=I$ and } \frac{p}{2s+1}<r< p\\
 v_S(\e) \text{ if $a=S$ and } r\leq \frac{p}{2s+1}
 \end{cases}$$
 where the $v_a(\e)$'s are the rates given in Theorem~\ref{rateslp}. 
This completes the proof.
As each subsection deals with a different case, from now we drop the upperscript $a$ for ease of notation.
%%%%%%%%%%%%%%%%%%%%%%%%%%%%%%%%%%%%%%%%%%%%%%%%

\subsubsection{
Proof of Theorem~\ref{rateslp}: homogeneous case }%$r\geq p$}
\label{proofupperhomogeneous}
In this section we assume that $r\geq p$}. Let us recall our sub-collection of models. A model $m=\cup_{j=-1}^{J} m_j\in\M=\M^H$ if for some $0\leq L\leq J$
\begin{eqnarray*}
\forall  j\leq L, \quad m_j&=&\Lambda_j=\{j\}\times K_j, \\%\qquad |m_j|=2^{j}\\
\forall  j>L, \quad m_j&=&\emptyset.
\end{eqnarray*}
Note that for any $m\in \M$, 
$$\E[V_p(m)]=\e^p\sigma_p^p\sum_{j\geq -1} \omega_j|m_j|
= \e^p\sigma_p^p\sum_{j=-1}^{L} \omega_j2^j
\leq C(p, \sigma_p) \e^p 2^{Lp/2},$$
with $C(p, \sigma_p)$ a constant only depending on $p$ and  $\sigma_p$. If $\theta$ belongs to $\mathcal{B}_{r,\infty}^s(R)$ we can prove that 
$B_p(m)\lesssim R^p2^{-Lps} $. 
Indeed the bias  verifies
$$B_p(m)=\sum_{(j,k)\notin m}\omega_j|\theta_{jk}|^p
%=\sum_{(j,k)\notin m}2^{j\left(\frac{p}2-1\right)}|\theta_{jk}|^p
=\sum_{j>L,k\in  K_j}2^{j\left(\frac{p}2-1\right)}|\theta_{jk}|^p.$$
Since $r\geq p$, H\"{o}lder's inequality gives for any set $E$, 
$$\sum_{k\in E}|\theta_{jk}|^p\leq |E|^{(1-p/r)}\left(\sum_k|\theta_{jk}|^r\right)^{p/r}.$$
When $\theta \in \mathcal{B}_{r,\infty}^{s}(R)$, that yields 
$\sum_{k\in E}|\theta_{jk}|^p\leq |E|^{(1-p/r)}R^p2^{-pj(s+\frac12-\frac1r)}.$
Then, if $E=K_j$ with cardinal $2^j$, it gives
$\sum_{k\in K_j}|\theta_{jk}|^p\leq R^p 2^{j(-ps+1-\frac{p}{2})}$,
and 
$$B_p(m)=\sum_{j>L}2^{j\left(\frac{p}2-1\right)}\sum_{k\in  K_j}|\theta_{jk}|^p\leq C(p,s)R^p 2^{-p Ls},$$
with $C(p, s)$ a constant only depending on $p$ and $s$. 
%Finally we get
%$$\E\|\hat \theta_{m}-\theta\|_p^p \leq  C (R^p2^{-Lps}+\varepsilon^p 2^{Lp/2}).$$
%Choosing $L$ such that $2^{L}=(R/\e)^{2/(2s+1)}$ 
%\comm{possible if $N\geq (2R/\e)^2$} gives
%$$\sup_{m\in \M}\E\|\hat \theta_{m}-\theta\|_p^p \leq C R^{\frac{ p}{2s+1}} \varepsilon^{\frac{2s p}{2s+1}}.$$
%This is the minimax rate when $r\geq p$. 
%Application of Theorem~\ref{theo:oracle4}: 
%We choose $x_{m_j}=c(\delta)\log|m_j|$ (with $\log 0=0$) so that $2\sum_{j=1}^J\sum_{m_j\in{\mathcal M}_j}e^{-x_{m_j}}<\delta $ and $x_{m_j}^{ \frac{p}{2}}\ll |m_j|$. So $\|\tilde \theta-\theta\|_p^p\leq C(p,\alpha, \delta)\inf_{m\in{\mathcal M}}\|\hat \theta^{(m)}-\theta\|_p^p$ with probability $1-\delta$. \\
For our PCO procedure we have chosen $x_{m_j}=K\log|m_j|$ (with $\log 0 =0$ and $K=p/2$) so that 
$$\pen^H(m)=\pen(m)=
\begin{cases}
2\e^p\sum_{j=-1}^J\omega_j \p_j(m_j)
& \text{ if } p\leq 2\\
2\e^p\sum_{j=-1}^J\omega_j \min\Big(\p_j(m_j),(2q\log N)^{\frac{p}{2}-1}\p_j^{\#}(m_j)\Big)
& \text{ if } p> 2
\end{cases}$$
with 
$$
\p_j(m_j)=\frac32\sigma_p^p|m_j|+\kappa_p2^{\frac{(p-2)_+}2}|m_j|^{\big(1-\frac{p}{2}\big)_+}(K\log|m_j|)^{\frac{p}{2}}$$
and
$$\p_j^{\#}(m_j)=\frac32\sigma_2^2|m_j|+\kappa_2 K\log|m_j|.
$$
%Theorem~\ref{theo:oracleesp} gives
%\begin{equation*}
%\E\|\tilde \theta-\theta\|_p^p\leq \tilde M_{p}\inf_{m\in{\mathcal M}}\left\{\E\|\hat \theta^{(m)}-\theta\|_p^p +\e^p\sum_{j=1}^J\omega_j|m_j|^{\left(1-\frac{p}{2}\right)_+}(\log|m_j|)^{\frac{p}{2}}\right\}+\breve M_{p} \varepsilon^p R(\M)
%\end{equation*}
%with $R(\M)=
%\sum_{j=1}^J\omega_j\sum_{m_j\in \M_j,{m_j\neq \emptyset}}|m_j|^{\left(1-p/2\right)_+}e^{-K\log|m_j|}.$
Thus, $\p_j(m_j)=0$ if $j>L$ ; and for $j\leq L$:
$$\p_j(m_j)=\frac32\sigma_p^p2^j+\kappa_p2^{\frac{(p-2)_+}2}(2^j)^{\big(1-\frac{p}{2}\big)_+}(Kj\log2)^{\frac{p}{2}}\leq C(p,\sigma_p) 2^j,$$
with $C(p, \sigma_p)$ a constant only depending on $p$ and  $\sigma_p$.  Then
\begin{eqnarray*}
\E\|\hat \theta^{(m)}-\theta\|_p^p +\pen(m) &\leq & \E[V_p(m)]+B_p(m) +\e^p\sum_{j=-1}^{L}\omega_j C(p,\sigma_p) 2^j\\
&\lesssim & \varepsilon^p2^{Lp/2}+R^p2^{-Lp s} +\e^p\sum_{j=-1}^{L}2^{j\frac{p}{2}}\\
&\lesssim & \varepsilon^p2^{Lp/2}+R^p2^{-Lps} 
\end{eqnarray*}
which provides
$$
\inf_{m\in \M}\left\{\E\|\hat \theta^{(m)}-\theta\|_p^p +\pen(m)\right\}\leq   C R^{\frac{ p}{2s+1}} \varepsilon^{\frac{2sp}{2s+1}},$$ for $C$ a constant, choosing $L$ such that $2^{L}\approx (R/\e)^{2/(2s+1)}$
(possible since $2^J=N/2\gtrsim (R/\e)^2$).
Moreover, we compute
\begin{eqnarray*}
R(\M)&=&
\sum_{j=-1}^J\omega_j\sum_{m_j\in \M_j,{m_j\neq \emptyset}}|m_j|^{\left(1-p/2\right)_+}e^{-K\log|m_j|}\\
&=&
\sum_{j=-1}^J2^{j\left(\frac{p}2-1\right)}\sum_{m_j\in \M_j}\1_{j\leq L}2^{j\left(1-\frac{p}2\right)_+}2^{-jK}\\
&\leq &
\sum_{j\geq -1}2^{j\left(\frac{p}2-1\right)}2^{j[\left(1-\frac{p}2\right)_+-K]}
=\sum_{j\geq -1}2^{j[\left(\frac{p}2-1\right)_+-K]}<\infty
\end{eqnarray*}
as soon as $K>(\frac{p}{2}-1)_+$.%, for instance $K=p/2$. 

\subsubsection{
Proof of Theorem~\ref{rateslp}: intermediate case} 
%$\frac{p}{2s+1}<r<p$}
\label{proofupperintermediate}

In this section, we assume that $\frac{p}{2s+1}<r<p$.
Let us now consider the following model, inspired from \cite{Massart2007}:
$m$ belongs to $\M(L)$ if 
$$m_j=\begin{cases}
 \Lambda_j & \text{ if } -1\leq j\leq {L-1}\\
m_{L+l} \subset \Lambda_{L+l}& \text{ if } l=j-L\geq 0 \text{ with }|m_{L+l}|=\lfloor 2^{L+l} A(l)\rfloor
\end{cases}$$
with $A(l):=2^{-lp/2}(l+1)^{-3}$.
At the end $\M=\bigcup_{L= 0} ^{J}\M(L)$. 

Note that the cardinal $\lfloor 2^{L+l} A(l)\rfloor$ is equal to 0 when 
$2^{l(p/2-1)}(l+1)^{3}>2^{L}$. Then
if $p\geq 2$,  $m_j=\emptyset$ as soon as $j>L+l_{\max}$, with $l_{\max}$ such that $$2^{L}2^{-l_{\max}(p/2-1)}(l_{\max}+1)^{-3}\approx 1.$$
Therefore, $l_{\max}$ is of order $L/(p/2-1)$ when $p>2$ and of order $2^{L/3}$ if $p=2$. When $p<2$, we only have $l\leq J-L$.

To apply our model selection strategy, we set
\begin{equation}\label{def-weight}
x_{m_j}:=K|m_j|\bigg(1+\log\Big(\frac{2^j}{|m_j|}\Big)\bigg)
\end{equation}
with $K$ large enough (see later).
Observe that at each level $j$,  we consider two types of models. Either the model $m_j$ is the whole slice 
$\{j\}\times K_j$, %$\{1,2,\ldots, 2^j\}$, 
or it is a strict subset of this slice and in this case, it means that there exists $L\leq j$ such that 
$$|m_j|=\lfloor 2^j A(j-L)\rfloor=\lfloor 2^j 2^{-(j-L)p/2}(j-L+1)^{-3}\rfloor.$$
Our choice of the factor $x_{m_j}$ automatically adapts to both types of models:
\begin{equation}\label{def:xmj}
\left\{
\begin{array}{llc}
x_{m_j}=K|m_j|  &\textrm{for the first type,}   &  \\
x_{m_j}\approx K|m_j|\bigg(1-\log A(j-L)\bigg)   & \textrm{for the second type.}   &   
\end{array}
\right.
\end{equation}
In particular, for the first type $x_{m_j}=K2^j$ and for the second type $x_{m_j}$ is of the same order as $K'|m_j|\times(j-L)$. 

%In order to apply Theorem~\ref{theo:oracleesp} or Theorem~\ref{theo:oracleesp2} , 
As explained in Section~\ref{proofupper}, we have to bound 
%$R(\M)=\sum_{j=-1}^J \omega_j \sum_{m_j\in \M_j,{m_j\neq \emptyset}}e^{-x_{m_j}}$. 
$R(\M)$, that is to
 show that the term $\sum_{j=-1}^J\omega_j\sum_{m_j\in\M_j,{m_j\neq \emptyset}} |m_j|^{\left(1-\frac{p}{2}\right)_+}e^{-x_{m_j}}$ is bounded.
 In the sequel,  for the sake of simplicity, we set
$$b:=\frac{p}{2}-1.$$ Considering the two types of models, we have, with ${\mathcal M}_j(L)=\big\{m_j:\ m\in\M(L)\big\},$
\begin{align*}
R(\M)
%&\leq \sum_{L= 0}^{J}\sum_{j=0}^J\omega_j\sum_{m_j\in\M_j(L)} e^{-x_{m_j}}\\
%&\leq \sum_{L= 0}^{J}\sum_{j=0}^J\omega_j\sum_{m_j\in\M_j(L)} e^{-x_{m_j}}(\1_{j<L}+\1_{j\geq L})\\
%&\leq \sum_{j=-1}^J\omega_j\textcolor{magenta}{(2^j)}^{\left(-b\right)_+}e^{-K\textcolor{magenta}{2^j}}+ \sum_{L= 0}^{J}\sum_{j=-1}^J\omega_j|\M_j(L)||m_j|^{\left(-b\right)_+}e^{-K|m_j|\big(1+\log\big(2^{(j-L)p/2}(j-L+1)^3\big)\big) }\1_{j\geq L}\\
% &=: T_1+T_2.
 &\leq \sum_{j=-1}^J\omega_j(2^j)^{\left(-b\right)_+}e^{-K2^j}+ \sum_{L= 0}^{J}\sum_{j=-1}^J\omega_j|\M_j(L)||m_j|^{\left(-b\right)_+}e^{-K|m_j|\big(1+\log\big(2^{(j-L)p/2}(j-L+1)^3\big)\big) }\1_{j\geq L}\\
 &=: T_1+T_2.
\end{align*}
Since $\omega_j=2^{j(\frac{p}{2}-1)}=2^{jb}$,
\begin{align*}
T_1
&\leq \sum_{j=-1}^{+\infty}2^{jb}(2^j)^{(-b)_+}e^{-K2^j}<\infty,
\end{align*}
for $K>0$. Furthermore, %\vinc{en fait, dans l'inegalite qui suit, remplacer $|m_j|$ par $2^j$ est fort} 
%\comm{todo: utiliser notation $b$ partout ou c'est possible  Ã  partir d'ici}
\begin{align*}
T_2&\leq\sum_{L= 0}^{J}\sum_{j=L}^J2^{jb}|\M_j(L)|(2^j)^{(-b)_+}e^{-K|m_j|\big(1+\log\big(2^{(j-L)p/2}(j-L+1)^3\big)\big)}\\
&\leq\sum_{L= 0}^{J}\sum_{l=0}^{J-L}2^{(L+l)b_+}|\M_{L+l}(L)|e^{-K 2^L2^{-lb}(l+1)^{-3}\big(1+\log\big(2^{lp/2}(l+1)^3\big)\big)}.
\end{align*}
For $j\geq L$,  the complexity of the collection at level $j=L+l$  is
\begin{eqnarray*}
\log |\M_j(L)|& \leq & \log \binom{2^{L+l}}{|m_{L+l}|}
\leq 
|m_{L+l}|\log \left(\frac{e2^{L+l}}{|m_{L+l}|}\right)
\end{eqnarray*}
where we have used the bound
$\log \binom{c}{d}
\leq 
d\log \left(\frac{ec}{d}\right)$.
Then
\begin{eqnarray*}
\log |\M_j(L)|
&\leq &
2^{L}2^{-lb}(l+1)^{-3}\log \left({e}2^{L+l}\lfloor 2^{L+l} A(l)\rfloor^{-1}\right)\\
&\leq & 2^{L}
2^{-lb}(l+1)^{-3}\log \left(\frac{e}{ 2^{-lp/2}(l+1)^{-3}-2^{-(L+l)}}\right)\\
&\leq &  C 2^{L}
2^{-lb}(l+1)^{-2},% =C A(l,L)(l+1)
\end{eqnarray*}
with $C$ a constant only depending on $p$.
%Since
%$$\log |\M_{L+l}(L)|\leq C 2^{L}
%2^{-l(p/2-1)}(l+1)^{-2} $$
Then we have
\begin{align*}
T_2&\leq\sum_{L= 0}^{J}\sum_{l=0}^{J-L}2^{(L+l)b_+}\exp\Big((C-Kp\log(2)/2) 2^L2^{-lb}(l+1)^{-2}\Big).
\end{align*}
Here we distinguish two cases.
\begin{description}
\item[Case $p\geq 2$:]
Recall that, if $l\geq 0$
%$|m_{L+l}|=2^{L}2^{-lb}(l+1)^{-3}.$Thus, when 
and $p\geq 2$, $m_{L+l}=\emptyset$ if $l>l_{\max}$. Then we have, for $K$ large enough such that $C-Kp\log(2)/2<0$,
 \begin{align*}
T_2&\leq \sum_{L= 0}^{J}\sum_{l=0}^{l_{\max}}2^{(L+l)b}\exp\Big((C-Kp\log(2)/2)2^{L}
2^{-lb}(l+1)^{-2}\Big)\\
&\leq \sum_{L= 0}^{J}\sum_{l=0}^{l_{\max}}2^{(L+l)b}\exp\Big((C-Kp\log(2)/2)2^{L}
2^{- l_{\max}b}( l_{\max}+1)^{-2}\Big)\\
&\leq \sum_{L= 0}^{J}\sum_{l=0}^{l_{\max}}2^{(L+l)b}\exp\Big((C-Kp\log(2)/2)( l_{\max}+1)\Big).
\end{align*}
We have used that
$$2^{L}2^{-l_{\max}b}(l_{\max}+1)^{-3}\approx 1.$$
For $p>2$, $l_{\max}\approx L/b$, and for $K$ constant large enough 
 \begin{align*}
T_2&\leq \sum_{L= 0}^{J}2^{Lb}\exp(-\tilde K L)
\end{align*}
with $\tilde K$ as large as desired
and $T_2<\infty$.\\
For $p=2$, $l_{\max}\approx 2^{L/3}$, and for $K$ constant large enough $T_2<\infty$.
%It remains to show that the weights $x_{m_j}$ give the correct rates. In the sequel, we set
%$$b:=\frac{p}{2}-1.$$
\item [Case $p<2$:]
The function $l\mapsto2^{-l(p/2-1)}(l+1)^{-2}$ is increasing except on a compact interval. Therefore, for $K$ constant large enough,
\begin{align*}
T_2&\leq\sum_{L= 0}^{J}\sum_{l=0}^{J-L}\exp\Big((C-Kp\log(2)/2) 2^L2^{-lb}(l+1)^{-2}\Big)\\
&\leq\sum_{L= 0}^{J}\sum_{l=0}^{J-L}\exp(-K'2^L)\lesssim J=\log_2(N/2),
\end{align*}
\end{description}
with $K'$ a positive constant. Finally, we have proved that that $R(\M)$ is bounded  by $\log N$ up to a constant. It means that in \eqref{RMM}, the last term
$\breve M_{p}\varepsilon^p R(\frak{M})$ is bounded by $\e^p |\log(\e)|$, which is negligible when compared to the rate.

It remains to bound $\inf_{m\in \M}\left\{\E\|\hat \theta^{(m)}-\theta\|_p^p +\pen(m)\right\},$
with 
$\pen(m)=
2\e^p\sum_{j}\omega_j\p_j(m_j)$ 
and, with a slight abuse of notation, 
$$\p_j(m_j)=\begin{cases}
\frac32\sigma_p^p|m_j|+\kappa_p |m_j|^{1-\frac{p}2}x_{m_j}^{\frac{p}2}
& \text{ if }p\leq 2 \\
(2q\log N)^{\frac{p}2-1}
\left(
\frac32\sigma_2^2|m_j|+\kappa_2 x_{m_j}
\right)
&\text{ if }p>2
\end{cases}$$
where $q=p+1$.
Since $\M=\bigcup_{L= 0} ^{J}\M(L)$, we can write 
\begin{equation}\label{tradeoff}
\inf_{m\in \M} \Big\{\E\|\hat \theta^{(m)}-\theta\|_p^p+\pen(m)\Big\}\leq\inf_{L}\left\{\sup_{m\in \M(L)}\E[V_p(m)]+\inf_{m\in \M(L)}B_p(m)+\sup_{m\in \M(L)}\pen(m)\right\}.
\end{equation}
Let us study the three terms in the right hand side.
Since $\omega_j=2^{j(p/2-1)}$, we have for any  $m\in \M(L)$, 
\begin{eqnarray*}
\E[V_p(m)]&\leq&\varepsilon^p\sigma_p^p\left(\sum_{j< L}2^{j(p/2-1)}2^j
+\sum_{l\geq 0}2^{(L+l)(\frac{p}{2}-1)}2^{L}2^{-l(p/2-1)}(l+1)^{-3}\right)\\
&\lesssim & \varepsilon^p\sigma_p^p\left(2^{Lp/2}
+2^{L(p/2-1+1)}\right)\lesssim \varepsilon^p\sigma_p^p2^{Lp/2}
\end{eqnarray*}
Therefore
$$\sup_{m\in \M(L)}\E[V_p(m)]
\lesssim\varepsilon^p\sigma_p^p2^{Lp/2}.$$
Moreover we can prove the following lemma (see Section~\ref{sec:proofbiaisBesov})
\begin{lemma}\label{biaisBesov}
If $\theta\in \mathcal{B}_{r,\infty}^{s}(R)$ 
%and $\omega_j=2^{j(p/2-1)}$
and $\frac{p}{2s+1}<r$
$$\inf_{m\in \M(L)}B_p(m)\lesssim R^p2^{-s p L}.$$
%as soon as $\frac{p}{2s+1}<r$.
\end{lemma}
For the last term $\sup_{m\in \M(L)}\pen(m)$, we distinguish two cases.
\begin{description}
\item[case $p> 2$:]
We have
$$\p_j(m_j)\leq(2q\log N)^{\frac{p}{2}-1}\big(
\frac32\sigma_2^2|m_j|+\kappa_2 x_{m_j}
\big)$$
and $\sum_{j=1}^J\omega_j\p_j(m_j)$ is bounded by (up to a constant): 
\begin{align*}
%\sum_{j=1}^J\omega_j\p_j(m_j)&\lesssim
& (\log N)^{\frac{p}{2}-1}\left[\sum_{j<L}2^{jb}|m_j|+\sum_{j\geq L}2^{jb}2^{(1+b)L}2^{-jb}(j-L+1)^{-3}\big((j-L+1)+\log(j-L+1)\big)\right]\\
&\lesssim (\log N)^{\frac{p}{2}-1}2^{Lp/2}.
\end{align*}
Finally equation %\eqref{16} and
\eqref{tradeoff} provides in the case $p> 2$
\begin{align*}
%\E\|\tilde \theta-\theta\|_p^p&\lesssim
\inf_{m\in \M} \Big\{\E\|\hat \theta^{(m)}-\theta\|_p^p+\pen(m)\Big\}&\lesssim
\inf_{L}\left\{\e^p2^{Lp/2}+R^p2^{-s p L}+\e^p(\log N)^{\frac{p}{2}-1}2^{Lp/2}\right\}
%+\e^p(\log N)^{(\frac{p}2-1)} 
\end{align*}
Using that $\log(N)\lesssim |\log(\e)|$, and considering an $L$ such that $2^L\approx\left(R^2\e^{-2}|\log(\e)|^{\frac{2}{p}-1}\right)^{\frac{1}{2s +1}}$ (possible since $2^J\gtrsim R^2\e^{-2}$), we obtain
\begin{align*}
%\E\|\tilde \theta-\theta\|_p^p
\inf_{m\in \M} \Big\{\E\|\hat \theta^{(m)}-\theta\|_p^p+\pen(m)\Big\}
\lesssim
R^{\frac{ p}{2s +1}}(\varepsilon^{2p}|\log(\e)|^{p-2})^{\frac{s }{2s +1}}.
\end{align*}
\item[case $p\leq 2$:]
We have
$$\p_j(m_j)=\frac32\sigma_p^p|m_j|+\kappa_p|m_j|^{\big(1-\frac{p}{2}\big)}x_{m_j}^{\frac{p}{2}}.$$
We just have to deal with the following term:
\begin{align*}
\sum_{j=1}^J\omega_j|m_j|^{\big(1-\frac{p}{2}\big)}x_{m_j}^{\frac{p}{2}}&\lesssim\sum_{j<L}2^{jb}|m_j|+\sum_{j\geq L}w_j|m_j|\big((j-L+1)+\log(j-L+1)\big)^{p/2}\\
&\lesssim 2^{Lp/2}+\sum_{j\geq L}2^{jb}2^{(1+b)L}2^{-jb}(j-L+1)^{-3}\big((j-L+1)+\log(j-L+1)\big)^{p/2}\\
&\lesssim 2^{Lp/2},
\end{align*}
since $3-p/2>1$. 
Finally equation\eqref{tradeoff} provides in the case $p\leq 2$
$$\inf_{m\in \M} \Big\{\E\|\hat \theta^{(m)}-\theta\|_p^p+\pen(m)\Big\}\lesssim
\inf_{L}\left\{\e^p2^{Lp/2}+R^p2^{-s p L}+\e^p2^{Lp/2}\right\}
\lesssim R^{\frac{ p}{2s +1}}\varepsilon^{\frac{2s p}{2s +1}}.$$
\end{description}

\subsubsection{
Proof of Theorem~\ref{rateslp}: sparse {and frontier} case}
\label{proofuppersparse}

%\noindent\underline{\textbf{Sparse case: $r\leq \frac{p}{2s+1}$}}

%\medskip

In this section we assume that $r\leq \frac{p}{2s+1}$. Since $s>1/r$, it only occurs when $p>2$
(since $p\geq r(2s+1)\Rightarrow  p>2+r>2$). 
%So we shall use the collection $\M^S$ only when $p>2$. 

Recall that our model collection is defined by  $\M_j=\{\{j\}\times E, \: E\in \mathcal{P}(K_j)
\}$ for any $j\geq -1$.
For this collection we choose $x_{m_j}=K{|m_j| j}$ with $K=p+1> 2+(\frac{p}{2}-1)\log(2)$. 
%
%
%With a view to apply Theorem~\ref{theo:oracleesp2}, 
As required in Section~\ref{proofupper} let us bound 
$R(\M)=\sum_{j=-1}^J \omega_j \sum_{m_j\in \M_j,{m_j\neq \emptyset}}e^{-x_{m_j}}$. 
We can write 
\begin{eqnarray*}
R(\M)&=&\sum_{j=-1}^J\omega_j \sum_{d=1}^{2^j} \sum_{m_j\in \M_j, |m_j|=d}e^{-x_{m_j}}\\
&=&\sum_{j=-1}^J\omega_j \sum_{d=1}^{2^j} {\rm card}\{m_j\in \M_j, |m_j|=d\}e^{-K{dj}}.
\end{eqnarray*}
Now we use that 
$$\log \binom{2^j}{d}\leq d\left(1+\log \left(\frac{2^j}{d}\right)\right)\leq 2 d j$$
to state
\begin{eqnarray*}
R(\M)&\leq &\sum_{j=-1}^J\omega_j \sum_{d=1}^{2^j} e^{2dj}e^{-Kdj}
\leq \sum_{j=-1}^J2^{j\left(\frac{p}{2}-1\right)} \sum_{d=1}^{\infty} e^{(2-K)dj}\\
&\leq & 
\sum_{j=-1}^J\omega_j  \frac{{e^{(2-K)j}}}{ 1- e^{(2-K)j}}
\lesssim\sum_{j=-1}^J e^{\left(2+(\frac{p}{2}-1)\log(2)-K\right)j}<\infty. 
\end{eqnarray*}
The following remark will be useful in Section~\ref{sec:end}.
\begin{remark} \label{RemarkResteSparse}
We also have:
\begin{eqnarray*}
R_2^{2/p}
&:=&\sum_{j=-1}^J \left(\omega_j \sum_{m_j\in \M_j,{m_j\neq \emptyset}}e^{-x_{m_j}}\right)^{2/p}
=\sum_{j=-1}^J\left(\omega_j \sum_{d=1}^{2^j} \sum_{m_j\in \M_j, |m_j|=d}e^{-Kdj}\right)^{2/p}\\
&\lesssim&
\sum_{j=-1}^J\left(e^{\left(2+(\frac{p}{2}-1)\log(2)-K\right)j} \right)^{2/p}<\infty
\end{eqnarray*}
\end{remark}

It remains to control the term $\inf_{m\in \M}\left\{\E\|\hat \theta^{(m)}-\theta\|_p^p +\pen(m)\right\}.$
Taking inspiration from the various works of Donoho, Johnstone, Kerkyacharian and Picard, we now define 
$$\check m_j=
\begin{cases} 
\{j\}\times K_j & \text{ if } j < j_1\\
\{j\}\times \Big\{k\in K_j,\: |\theta_{jk}|> \e \sqrt{j}\Big\} &\text{ if } j_1\leq j \leq  j_0\\
\emptyset &\text{ if } j >  j_0
\end{cases}$$
where $j_1=j_1(\e)$ and $j_0=j_0(\e)$ are defined by 
$$2^{j_1}\approx(R^{-1}\e |\log \e|^{\frac12})^{4\beta-2},%=\e^{4\beta-2}{|\log\e|^{2\beta-1}}, 
\quad 2^{j_0}\approx(R^{-1}\e |\log \e|^{\frac12})^{-2\beta/s'}$$
with
$$s'=s-\frac{1}{r}+\frac{1}{p},\quad
\beta=\frac{s'}{2s+1-\frac2{r}}$$
so that
$$p-r+\frac{2\beta}{s'}r\left(s+\frac12-\frac{p}{2r}\right)= 2p\beta.$$
%\vinc{Il faudrait sans doute faire rentrer du $R$ dans $j_1$ de sorte que l'on ait $j_1\leq j_0$.  Fait }
Observe that, since $p>2$ and $s>1/r,$ $$0<\beta<\min\Big(\frac12,s'\Big).$$ Moreover, since $R\e^{-1}\geq 1$, 
$$2^{j_0}\ll
(R^{-1}\e )^{-2\beta/s'}=(R\e^{-1} )^{2\beta/s'}\leq(R\e^{-1} )^2\lesssim N/2= 2^J.$$
Then, the model $\check m=\cup_j \check m_j$ belongs to $\M$ (even if $\check m$ depends on $\theta$). 
%\vinc{Preciser $\ll$}
It satisfies the following property, proved in Section~\ref{sec:proofvitessesparse}.
\begin{proposition} 
\label{vitessesparse2} Assume that $r<p/(2s +1)$.
There exists a positive constant $C$ (depending on $s, r , p$) such that 
$$\sup_{\theta\in \mathcal{B}_{r,\infty}^s(R)}\E\|\hat \theta^{(\check m)}-\theta\|_p^p
\leq  CR^{p(1-2\beta)} |\log \varepsilon|^{p\beta}\: \e^{2p\beta}$$
Moreover, if $r=p/(2s +1)$, $$\sup_{\theta\in \mathcal{B}_{r,\infty}^s(R)}\E\|\hat \theta^{(\check m)}-\theta\|_p^p
\leq  CR^{p(1-2\beta)} |\log \varepsilon|^{p\beta+1}\: \e^{2p\beta}.$$
%This is the minimax rate. 
\end{proposition}
The following remark will be useful in Section~\ref{sec:end}.
\begin{remark}\label{RemarkBiaisVarSparse}
We also have,  for $r<p/(2s +1)$,
$$\sup_{\theta\in \mathcal{B}_{r,\infty}^s(R)}\sum_j\left(2^{j(\frac{p}{2}-1)}\sum_k\E|(\hat \theta^{(\check m)}-\theta)_{jk}|^p\right)^{\frac{2}{p}}\leq C  R^{2(1-2\beta)}|\log \varepsilon|^{2\beta}\: \e^{4\beta},$$
and for $r=p/(2s +1)$, the right hand side  is replaced by 
 $C  R^{2(1-2\beta)}|\log \varepsilon|^{2\beta+1}\: \e^{4\beta}$.
\end{remark}
%\begin{remark} Note that the same model, replacing $\beta$ by $s/(1+2s)$ allows to obtain the rate 
%$ (\e\sqrt{|\log \e|})^{2ps/(2s+1)}$ 
%in the intermediate case ($p>r>p/(2s+1)$). \vinc{Cela ne marche pas pour le cas homogene $r\geq p$ ? Probablement oui mais calculs differents. A verifier proprement si necessaire. Remarque a enlever}
%%, with (to check) $\square=\frac{p/r-1}{4s+2}$. 
%\end{remark}
Let us now bound, for $m=\check{m}$, the term
 $$\pen(m)=
2\e^p\sum_{j=-1}^J\omega_j \min\Big(\p_j(m_j),(2q\log N)^{\frac{p}{2}-1}\p_j^\#(m_j)\Big)$$
with 
$$\begin{cases}
\p_j(m_j)=\frac32\sigma_p^p|m_j|+\kappa_p2^{\frac{(p-2)}2}(Kj|m_j|)^{\frac{p}{2}}\\
\p_j^\#(m_j)=\frac32\sigma_2^2|m_j|+\kappa_2Kj|m_j|.
\end{cases}$$
Note that for $j>j_0$, $|m_j|=0$
so that $\pen(m)\leq
2\e^p(2q\log N)^{\frac{p}{2}-1}\sum_{j=-1}^{j_0}\omega_j \p_j^\#(m_j)$.
 For $j<j_1$, $|m_j|=2^j$ so that, since $|\log(R^{-1}\e)|\leq 2  |\log\e|$,
$$\e^p\sum_{j=-1}^{j_1-1}\omega_j\p_j^\#(\check m_j)\lesssim\e^p\sum_{j=-1}^{j_1-1}\omega_j|\check m_j|j
\leq \e^p\sum_{j=-1}^{j_1-1}j2^{jp/2}\lesssim \e^p 2^{j_1p/2} j_1
\lesssim  R^{p(1-2\beta)}\e^{2\beta p}|\log\e|^{p\beta-p/2+1}.$$
Then, since $\log(N)\lesssim |\log(\e)|$,
 $$(\log N)^{\frac{p}{2}-1}\e^p\sum_{j=-1}^{j_1-1}\omega_j\p_j^\#(\check m_j)
\lesssim R^{p(1-2\beta )} \e^{2\beta p}|\log\e|^{p\beta}.$$
\\
For $j_1\leq j \leq j_0$,  we write 
$|\check{m}_j|=\sum_{k\in  K_j} \1_{|\theta_{jk}|>\e\sqrt{j}}$, thus 
$$\e^p\sum_{j=j_1}^{j_0}\omega_j|\check m_j|j=\e^p\sum_{j=j_1}^{j_0}\omega_j j \sum_{k\in  K_j} \1_{|\theta_{jk}|>\e\sqrt{j}}$$
To control this term we use the same method and bounds as used for term $A_{31}$ in the proof of Proposition~\ref{vitessesparse2}. This gives 
\begin{eqnarray*}
\e^p\sum_{j=j_1}^{j_0}\omega_j|\check m_j|j
& \lesssim & R^r\e ^{p-r}  j_0^{1-r/2}2^{-j_0r(s+\frac12-\frac{p}{2r})}\\
& \lesssim & R^r\e ^{p-r}  |\log \e|^{1-r/2}(R^{-1}\e|\log \e|^{\frac12})^{\frac{2\beta}{s'}r(s+\frac12-\frac{p}{2r})}\\
& \lesssim & R^{p-2p\beta}|\log \e|^{1-\frac{r}2+\frac{\beta}{s'}r(s+\frac12-\frac{p}{2r})}\e^{2p\beta}
\end{eqnarray*}
Then, with $\log(N)\lesssim |\log(\e)|$, 
$$(\log N)^{\frac{p}{2}-1}\e^p\sum_{j=j_1}^{j_0}\omega_j\p_j^\#(\check m_j)
\lesssim  R^{p-2p\beta}\e^{2\beta p}|\log\e|^{\frac{p}{2}-\frac{r}2+\frac{\beta}{s'}r(s+\frac12-\frac{p}{2r})}\lesssim  R^{p(1-2\beta)}\e^{2\beta p}|\log\e|^{p\beta}.$$
%\vinc{Pourquoi on perd en puissance de $R$ ?}
(In the frontier case $r=p/(2s +1)$, we obtain the bound
$$(\log N)^{\frac{p}{2}-1}\e^p\sum_{j=j_1}^{j_0}\omega_j\p_j^\#(\check m_j)\lesssim
R^{p(1-2\beta)}\e^{2\beta p}|\log\e|^{p/2-1+2-r/2}.$$)
Finally, reminding Proposition~\ref{vitessesparse2}, we obtain 
$$\inf_{m\in \M}\left\{\E\|\hat \theta^{(m)}-\theta\|_p^p +\pen(m)\right\}\leq \E\|\hat \theta^{(\check m)}-\theta\|_p^p +\pen(\check m)
\lesssim  R^{p(1-2\beta)} (\e\sqrt{|\log\e|})^{2\beta p}.$$
(In the frontier case $r=p/(2s +1)$, we obtain the bound $R^{p(1-2\beta)}\e^{2\beta p}|\log\e|^{1+p\beta}$.)
%%%%%%%
 \subsubsection{Proof of Lemma~\ref{biaisBesov}}\label{sec:proofbiaisBesov}
%  \begin{proof}[Proof of Lemma~\ref{biaisBesov}]
%  For any $j\geq 0$ and any $k\in \{2^{j}+1, \dots, 2^{j+1}\}$, we denote $\theta_{jk}=\theta_k$.
  We sort the $\theta_{jk}$'s in the following way: for any 
 $j$ we denote
$$|\theta_{j,(1)}|\geq |\theta_{j,(2)}|\geq \dots \geq $$
Then
$$\inf_{m\in \M(L)}B_p(m)=\inf_{m\in \M(L)}\sum_{(j,k)\notin m}\omega_j|\theta_{jk}|^p
%=\inf_{m\in \M(L)}
%\sum_{j,k\notin {m_j}}2^{jp(\frac{1}{2}-\frac{1}{p})}|\theta_{jk}|^p
=\sum_{l\geq 0}2^{(L+l)(\frac{p}{2}-1)}\sum_{k>\lfloor  2^{L+l}A(l)\rfloor }|\theta_{L+l,(k)}|^p.$$
To bound the last term we use the following lemma.
\begin{lemma}\label{lemmeintermediaire}
If $a_{(1)}\geq \dots \geq a_{(N)}\geq 0$, then for any $0<r\leq p$ and any
$0\leq n \leq N-1$, we have
$$\sum_{k=n+1}^N a_{(k)}^p \leq \left(\sum_{i=1}^N a_i ^r\right)^{p/r} (n+1)^{1-p/r}.$$
\end{lemma}
\begin{proof}
Let $a=a_{(n+1)}.$ Then $a\leq a_{(j)}$ for any $j\leq n+1$ and then
$(n+1)a^r\leq \sum_{i}a_{i}^r$. Therefore,
\begin{eqnarray*}
\sum_{k=n+1}^N a_{(k)}^p =\sum_{k=n+1}^N a_{(k)}^{p-r} a_{(k)}^{r} 
\leq a^{p-r}\sum_{k=n+1}^N  a_{(k)}^{r}
\leq \left(\frac{\sum_{i=1}^N a_{i}^r}{n+1}\right)^{p/r-1}\sum_{i=1}^N a_{i}^r. \end{eqnarray*}
\end{proof}
Using Lemma~\ref{lemmeintermediaire}, we can write 
\begin{eqnarray*}
\inf_{m\in \M(L)}B_p(m)&\leq &\sum_{l\geq 0}2^{(L+l)(\frac{p}{2}-1)}(\sum_k |\theta_{L+l,k}| ^r)^{\frac{p}{r}} (\lfloor 2^{L+l} A(l)\rfloor+1)^{1-\frac{p}{r}}\\
&\leq &\sum_{l\geq 0}2^{(L+l)(\frac{p}{2}-1)}(\sum_k |\theta_{L+l,k}| ^r)^{\frac{p}{r}} \big(2^{L+l}2^{-lp/2}(l+1)^{-3}\big)^{1-\frac{p}r}\\
&\leq & 2^{L(\frac{p}{2}-\frac{p}{r})}\sum_{l\geq 0}2^{l(\frac{p}{2}-1)}
(\sum_k |\theta_{L+l,k}| ^r)^{\frac{p}{r}} 
 \big(2^{l(1-p/2)}(l+1)^{-3}\big)^{1-\frac{p}r}
\end{eqnarray*}
Since $\theta\in \mathcal{B}_{r,\infty}^{s}(R)$,  for all $l$
$$\left(\sum_{k}|\theta_{L+l,k}|^r \right)^{1/r}\leq R  2^{-(s+\frac12-\frac1r)(L+l)}\Rightarrow
\left(\sum_{k}|\theta_{L+l,k}|^r \right)^{p/r}\leq R^p  2^{-(s p+\frac{p}2-\frac{p}r)(L+l)}.
$$
Finally
\begin{eqnarray*}
\inf_{m\in \M(L)}B_p(m)&\leq & R^p 2^{-s p L}\sum_{l\geq 0}
2^{lp(\frac{p}{2r}-\frac12-s )}(l+1)^{3\frac{p}r-3} 
\leq C R^p2^{-s p L}.
\end{eqnarray*}
The series converges because $\frac{p}{2r}-\frac{1}2-s <0\Leftrightarrow \frac{p}{2s+1}<r$. Lemma~\ref{biaisBesov} is proved.
%%%%%
\subsubsection{Proof of Proposition~\ref{vitessesparse2} }
\label{sec:proofvitessesparse}
We denote by $\hat \theta$ the estimator $\hat \theta^{(\check m)}$. First observe that 
$$\|\hat \theta-\theta\|_p^p=
\sum_{j}\sum_{k\in K_j}\omega_j |\hat \theta_{jk} - \theta_{jk}|^p
=A_1+A_2+A_3,$$
with
\begin{align*}
A_1:=\sum_{j>j_0}\sum_{k\in K_j}\omega_j|\hat \theta_{jk}-\theta_{jk}|^p,\quad
A_2:=\sum_{j<j_1}\sum_{k\in K_j}\omega_j|\hat \theta_{jk}-\theta_{jk}|^p\quad
A_3:=\sum_{j=j_1}^{j_0}\sum_{k\in K_j}\omega_j|\hat \theta_{jk}-\theta_{jk}|^p.
\end{align*}
We bound each sum.

\medskip

\noindent Since $\hat\theta_{jk}=0$ for $j>j_0$,
\begin{eqnarray*}
A_1=\sum_{j>j_0}\sum_{k\in K_j}\omega_j|\hat \theta_{jk}-\theta_{jk}|^p=
\sum_{j>j_0}\omega_j\sum_{k\in K_j}|\theta_{jk}|^p.
\end{eqnarray*}
Since $r\leq p $, we have $\sum_{k}|\theta_{jk}|^p\leq \left(\sum_{k}|\theta_{jk}|^r\right)^{p/r}.$
Then,  if $\theta\in \mathcal{B}_{r,\infty}^{s}(R)$
$$A_1\leq  R^p \sum_{j>j_0} 2^{j(\frac p2-1)}2^{-jp(s+\frac12-\frac1r)}\leq  R^p \sum_{j>j_0} 2^{-jp(s+\frac1p-\frac1r)}\leq R^p 2^{-j_0p s'}.$$
With the value of $j_0$ this gives
$$A_1\leq   R^{p(1-2\beta)} (\e |\log \e|^{\frac12})^{2p\beta}.$$
Let us compute
\begin{eqnarray*}
A_2&=&\sum_{j<j_1}\sum_{k\in K_j}\omega_j|\hat \theta_{jk}-\theta_{jk}|^p=
\sum_{j<j_1}\omega_j\sum_{k\in K_j}|Y_{jk}-\theta_{jk}|^p
=\e^p \sum_{j<j_1}2^{j(p/2-1)}\sum_k|\xi_{jk}|^p\\
%\E(A_2)&=&\e^p \sum_{j<j_1}2^{jp/2}\sigma_p^p \leq C_p \e^p 2^{j_1p/2}
\end{eqnarray*}
Then 
$\E\big[A_2\big]=\e^p \sum_{j<j_1}2^{jp/2}\sigma_p^p \lesssim \e^p 2^{j_1p/2}.$
With the value of $j_1$ this gives, since $\beta<1/2$,
$$\E\big[A_2\big]\lesssim R^{p(1-2\beta)}\e^{2p\beta}|\log\e|^{p\beta-p/2}\lesssim R^{p(1-2\beta)} \e^{2p\beta}. $$
 We can split the last term in the following way:
\begin{eqnarray*}
A_3&=&\sum_{j=j_1}^{j_0}\sum_{k\in K_j}\omega_j|\hat \theta_{jk}-\theta_{jk}|^p\\
&=&\underbrace{\sum_{j=j_1}^{j_0}\sum_k\omega_j|Y_{jk}-\theta_{jk}|^p\1_{|\theta_{jk}|>\e \sqrt{j}}}_{A_{31}}+\underbrace{\sum_{j=j_1}^{j_0}\sum_k\omega_j|\theta_{jk}|^p\1_{|\theta_{jk}|\leq \e \sqrt{j}}}_{A_{32}}
\end{eqnarray*}

Let us bound the expectation of the term $A_{31}$:
%\textit{Term $A_{31}$:}
\begin{eqnarray*}
\E\big[A_{31}\big] &=& 
\sum_{j=j_1}^{j_0}\sum_k\omega_j\E|Y_{jk}-\theta_{jk}|^p\1_{|\theta_{jk}|>\e \sqrt{j}}\\
&=& 
\sum_{j=j_1}^{j_0}\sum_k\omega_j\e^p \sigma_p^p\1_{|\theta_{jk}|>\e \sqrt{j}}\\
&\leq & 
\e^p \sigma_p^p\sum_{j=j_1}^{j_0}\sum_k\omega_j(\e \sqrt{j})^{-r}|\theta_{jk}|^r
\end{eqnarray*}
using that $\1_{|\theta_{jk}|>\e \sqrt{j}}\leq (|\theta_{jk}|/\e \sqrt{j})^r$. 
Now recall that  $\theta$ belongs to $\mathcal{B}_{r,\infty}^s(R)$. Then 
\begin{eqnarray*}
\E\big[A_{31}\big]
&\leq & 
\sigma_p^p\e ^{p-r} \sum_{j=j_1}^{j_0}\omega_j j^{-r/2}\sum_k|\theta_{jk}|^r\\
&\leq & R^r\sigma_p^p\e ^{p-r} \sum_{j=j_1}^{j_0}2^{j(p/2-1)} j^{-r/2}2^{-jr(s+\frac12-\frac1r)}\\
& \lesssim &R^r\e ^{p-r}  j_1^{-r/2} \sum_{j=j_1}^{j_0}2^{-jr(s+\frac12-\frac{p}{2r})}.
\end{eqnarray*}
In the sparse case, we have $s+\frac12-\frac{p}{2r}<0$.
Thus, with the definition of $j_0$:
\begin{eqnarray*}
\E\big[A_{31}\big]
& \lesssim &R^r\e ^{p-r}  j_1^{-r/2}2^{-j_0r(s+\frac12-\frac{p}{2r})}\\
& \lesssim &R^r\e ^{p-r}  |\log \e|^{-r/2}\big(R^{-1}\e|\log \e|^{\frac12}\big)^{\frac{2\beta}{s'}r(s+\frac12-\frac{p}{2r})}.
\end{eqnarray*}
%\vinc{De ce terme, on deduit une vitesse pour le cas $s+\frac12-\frac{p}{2r}=0$ ?}
%
Note that in the frontier case $s+\frac12-\frac{p}{2r}=0$, we obtain
\begin{eqnarray*}
\E\big[A_{31}\big]
& \lesssim & R^r\e ^{p-r}  j_1^{-r/2}j_0
 \lesssim  R^r\e ^{p-r} |\log(\e)|^{1-r/2}.
\end{eqnarray*} 
It remains to control the  term $A_{32}$, which is deterministic.
%\textit{Term $A_{32}$:} 
Since $p-r>0$
\begin{eqnarray*}
A_{32} &:=& \sum_{j=j_1}^{j_0}\sum_k\omega_j|\theta_{jk}|^p\1_{|\theta_{jk}|\leq  \e \sqrt{j}}\\
&\leq & \sum_{j=j_1}^{j_0}\sum_k\omega_j|\theta_{jk}|^p(\e \sqrt{j}/|\theta_{jk}|)^{p-r}.
%\\
%&\leq & \sum_{j=j_1}^{j_0}\sum_k\omega_j|\theta_{jk}|^r(\e \sqrt{j})^{p-r}
\end{eqnarray*}
Now recall  that $\theta$ belongs to $\mathcal{B}_{r,\infty}^s(R)$. Then 
\begin{eqnarray*}
A_{32}& \lesssim & \e^{p-r} \sum_{j=j_1}^{j_0}j^{(p-r)/2}\omega_j\sum_k|\theta_{jk}|^r\\
& \lesssim &R^r \e^{p-r}j_0^{(p-r)/2} \sum_{j=j_1}^{j_0}2^{-jr(s+\frac12-\frac{p}{2r})}.
\end{eqnarray*}
In the sparse case, we have $s+\frac12-\frac{p}{2r}<0$.
Thus, with the definition of $j_0$:
\begin{eqnarray*}
A_{32}& \lesssim &R^r \e^{p-r}j_0^{(p-r)/2}2^{-j_0r(s+\frac12-\frac{p}{2r})}\\ 
& \lesssim &R^r\e^{p-r} |\log \e|^{\frac{p-r}{2}}(R^{-1}\e|\log \e|^{\frac12})^{\frac{2\beta}{s'}r(s+\frac12-\frac{p}{2r})}.
\end{eqnarray*}
But remember that
$$p-r+\frac{2\beta}{s'}r\left(s+\frac12-\frac{p}{2r}\right)= 2p\beta.$$
Then $$\E\big[A_{31}+A_{32}\big]\lesssim R^{p-2p\beta}(\e|\log \e|^{\frac12})^{2\beta p},$$ which concludes the proof for $r<p/(2s+1)$.

In the frontier case $s+\frac12-\frac{p}{2r}=0$, we obtain
\begin{eqnarray*}
A_{32}
& \lesssim &R^r\e ^{p-r}  j_0^{(p-r)/2}j_0
 \lesssim R^r\e ^{p-r} |\log(\e)|^{1+(p-r)/2}
\end{eqnarray*} 
and then, using $p-r=2p\beta$,   $$\E\big[A_{31}+A_{32}\big]\lesssim R^{p-2p\beta}\e^{2\beta p}|\log \e|^{1+\beta p}.$$
%%%%%%%%%%%%%%%%%%%%
\subsection{Proofs of results of Section~\ref{sec:regression}}\label{sec:proofTheoLp}
This section is devoted to the proofs of results of Section~\ref{sec:regression} and in particular to the proof of Theorem~\ref{ratesLp}. We first give the proof of intermediary technical results stated in Section~\ref{sec:regression}. %%%%%
\subsubsection{Proof of Proposition~\ref{prop:noise} }\label{sec:proofnoiseok}
The proof of Proposition~\ref{prop:noise} needs following lemmas. The first one recalls classical facts about Orlicz norms.
\begin{lemma} \label{lem:orlicz}Let $\xi$ be a sub-Gaussian random variable.   
\begin{enumerate}
\item $\left(\E|\xi|^p\right)^{1/p}\leq 2\sqrt{p}\|\xi\|_{\psi_2}.$
\item $\||\xi|^p\|_{\psi_{2/p}} = \|\xi\|_{\psi_2}^p.$
\item Let $X=|\xi|^p-\E|\xi|^p.$ There exists  $C_1$ a positive constant only depending on $p$ such that $\|X\|_{\psi_{2/p}} \leq C_1 \|\xi\|_{\psi_2}^p$.
\end{enumerate}
\end{lemma}

\begin{proof}[Proof of Lemma~\ref{lem:orlicz}]

\begin{enumerate}
\item See \cite{vershynin2018} Proposition 2.5.2 (sub-Gaussian properties) and its proof, as well as Definition 2.5.5.
\item This comes directly from the definitions of the Orlicz norms. 
\item We can prove that for any variable $Y$ we have $\|Y-\E Y\|_{\psi_{2/p}}\leq C_1\|Y\|_{\psi_{2/p}}$ similarly as  Lemma 2.6.6  in \cite{vershynin2018} (it uses the triangular inequality and the fact that $\E|Y|\leq C(p)\|Y\|_{\psi_{2/p}}$). Then we take $Y=|\xi|^p$ and we use the previous point.
\end{enumerate}
\end{proof}
In Section~\ref{sec:regression}, we are faced with non identically distributed variables. In this case, we use the following result, derived from Theorem~\ref{theo:concentration}.
\begin{lemma}\label{coro:concentrationpasid}
Let $p\geq 1$. Assume that the $\xi_{\lambda}$'s are centered independent sub-Gaussian variables.  
We assume that there exists a positive constant  $\tau$ such that for any $\lambda \in \mathcal{I},\; \|\xi_{\lambda}\|_{\psi_2} \leq \tau$. Then,  for all $\lambda \in I$, $\left(\E|\xi_{\lambda}|^p\right)^{1/p}\leq 2\sqrt{p}\tau$.
Moreover, denoting $$Z:= \sum_{\lambda \in \mathcal{I}} |\xi_\lambda|^p,$$ we have $$\E(Z)\leq \sigma_p^{p} D,$$ with $\sigma_p:= 2 \sqrt{p}\tau$ and $D=\mbox{card}(\mathcal{I})$. Furthermore, 
for any $x\geq 1$,
with probability larger than $1-2\exp(-x)$,
\begin{equation*}%\label{conc-coro}
\sum_{\lambda \in \mathcal{I}} |\xi_\lambda|^p < \frac{3}{2}\sigma_p^pD+\kappa_pD^{\big(1-\frac{p}{2}\big)_+}x^{\frac{p}{2}},
\end{equation*}
where $\kappa_p$ is a positive constant only depending on $p$ 
and $\sigma_p$.
\end{lemma}
\begin{proof}[Proof of Lemma~\ref{coro:concentrationpasid}]
Using the first point of Lemma~\ref{lem:orlicz} and our assumption on uniform subgaussianity, we have:
$$\E(Z)=\sum_{\lambda \in \mathcal{I}}\E(|\xi_{\lambda}|^p)
\leq \sum_{k \in \mathcal{I}}(2 \sqrt{p}\|\xi_{\lambda}\|_{\psi_2})^{p}
\leq (2 \sqrt{p}\tau)^{p} |\mathcal{I}|=\sigma_p^pD.
$$
%In the sequel we denote $$\sigma_p:= C_1 \sqrt{p}\tau$$
Now we apply Theorem~\ref{theo:concentration}. Recall that $b_\lambda=\|X_{\lambda}\|_{\psi_{2/p}}$ with 
$X_{\lambda}=|\xi_{\lambda}|^p-\E|\xi_{\lambda}|^p$.
The third point of Lemma~\ref{lem:orlicz} gives:
$$\|b\|_{\ell_q}^q=\sum_{\lambda\in \mathcal{I}}\|X_{\lambda}\|_{\psi_{2/p}}^q\leq \sum_{\lambda \in \mathcal{I}} (C_1 \|\xi_{\lambda}\|_{\psi_2}^p)^q\leq (C_1 \tau^p)^q |\mathcal{I}|$$
so that $\|b\|_{\ell_q}\leq C_1\tau^p D^{1/q}$.
Theorem~\ref{theo:concentration} gives that with probability larger than $1-2e^{-x}$,
\begin{eqnarray*}
|Z-\E(Z)|&\leq &d_{1,p}\|b\|_{\ell_2}\sqrt{x}+d_{2,p}\|b\|_{\ell_{1/(1-p/2)_+}}x^{p/2}
\\
&\leq &d_{1,p}C_1\tau^p\sqrt{Dx}+d_{2,p}C_1\tau^p
D^{(1-p/2)_+}x^{p/2}\\
&\leq &d'_{1,p}\sigma_p^p\sqrt{Dx}+d'_{2,p}\sigma_p^p
D^{(1-p/2)_+}x^{p/2}
\end{eqnarray*}
with $d'_{ip}=d_{ip} C_1/(2\sqrt{p})^p$.
Recalling that 
$\E(Z)\leq \sigma_p^{p} D$, we can obtain
with probability larger than $1-2\exp(-x)$,
\begin{equation*}
Z\leq \frac{3}{2}\sigma_p^pD+\kappa_pD^{\big(1-\frac{p}{2}\big)_+}x^{\frac{p}{2}},
\end{equation*}
with the same proof as the one of Corollary 2.4.
\end{proof}
We now prove Proposition~\ref{prop:noise}.
Let us fix $j\geq -1$ and $k\in K_j$. Remember that 
 $$\xi_{jk}=\frac1{\sqrt{n}}\sum_{i=1}^n \eta_i\varphi_{jk} (t_i).$$
The $\xi_{jk}$'s are centered sub-Gaussian random variables. Indeed, since  the $\eta_i$'s are i.i.d. centered sub-Gaussian random variables, there exists a constant $c$ such that for any $t\in\R$,
$$\E[\exp(t\eta_i)]\leq\exp (ct^2), \quad i=1,\ldots,n$$
(see Proposition~2.5.2 of \cite{vershynin2018}, actually $c=2\|\eta_1\|_{\psi_2}$) and
\begin{align*}
\E[\exp(t\xi_{jk})]\leq \exp \Bigg(ct^2\frac{\sum_{i=1}^n \varphi_{jk}^2 (t_i)}{n}\Bigg),\quad t\in\R.
\end{align*}
It remains to prove that $\frac{\sum_{i=1}^n \varphi_{jk}^2 (t_i)}{n}$ is bounded by a constant independent of $n$ and $(j,k)$.

In the sequel, we assume that $j\geq 0$ so that $\varphi_{jk}=\psi_{jk}$. Remember that the father and mother wavelets $\phi$ and $\psi$ are assumed to be supported by the compact interval $[A,B]$. Therefore, if $\psi_{jk}(t_i)\not=0$ then
$$A\leq 2^j \frac{i}{n}-k\leq B$$
and the size of the set of $i$'s such that $\psi_{jk}(t_i)\not=0$ is not larger than $n2^{-j}$ up to a contant  only depending on $\psi$. This yields that
$$\sum_{i=1}^n \frac{\psi_{jk}^2 (t_i)}{n}=\frac{2^j}{n}\sum_{i=1}^n\psi^2 \Big( 2^j \frac{i}{n}-k\Big)$$
is bounded by a constant only depending on $\psi$. The proof for the case $j=-1$ is similar.  This shows that the $\xi_{jk}$'s are sub-Gaussian variables. Moreover, using property 4 of Proposition~2.5.2 of \cite{vershynin2018}, this ensures the existence of a positive constant $\tau$ only depending on $\phi$, $\psi$ and $\|\eta_1\|_{\psi_2}$ such that for all $j\geq -1$ and for all $k\in K_j$,
$\|\xi_{jk}\|_{\psi_2}\leq \tau$.
\label{tau}

We wish to apply Theorem~\ref{theo:concentration}. However, the $\xi_{jk}$'s  are not independent. 
But, still assuming that $j\geq 0$, we have:
$$\xi_{j,k}=\frac1{\sqrt{n}}\sum_{i=1}^n \eta_i\psi_{j,k} (t_i)=\frac1{\sqrt{n}}\sum_{i\in M_{jk}} \eta_i2^{j/2}\psi(2^jt_i-k),$$
where
$$M_{jk}=\big\{i:\ A\leq 2^jt_i-k\leq B\big\}=\big\{i:\  2^jt_i-B\leq k\leq  2^jt_i-A\big\}.$$
Finally, if $\xi_{j,k}$ depends on $(t_i,\eta_i)$, it means that $i\in M_{jk}$.
Now, we take $k'>k+B-A$. If $i\in M_{jk}$ then $k\geq 2^jt_i-B$, which yields that $k'>2^jt_i-A$ meaning that $i\not\in  M_{jk'}$. Therefore $\xi_{j,k'}$ does not depend on $(t_i,\eta_i)$. We conclude that if $k'>k+B-A$ then $\xi_{j,k}$ and $\xi_{j,k'}$ are independent. Now, we build a deterministic partition of ${\mathcal I}_j$
$${\mathcal I}_j={\mathcal I}_{j1}\cup\cdots\cup {\mathcal I}_{jK},$$
where the set $({\mathcal I}_{j\ell})_{\ell}$'s are built so that if for some given $\ell$, if we take two distinct elements $k$ and $k'$ of ${\mathcal I}_{j\ell}$, then $|k-k'|>B-A$. Therefore $\xi_{j,k}$ and $\xi_{j,k'}$ are independent. In particular, we can take $K$, the size of the partition, of order $B-A$ and the size of ${\mathcal I}_{j\ell}$ is smaller than~$|{\mathcal I}_{j}|$.
\\
Now, we apply Lemma~\ref{coro:concentrationpasid}, with $\tau=\sup_{jk}\|\xi_{jk}\|_{\psi_2}$. Let $x\geq 1$.
%Now, we apply Theorem~\ref{theo:concentration} and we set $d'_{1,p}=d_{1,p}K$, $d_{2,p}'=d_{2,p}K$, $b_{j}=(
%\||\xi_{jk}|^p-\E|\xi_{jk}|^p\|_{\psi_{2/p}})_{k\in K_j}$. Therefore, given $x>0$,
%\begin{align*}
%&\P(|Z_j-\E[Z_j]|\geq d'_{1,p}|b_j|_2\sqrt{x}+d'_{2,p}|b|_{1/(1-p/2)_+}x^{p/2})\\
%%&=\P\Big(\Big| \sum_{k \in \mathcal{I}_j} |\xi_{jk}|^p-\E[|\xi_{jk}|^p]\Big|\geq d'_{1,p}|b|_2\sqrt{|{\mathcal I}_{j}|x}+d_{2,p}|{\mathcal I}_{j}|^{(1-p/2)_+}x^{p/2}\Big)\\
%&=\P\Big(\Big|\sum_{\ell=1}^K\sum_{k\in{\mathcal I}_{j\ell}}  |\xi_{jk}|^p-\E[|\xi_{jk}|^p]\Big|\geq d'_{1,p}|b|_2\sqrt{|{\mathcal I}_{j}|x}+d_{2,p}|{\mathcal I}_{j}|^{(1-p/2)_+}x^{p/2}\Big)\\
%&\leq\P\Big(\sum_{\ell=1}^K\Big|\sum_{k\in{\mathcal I}_{j\ell}}  |\xi_{jk}|^p-\E[|\xi_{jk}|^p]\Big|\geq d'_{1,p}|b|_2\sqrt{|{\mathcal I}_{j}|x}+d_{2,p}|{\mathcal I}_{j}|^{(1-p/2)_+}x^{p/2}\Big)\\
%&\leq\sum_{\ell=1}^K\P\Big(\Big|\sum_{k\in{\mathcal I}_{j\ell}} |\xi_{jk}|^p-\E[|\xi_{jk}|^p]\Big|\geq d_{1,p}|b_{j\ell}|_2\sqrt{x}+d_{2,p}|{\mathcal I}_{j\ell}|^{(1-p/2)_+}x^{p/2}\Big)\\
%%&\leq \sum_{\ell=1}^K\P\Big(\Big|\sum_{jk\in{\mathcal I}_{j\ell}}  |\xi_{jk}|^p-\E[|\xi_{jk}|^p]\Big|\geq d_{1,p}\sqrt{|{\mathcal I}_{j\ell}|xK^{-2}}+d_{2,p}|{\mathcal I}_{j\ell}|^{(1-p/2)_+}(xK^{-2/p})^{p/2}\Big)\\
%&\leq \sum_{\ell=1}^K2e^{-x}=2Ke^{-x}.
%\end{align*}
%and $\sigma_p:= 2 \sqrt{p}\tau$. 
For any $1\leq \ell \leq K$
\begin{equation*}%\label{conc-coro}
\P\left(\sum_{k \in \mathcal{I}_{j\ell}} |\xi_{jk}|^p\geq  \frac{3}{2}\sigma_p^p|{\mathcal I}_{j\ell}|+\kappa_p|{\mathcal I}_{j\ell}|^{\big(1-\frac{p}{2}\big)_+}x^{\frac{p}{2}}\right)\leq 2 e^{-x}.
\end{equation*}
Then, with probability larger that $1-2Ke^{-x}$
\begin{align*}
\sum_{k \in \mathcal{I}_{j}} |\xi_{jk}|^p=\sum_{\ell=1}^K\sum_{k \in \mathcal{I}_{j\ell}} |\xi_{jk}|^p < \frac{3}{2}\sigma_p^p\sum_{\ell=1}^K|{\mathcal I}_{j\ell}|+\kappa_p\sum_{\ell=1}^K|{\mathcal I}_{j\ell}|^{\big(1-\frac{p}{2}\big)_+}x^{\frac{p}{2}}\\
%\leq  \frac{3}{2}\sigma_p^p|{\mathcal I}_{j}|+\kappa_p\sum_{\ell=1}^K|{\mathcal I}_{j\ell}|^{\big(1-\frac{p}{2}\big)_+}x^{\frac{p}{2}}\\
<   \frac{3}{2}\sigma_p^p|{\mathcal I}_{j}|+\kappa_pK^{1-\big(1-\frac{p}{2}\big)_+}\left(\sum_{\ell=1}^K|{\mathcal I}_{j\ell}|\right)^{\big(1-\frac{p}{2}\big)_+}x^{\frac{p}{2}}
\end{align*}
using the concavity of $x\mapsto x^{\big(1-\frac{p}{2}\big)_+}$. This gives
\begin{equation*}%\label{conc-coro}
\P\left(\sum_{k \in \mathcal{I}_{j}} |\xi_{jk}|^p\geq  \frac{3}{2}\sigma_p^p|{\mathcal I}_{j}|+\kappa_pK^{1-\big(1-\frac{p}{2}\big)_+}|{\mathcal I}_{j}|^{\big(1-\frac{p}{2}\big)_+}x^{\frac{p}{2}}\right)\leq 2 Ke^{-x}.
\end{equation*}
The proof for the case $j=-1$ is similar.
%%%%%
\subsubsection{Proof of Lemma~\ref{lemma:approx}}
\label{sec:prooflemmaapprox}
Observe that, since $\phi$ and $\psi$ are assumed to be $C^{M+1}$ and compactly supported, then Corollary~5.5.2 of \cite{MR1162107} ensures that $\psi$ is orthogonal to polynomials of degree less or equal to~$M$. Therefore Assumption~1 of  \cite{MR1426459} is satisfied. We then use the following result.
\begin{proposition}[Proposition 2 of \cite{MR1426459}]\label{PropRemTerm} We assume that $\phi$ and $\psi$ are $C^{M+1}$ and that $f\in{\mathcal B}^s_{r,\infty}(R)$, with $1/r<s< M+1$. Then, if $p\geq r$,
$$\Big\|\sum_{\lambda\in\Lambda^{(N)}}\theta_{jk}\varphi_{jk}-f\Big\|_{\Lp}\lesssim RN^{-(s-1/r+1/p)},$$
where $N$ is the cardinal of $\Lambda^{(N)}$.
\end{proposition}
%\vinc{C'est a cause de ce resultat que l'on n'a pas la bonne dependance en le rayon. Voir ci-apres}
Now we consider the three cases:\\
- If $r<\frac{p}{2s+1}$, since $s>1/r$, 
\begin{align*}
\Big\|\sum_{\lambda\in\Lambda^{(N)}}\theta_{jk}\varphi_{jk}-f\Big\|_{\Lp}^p&\lesssim R^pN^{-p(s-\frac{1}{r}+\frac{1}{p})}
\\& 
\lesssim R^p\big(|\log(\e)|\e^{2}\big)^{p(s-\frac{1}{r}+\frac{1}{p})}\\
&\lesssim R^p\big(|\log(\e)|\e^{2}\big)^{p\frac{s-\frac{1}{r}+\frac{1}{p}}{2s+1-\frac2{r}}}. 
\end{align*}
- If $\frac{p}{2s+1}\leq r\leq p$, observe that
\begin{align*}
r\geq\frac{p}{2s+1}&\iff p\leq r(1+2s)
%\\&
\iff -1/p\leq -1/(r(1+2s))\\
&\iff 1/r-1/p\leq 1/r(1-1/(1+2s))\\
&\iff 1/r-1/p\leq 1/r\times 2s/(1+2s).
\end{align*}
Therefore, using $s>1/r$, we have:
$$1/r-1/p<2s^2/(1+2s),$$
which means that
$$s-1/r+1/p>s/(1+2s)$$
and
\begin{align}\label{randp}
\Big\|\sum_{\lambda\in\Lambda^{(N)}}\theta_{jk}\varphi_{jk}-f\Big\|_{\Lp}^p&\lesssim R^p\big(|\log(\e)|\e^{2}\big)^{p(s-\frac{1}{r}+\frac{1}{p})}\nonumber\\
&\lesssim R^p\e^{\frac{2ps}{1+2s}} .
\end{align}
- If $r\geq p$, we have, since functions are compactly supported,
\begin{align*}
\Big\|\sum_{\lambda\in\Lambda^{(N)}}\theta_{jk}\varphi_{jk}-f\Big\|_{\Lp}&\lesssim \Big\|\sum_{\lambda\in\Lambda^{(N)}}\theta_{jk}\varphi_{jk}-f\Big\|_{{\mathbb L}_r}\lesssim R\e^{\frac{2s}{1+2s}},
\end{align*}
using \eqref{randp}.
%%%%%
\subsubsection{Proof of Lemma~\ref{lp-besov}}
\label{sec:lp-besov}
In the sequel, we assume that
$\|g\|_{{\mathcal B}^0_{p,p\wedge 2}}^{p\wedge 2}<\infty$.\\[0.5cm]
Section 9.2 of Daubechies (1992) shows that for $1< p<\infty$
\begin{align}
g\in\L_p&\iff \Big[\sum_{j\geq -1}\sum_{k\in K_j}\big|\langle g,\varphi_{jk}\rangle\big|^2\varphi_{jk}^2(\cdot)\Big]^\frac12\in\L_p\label{Daub1}\\
&\iff\Big[\sum_{j\geq -1}\sum_{k\in K_j}\big|\langle g,\varphi_{jk}\rangle\big|^22^j1_{[2^{-j}k,2^{-j}(k+1)]}(\cdot)\Big]^\frac12\in\L_p\label{Daub2}
\end{align}
%\textcolor{red}{(voir la convention adoptee par Daubechies dans son livre, par exemple Theorem 9.1.6).}\\ [0.5cm]
\textbf{Case $1<p\leq 2$.} We use \eqref{Daub1}. Since $\|\cdot\|_{\ell_2}\leq \|\cdot\|_{\ell_p},$
\begin{align*}
\Big[\sum_{j\geq -1}\sum_{k\in K_j}\big|\langle g,\varphi_{jk}\rangle\big|^2\varphi_{jk}^2\Big]^\frac{p}{2}\leq\sum_{j\geq -1}\sum_{k\in K_j}\big|\langle g,\varphi_{jk}\rangle\big|^p\big|\varphi_{jk}\big|^p
\end{align*}
and
\begin{align*}
\|g\|_{\L_p}^p&\lesssim
\int\Big[\sum_{j\geq -1}\sum_{k\in K_j}\big|\langle g,\varphi_{jk}\rangle\big|^2\varphi_{jk}^2(x)\Big]^\frac{p}{2}dx\\
&\leq\sum_{j\geq -1}\sum_{k\in K_j}\big|\langle g,\varphi_{jk}\rangle\big|^p\int\big|\varphi_{jk}(x)\big|^pdx\\
&\lesssim \sum_{j\geq -1}2^{j(\frac{p}{2}-1)}\sum_{k\in K_j}\big|\langle g,\varphi_{jk}\rangle\big|^p=\|g\|_{{\mathcal B}^0_{p,p}}^{p}<\infty.
\end{align*}
This shows that $\|g\|_{\L_p}\lesssim \|g\|_{{\mathcal B}^0_{p,p\wedge 2}}<\infty$.\\ [0.5cm]
\textbf{Case $p\geq 2$.} We use \eqref{Daub2}
We recall the generalized Minkowski inequality: Let $(X,{\mathcal A},\mu)$ and $(Y,{\mathcal B},\nu)$ two $\sigma$-finite measure spaces and $F$ a measurable function. Then, for any $q\in [1,+\infty)$,
$$\Bigg(\int_X\Big(\int_Y|F(x,y)|d\nu(y)\Big)^qd\mu(x)\Bigg)^{\frac{1}{q}}\leq\int_Y\Big(\int_X|F(x,y)|^qd\mu(x)\Big)^{\frac{1}{q}}d\nu(y).$$
We apply this inequality with $\mu$ the Lebesgue measure and $\nu$ the counting measure. We take $q=p/2$ and $F=\sum_{k\in K_j}\big|\langle g,\varphi_{jk}\rangle\big|^22^j1_{[2^{-j}k,2^{-j}(k+1)]}$. We have
\begin{align*}
\|g\|_{\L_p}^2&\lesssim
%\Bigg(\int\Big(\int|F(x,y)|d\nu(y)\Big)^qd\mu(x)\Bigg)^{\frac{1}{q}}&=
\Bigg(\int\Big(\sum_{j\geq -1}\sum_{k\in K_j}\big|\langle g,\varphi_{jk}\rangle\big|^22^j1_{[2^{-j}k,2^{-j}(k+1)]}(x)\Big)^{\frac{p}{2}}dx\Bigg)^{\frac{2}{p}}\\
&\leq\sum_{j\geq -1}\Bigg( \int \Big(\sum_{k\in K_j}\big|\langle g,\varphi_{jk}\rangle\big|^22^j1_{[2^{-j}k,2^{-j}(k+1)]}(x)\Big)^{\frac{p}{2}} dx\Bigg)^{\frac{2}{p}}\\
&=\sum_{j\geq -1}\Bigg( \int \sum_{k\in K_j}\big|\langle g,\varphi_{jk}\rangle\big|^p2^{j\frac{p}{2}}1_{[2^{-j}k,2^{-j}(k+1)]}(x) dx\Bigg)^{\frac{2}{p}}\\
&\leq\sum_{j\geq -1}\Bigg(\sum_{k\in K_j}\big|\langle g,\varphi_{jk}\rangle\big|^p2^{j(\frac{p}{2}-1)}\Bigg)^{\frac{2}{p}}\\
&\leq\sum_{j\geq -1}2^{j2(\frac{1}{2}-\frac{1}{p})}\Bigg(\sum_{k\in K_j}\big|\langle g,\varphi_{jk}\rangle\big|^p\Bigg)^{\frac{2}{p}}=\|g\|_{{\mathcal B}^0_{p, 2}}^{ 2}<\infty.
\end{align*}
This shows that $\|g\|_{\L_p}\lesssim \|g\|_{{\mathcal B}^0_{p,p\wedge 2}}<\infty$.\\ [0.5cm]
\textbf{Case $p=1$.}  Finally, we deal with the case $p=1$. We have
$$\|g\|_{{\mathcal B}^0_{p,p\wedge 2}}^{p\wedge 2}=\sum_{j\geq -1}2^{-j/2}\sum_{k\in K_j}\big|\langle g,\varphi_{jk}\rangle\big|$$
and, with $g=\sum_{j\geq -1}\sum_{k\in K_j}\langle g,\varphi_{jk}\rangle\varphi_{jk},$
\begin{align*}
\|g\|_{\L_1}&=\int\Big|\sum_{j\geq -1}\sum_{k\in K_j}\langle g,\varphi_{jk}\rangle\varphi_{jk}(x)\Big|dx\\
&\leq\sum_{j\geq -1}\sum_{k\in K_j}\big|\langle g,\varphi_{jk}\rangle\big|\int \big|\varphi_{jk}(x)\big|dx\\
&\lesssim\sum_{j\geq -1}2^{-j/2}\sum_{k\in K_j}\big|\langle g,\varphi_{jk}\rangle\big|=\|g\|_{{\mathcal B}^0_{1, 1}}<\infty.
\end{align*}
This shows that $\|g\|_{\L_1}\lesssim \|g\|_{{\mathcal B}^0_{1,1}}<\infty$.
%%%%
\subsubsection{End of the proof of Theorem~\ref{ratesLp}}\label{sec:end}
The proof of Theorem~\ref{ratesLp} uses the following lemma which is a consequence of Proposition~2 of \cite{MR1426459} and the assumption $f\in{\mathcal B}^s_{r,\infty}(R)$. It uses that $\log_2(n)$ is an integer.\\
%\vinc{J'admets ce lemme. Je pense qu'il est vrai.}
\begin{lemma}\label{lemma:thetabeta} If $f$ belongs to the Besov set ${\mathcal B}^s_{r,\infty}(R)$, then we have 
\begin{equation}\label{decroitheta}
\sum_{k\in K_j}|\theta_{jk}|^r\leq C R^r2^{-jr(s+1/2-1/r)},
\end{equation}
for $C$ a constant only depending on $\phi$, $\psi$, $s$ and $r$.
%where $\theta_{jk}=\frac1n\sum_{i=1}^nf(t_i) \varphi_{jk} (t_i)$.
\end{lemma}

Now, when $p\leq 2$, the proof of Theorem~\ref{ratesLp} follows from \eqref{Decomp-LP} combined with Proposition~\ref{prop:noise}, Theorem~\ref{rateslp}, Lemmas~\ref{lemma:thetabeta} and~\ref{lemma:approx} and Inequality~\eqref{lp-Lp<2}.

When $p> 2$, Inequality~\eqref{lp-Lp<2} does not hold, but Inequality~\eqref{lp-Lp>2} shows that we only have to bound 
\begin{equation}\label{def-A}
%A:=
\left[\sum_{j= -1}^J\Bigg(\E\Big[2^{j(\frac{p}{2}-1)}\sum_{k\in K_j}\big|\hat \theta_{jk}^{(\hat m)}-\theta_{jk}\big|^p\Big]\Bigg)^{\frac{2}{p}}\right]^{\frac{p}{2}}
\end{equation}
to conclude.
%\begin{center}
%\vinc{\Large Fin de la relecture}
%\end{center}

Remember that $\hat{m}=\argmin_{m\in \mathcal{M}}Crit(m)$ with
$$Crit(m)=-\sum_{\la \in m}w_\lambda|Y_{\la}|^p+\pen(m)$$
and
$\pen(m)$ has the form $\pen(m)=2\e^p\sum_{j=1}^J\omega_j P_j(m_j)$ (see \eqref{pen2}).
This can be rewritten
$$Crit(m)=-\sum_{j\geq -1}2^{j(\frac{p}{2}-1)}\Big[\sum_{k\in K_j}|Y_{jk}|^p-2\e^pP_j(m_j)\Big],$$
so that $\hat{m}$ is the union of disjoint sets $\hat{m}_j$ obtained by maximizing
$$m_j\longmapsto Crit_j(m_j):=\sum_{k\in K_j}|Y_{jk}|^p-2\e^pP_j(m_j).$$
Let $m_j$ be some model in $\M_j$. Replacing for $j'\not=j$, $w_{j'}$ by 0 we obtain directly from Theorem~\ref{theo:oracle1}, for any model $m_j$,
$$\|\tilde \theta_{j.}-\theta_{j.}\|_p^p\leq M_{p}\|\hat \theta^{(m_j)}-\theta_{j.}\|_p^p +
2\Big[2V_p(\hat m_j)-\pen(\hat m_j)\Big]
- 2\Big[2V_p(m_j)-\pen(m_j)\Big],$$
%\vinc{Attention a la constante $M_p$}
which means
%\color{blue}
\begin{small}
$$\sum_{k\in K_j}\big|\tilde\theta_{jk}-\theta_{jk}\big|^p\leq M_{p}\sum_{k\in K_j}\big|\hat \theta_k^{(m_j)}-\theta_{jk}\big|^p +
4\e^p\Big[\sum_{k\in \hat m_j}|\xi_{jk}|^p-P_j(\hat m_j)\Big]
- 4\e^p\Big[\sum_{k\in m_j}|\xi_{jk}|^p-P_j(m_j)\Big].$$
\end{small}
With $Z(m_j)=\sum_{k\in m_j}|\xi_{jk}|^p$, mimicking the proof of Theorem~\ref{theo:oracleesp2}, by taking $q$ large enough,  
%\textcolor{gray}{(et il y a du $\tau$ (defini page \pageref{tau}) cach\'e la-dedans, et aussi un analogue du lemme \ref{normetheta})}
\begin{align*}
\sum_{j= -1}^J\Big(2^{j(\frac{p}{2}-1)}\sum_{k\in K_j}\E\Big[\big|\tilde\theta_{jk}-\theta_{jk}\big|^p\Big]\Big)^{\frac{2}{p}}&\lesssim \sum_{j\geq -1}\Big(2^{j(\frac{p}{2}-1)}\sum_{k\in K_j}\E\Big[\big|\hat \theta_k^{(m_j)}-\theta_{jk}\big|^p\Big]+\e^p2^{j(\frac{p}{2}-1)}P_j(m_j)\Big)^{\frac{2}{p}}\\
&\hspace{1cm}+ (\log N)^{1-\frac{2}{p}}\e^2\sum_{j= -1}^J\Big(2^{j(\frac{p}{2}-1)}\sum_{\substack{m_j\in \M_j\\m_j\neq\emptyset}}e^{-x_{m_j}}\Big)^{\frac{2}{p}}+O(\e^2).
\end{align*}
%\color{black}
We have
\begin{align*}
2^{j(\frac{p}{2}-1)}\sum_{k\in K_j}\big|\hat \theta_k^{(m_j)}-\theta_{jk}\big|^p&=2^{j(\frac{p}{2}-1)}\sum_{k\in m_j} |Y_{jk}-\theta_{jk}|^p+2^{j(\frac{p}{2}-1)}\sum_{k \notin m_j} |0-\theta_{jk}|^p\\
&=:V_p(m_j)+B_p(m_j),
\end{align*}
with
\begin{equation}
B_p(m_j):=2^{j(\frac{p}{2}-1)}\sum_{k \notin m_j} |\theta_{jk}|^p\quad\mbox{and}\quad V_p(m_j):=\e^p\sum_{k \in m_j}2^{j(\frac{p}{2}-1)}  | \xi_{jk}|^p.
\end{equation}
This gives 
$$
\E[V_p(m_j)]\leq \e^p\sigma_p^p2^{j(\frac{p}{2}-1)}  |m_j|
$$
and 
\begin{align*}
A(m)&:=\sum_{j=-1}^J\Big(2^{j(\frac{p}{2}-1)}\sum_{k\in K_j}\E\Big[\big|\hat \theta_k^{(m_j)}-\theta_{jk}\big|^p\Big]+\e^p2^{j(\frac{p}{2}-1)}P_j(m_j)\Big)^{\frac{2}{p}}\\
&\lesssim \sum_{j=-1}^J\Big(B_p(m_j)+\e^p\sigma_p^p2^{j(\frac{p}{2}-1)}  |m_j|+\e^p2^{j(\frac{p}{2}-1)}P_j(m_j)\Big)^{\frac{2}{p}}\\
&\lesssim \sum_{j=-1}^JB_p^{\frac{2}{p}}(m_j)+\sum_{j= -1}^J\Big(\e^p\sigma_p^p2^{j(\frac{p}{2}-1)}  |m_j|\Big)^{\frac{2}{p}}+\sum_{j= -1}^J\Big(\e^p2^{j(\frac{p}{2}-1)}P_j(m_j)\Big)^{\frac{2}{p}}\\
&=:\tilde{B}_p(m)+ \tilde{V}_p(m)+\tilde{P}_p(m).
\end{align*}
Then we have to control for each collection $\M$   (Homogeneous, Intermediate, Sparse) the terms $\tilde{B}_p(m)$, $\tilde{V}_p(m)$, $\tilde{P}_p(m)$ for $m\in \M$, as well as 
$$\tilde R(\M):=(\log N)^{1-\frac{2}{p}}\e^2\sum_{j= -1}^J\Big(2^{j(\frac{p}{2}-1)}\sum_{\substack{m_j\in \M_j\\m_j\neq\emptyset}}e^{-x_{m_j}}\Big)^{\frac{2}{p}}.$$
Thus, the proof is similar to that of  Theorem~\ref{rateslp}, but with the sum in $j$ in a different position, i.e. terms of type $\sum_j z_j$ are replaced by $(\sum_j z_j^{2/p})^{p/2}$. 

$\bullet$ In the Homogeneous case, recall that a model $m=\cup_{j=-1}^{J} m_j$ belongs to $\M$ if for some $L\leq J$
$$
\forall  j\leq L,\:  m_j=\{j\}\times K_j, \qquad
\forall  j>L,\:  m_j=\emptyset.
$$ Then we can prove that 
$ \tilde{V}_p(m)\lesssim \e^2 2^L$, $\tilde{B}_p(m)\lesssim R^2 2^{-Ls}$ for $\theta \in \mathcal{B}_{r,\infty}^{s}(R)$ and $\tilde{P}_p(m)\leq \e^2 2^L$. Moreover $\tilde R(\M)\lesssim (\log N)^{1-\frac{2}{p}}\e^2$. This allows us to conclude.

$\bullet$  We now deal with the Intermediate case ($\frac{p}{2s+1}<r<p$). For this case, we slightly modify the previous collection of models: we still have $\M=\bigcup_{L= 0} ^{J}\M(L)$ and 
$m$ belongs to $\M(L)$ if 
$$m_j=\begin{cases}
 \Lambda_j & \text{ if } -1\leq j\leq {L-1}\\
m_{L+l} \subset \Lambda_{L+l}& \text{ if } l=j-L\geq 0 \text{ with }|m_{L+l}|=\lfloor 2^{L+l} \tilde A(l)\rfloor
\end{cases}$$
but this time {$\tilde A(l):=2^{-lp/2}(l+1)^{-3p/2}$}.
We follow Section~\ref{proofupperintermediate} step by step, keeping in mind that $p>2$ and $b=\frac{p}{2}-1$. We have $\tilde R(\M)=(\log N)^{1-\frac2p}\e^2(T_1+T_2)$ with 
$T_1=\sum_{j=-1}^J\left[2^{jb}|m_j|^{\left(-b\right)_+}e^{-K|m_j|}\right]^{\frac2p}<\infty $ and 
%\vinc{remettre la def de $T_1$}
\begin{align*}
T_2&\leq\sum_{L= 0}^{J}\sum_{j=L}^J\left[2^{jb}|\M_j(L)|e^{-K|m_j|\big(1+\log\big(2^{(j-L)p/2}(j-L+1)^{3p/2}\big)\big)}\right]^{\frac2p}\\
&\leq\sum_{L= 0}^{J}\sum_{l=0}^{J-L}\left[2^{(L+l)b}|\M_{L+l}(L)|e^{-K 2^L2^{-lb}(l+1)^{-3p/2}\big(1+\log\big(2^{lp/2}(l+1)^{3p/2}\big)\big)}\right]^{\frac2p},
\end{align*}
with
 $
\log |\M_j(L)|\leq  C 2^{L}2^{-lb}(l+1)^{-p}% =C A(l,L)(l+1)
$
for $j\geq L$. 
Then, for $K$ large enough 
\begin{align*}
T_2&\leq \sum_{L= 0}^{J}\sum_{l=0}^{l_{\max}}\left[2^{(L+l)b}\exp\Big((C-Kp\log(2)/2)2^{L}
2^{-lb}(l+1)^{-p}\Big)\right]^{\frac2p}\\
&\leq \sum_{L= 0}^{J}\sum_{l=0}^{l_{\max}}2^{2(L+l)b/p}\exp\Big(\frac2p(C-Kp\log(2)/2)2^{L}
2^{- l_{\max}b}( l_{\max}+1)^{-p}\Big)\\
&\leq \sum_{L= 0}^{J}\sum_{l=0}^{l_{\max}}2^{2(L+l)b/p}\exp\Big(\frac2p(C-Kp\log(2)/2)( l_{\max}+1)^{p/2}\Big)\\
& \lesssim \sum_{L= 0}^{J}2^{2Lb/p}\exp(-\tilde K L) <\infty
\end{align*}
where we have used that
$2^{L}2^{-l_{\max}b}(l_{\max}+1)^{-3p/2}\approx 1$
and $l_{\max}\approx L/b$. Thus $\tilde R(\M)\lesssim (\log N)^{1-\frac{2}{p}}\e^2$.
We now deal with $\tilde V_p(m)$: for any $m\in \M(L)$,
\begin{align*}
\tilde V_p(m)&=\sum_{j= -1}^J\Big(\e^p\sigma_p^p2^{j(\frac{p}{2}-1)}  |m_j|\Big)^{\frac{2}{p}}\\
&\leq\e^2\sigma_p^2\left(\sum_{j< L}\Big(2^{j(p/2-1)}2^j\Big)^{\frac{2}{p}}
+\sum_{l\geq 0}\Big(2^{(L+l)(\frac{p}{2}-1)}2^{L}2^{-l(p/2-1)}(l+1)^{-3p/2}\Big)^{\frac{2}{p}}\right)\\
&\lesssim \e^22^{L}.
\end{align*}
For the study of $\tilde B_p(m)$, we follow the proof of Lemma~\ref{biaisBesov} \textit{mutatis mutandis}:
\begin{align*}
\inf_{m\in \M(L)}\tilde B_p(m)&=\sum_{l\geq 0}\Big(2^{(L+l)(\frac{p}{2}-1)}\sum_{k>\lfloor  2^{L+l}\tilde A(l)\rfloor }|\theta_{L+l,(k)}|^p\Big)^{\frac{2}{p}}\\
&\leq \sum_{l\geq 0}\Big(2^{(L+l)(\frac{p}{2}-1)}(\sum_k |\theta_{L+l,k}| ^r)^{\frac{p}r} (\lfloor  2^{L+l}\tilde A(l)\rfloor +1)^{1-\frac{p}r}\Big)^{\frac{2}{p}}\\
&\leq \sum_{l\geq 0}\Big(2^{(L+l)(\frac{p}{2}-1)}(\sum_k |\theta_{L+l,k}| ^r)^{\frac{p}{r}} (2^{L}2^{-l(\frac{p}2-1)}(l+1)^{-3p/2})^{1-\frac{p}r}\Big)^{\frac{2}{p}}\\
&\leq \Big(2^{L(\frac{p}{2}-\frac{p}{r})}\Big)^{\frac{2}{p}}\sum_{l\geq 0}\Big(
2^{l(\frac{p}2-1)\frac{p}r}(l+1)^{\frac{3p}{2}(\frac{p}{r}-1)} (\sum_k |\theta_{L+l,k}| ^r)^{p/r}\Big)^{\frac{2}{p}}. 
\end{align*}
Since $\theta\in \mathcal{B}_{r,\infty}^{s}(R)$,  for all $l$:
$
\left(\sum_{k}|\theta_{L+l,k}|^r \right)^{p/r}\leq R^p  2^{-(s p+\frac{p}2-\frac{p}r)(L+l)}.
$
Finally
\begin{align*}
\inf_{m\in \M(L)}\tilde B_p(m)&\leq R^22^{L(1-\frac{2}{r})}\sum_{l\geq 0}2^{l(\frac{p}2-1)\frac{2}r}(l+1)^{3(\frac{p}r-1)}2^{-(2s+1-\frac{2}{r})(L+l)}
\\
 &\leq R^2 2^{-2s L}\sum_{l\geq 0}
2^{2l(\frac{p}{2r}-\frac12-s )}(l+1)^{3(\frac{p}r-1)} 
\lesssim R^22^{-2s L}
\end{align*}
since the series converges because $\frac{p}{2r}-\frac{1}2-s <0\Leftrightarrow \frac{p}{2s+1}<r$.\\
It remains to control $\tilde P_p(m)=\sum_{j= -1}^J\Big(\e^p2^{jb}P_j(m_j)\Big)^{{2}/{p}}$ with $P_j$ given through \eqref{pen2}.
We have
$$P_j(m_j)\leq(2q\log N)^{\frac{p}{2}-1}\big(\frac32\sigma_2^2|m_j|+\kappa_2x_{m_j}\big)$$
so that, for any $m\in \M(L)$,
\begin{align*}
\tilde P_p(m)
%&=\sum_{j\geq -1}\Big(\e^p2^{j(\frac{p}{2}-1)}p(m_j)\Big)^{\frac{2}{p}}\\
&\lesssim  \e^2(\log N)^{1-\frac{2}{p}}\Bigg[\sum_{j<L}\Big(2^{jb}|m_j|\Big)^{\frac{2}{p}}
%\\&&
+\sum_{j\geq L}\Big(2^{jb}2^{(1+b)L}2^{-jb}(j-L+1)^{-\frac{3p}2}\big((j-L+1)+\log(j-L+1)\big)\Big)^{\frac{2}{p}}\Bigg]\\
&\lesssim \e^2(\log N)^{1-\frac{2}{p}}\Big[2^{L}+2^{L}\sum_{l\geq 0}(l+1)^{-3}\big(l+1+\log(l+1)\big)^{\frac{2}{p}}\Big]
\\&\lesssim \e^2(\log N)^{1-\frac{2}{p}}2^{L}.
\end{align*}
We conclude
\begin{align*}
\inf_{m\in\M}A(m)
%&=\inf_{m\in\M}\left\{\sum_{j\geq -1}\Big(2^{j(\frac{p}{2}-1)}\sum_{k\in K_j}\E\Big[\big|\hat \theta_k^{(m_j)}-\theta_{jk}\big|^p\Big]+\e^p2^{j(\frac{p}{2}-1)}P_j(m_j)\Big)^{\frac{2}{p}}\right\}\\
&\leq\inf_{L}\left\{\inf_{m\in \M(L)}\tilde B_p(m)+\sup_{m\in \M(L)}\tilde V_p(m)+\sup_{m\in \M(L)}\tilde P_p(m)\right\}\\
&\lesssim\inf_{L}\left\{R^22^{-2s L}+\e^2(\log N)^{1-\frac{2}{p}}2^{L}\right\}\\
&\lesssim R^{\frac{2}{1+2s}}\e^{\frac{4s}{1+2s}}|\log \e|^{\frac{2s(p-2)}{p(1+2s)}},
\end{align*}
by taking $L$ such that 
$$2^L\approx (\e^{-1} R)^{\frac{2}{1+2s}}|\log \e|^{-\frac{(p-2)}{p(1+2s)}}.$$
$\bullet$ In the Sparse case,
Remark~\ref{RemarkResteSparse} gives $\tilde R(\M)\lesssim (\log N)^{1-\frac{2}{p}}\e^2$ and Remark~\ref{RemarkBiaisVarSparse} provides $$\sup_{\theta\in \mathcal{B}_{r,\infty}^s(R)} \left(\tilde{B}_p(m)+\tilde{V}_p(m)\right) \lesssim  R^{2(1-2\beta)}|\log \varepsilon|^{2\beta}\: \e^{4\beta}.$$
Finally, following the outlines of the end of Section~\ref{proofuppersparse}, we can prove 
$$\tilde P_p(m)\lesssim \left(R^{p(1-2\beta)}\e^{2\beta p}|\log\e|^{p\beta}\right)^{2/p}.$$
In the frontier case, $r=p/(2s+1)$,  we obtain
$$\sup_{\theta\in \mathcal{B}_{r,\infty}^s(R)} \left(\tilde{B}_p(m)+\tilde{V}_p(m)\right)+\tilde P_p(m) \lesssim  R^{2(1-2\beta)}|\log \varepsilon|^{2\beta+1}\: \e^{4\beta}.$$
This ends the proof.

\section*{Acknowledgements}

We are grateful to Rados\l aw Adamczak for his advices about concentration inequalities for sub-Weibull variables. 

%%%%%%%%%%%%%%%%%%%%
\bibliographystyle{alpha}
\bibliography{bibli}

%%%%%%%%%%%%%%%%%%

\end{document}